\documentclass[smallextended,final,numbook]{svjour3}     

\usepackage[colorlinks=true, linkcolor=blue, citecolor=blue, urlcolor=blue, pdfborder={0 0 0}]{hyperref}
\usepackage[letterpaper,top=2in, bottom=1.5in, left=1in, right=1in]{geometry}
\usepackage{graphicx}

\graphicspath{ {./figures/} }
\usepackage{amssymb}
\usepackage{amsmath}
\usepackage{enumitem}
\usepackage{tikz}

\usepackage{booktabs}
\usepackage{multirow} 
\usepackage{arydshln} 
\usepackage[%
    font={small,sf},
    labelfont=bf,
    format=hang,
    format=plain,
    margin=0pt,
    width=0.8\textwidth,
]{caption}
\usepackage[list=true]{subcaption}
\usepackage{xcolor,colortbl}

\usepackage{mathtools}
\usepackage[normalem]{ulem} 

\newtheorem{algo}{Algorithm}

\newcommand{\remove}[1]{{}}

\newcommand{\sfS}{\mathsf{S}}

\newcommand{\umu}{\overline{M}}
\newcommand{\lmu}{\underline{M}}

\newcommand{\vN}{{\mathbf{N}}}

\newcommand{\vR}{{\mathbf{R}}}

\newcommand{\cB}{{\mathcal{B}}}

\newcommand{\cE}{{\mathcal{E}}}
\newcommand{\cF}{{\mathcal{F}}}
\newcommand{\cG}{{\mathcal{G}}}
\newcommand{\cH}{{\mathcal{H}}}
\newcommand{\cI}{{\mathcal{I}}}

\newcommand{\cS}{{\mathcal{S}}}

\newcommand{\cX}{{\mathcal{X}}}

\newcommand{\RR}{\mathbb{R}}
\newcommand{\NN}{\mathbb{N}}

\newcommand{\ZZ}{\mathbb{Z}}

\newcommand{\EE}{\mathbb{E}}

\newcommand{\mkQ}{\mathfrak{Q}}
\newcommand{\pr}{\mathrm{prod}}

\newcommand{\prox}{\mathbf{prox}}

\newcommand{\minimize}{\text{minimize}}

\DeclareMathOperator*{\argmin}{arg\,min}

\DeclareMathOperator*{\Min}{minimize}

\DeclareMathOperator*{\Fix}{Fix}
\DeclareMathOperator*{\zer}{zer}

\DeclareMathOperator*{\as}{a.s.} 

\newcommand{\bc}{\begin{center}}
\newcommand{\ec}{\end{center}}

\newcommand{\bdm}{\begin{displaymath}}
\newcommand{\edm}{\end{displaymath}}

\newcommand{\beq}{\begin{equation}}
\newcommand{\eeq}{\end{equation}}

\newcommand{\bfl}{\begin{flushleft}}
\newcommand{\efl}{\end{flushleft}}

\newcommand{\bt}{\begin{tabbing}}
\newcommand{\et}{\end{tabbing}}

\newcommand{\beqn}{\begin{align}}
\newcommand{\eeqn}{\end{align}}

\newcommand{\beqs}{\begin{align*}} 
\newcommand{\eeqs}{\end{align*}}  

\newtheorem{assumption}{Assumption}

\usepackage[lined,boxed,commentsnumbered, ruled,vlined]{algorithm2e}

\newcommand\numberthis{\addtocounter{equation}{1}\tag{\theequation}}

\DeclarePairedDelimiter{\dotp}{\langle}{\rangle}

\usepackage{booktabs}

\usepackage{mdframed}
\usepackage{etoolbox}
\AtBeginEnvironment{algo}{\begin{minipage}{\textwidth}}
\AtEndEnvironment{algo}{\end{minipage}}

\usepackage{empheq}
\definecolor{myblue}{rgb}{.8, .8, 1}
\newcommand*\mybluebox[1]{%
\colorbox{myblue}{\hspace{1em}#1\hspace{1em}}}

\def\cut#1{{}}
\smartqed

\title{SMART: The Stochastic Monotone Aggregated Root-Finding Algorithm \thanks{This material is based upon work supported by the National Science Foundation under Award No. 1502405.
}}

\author{Damek Davis}

\institute{D. Davis\at
              School of Operations Research and Information Engineering, Cornell University\\
              Ithaca, NY 16850, USA\\
              \email{dsd95@cornell.edu}}

\date{\today}
\journalname{Report} 

\begin{document}

\maketitle
\abstract{
We introduce the Stochastic Monotone Aggregated Root-Finding (SMART) algorithm, a new randomized operator-splitting scheme for finding roots of finite sums of operators. These algorithms are similar to the growing class of incremental aggregated gradient algorithms, which minimize finite sums of functions; the difference is that we replace gradients of functions with black-boxes called operators, which represent subproblems to be solved during the algorithm. By replacing gradients with operators, we increase our modeling power, and we simplify the application and analysis of the resulting algorithms. The operator point of view also makes it easy to extend our algorithms to allow arbitrary sampling and updating of blocks of coordinates throughout the algorithm. Implementing and running an algorithm like this on a computing cluster can be slow if we force all computing nodes to be synched up at all times. To take better advantage of parallelism, we allow computing nodes to delay updates and break synchronization. 

This paper has several technical and practical contributions. We prove the weak, almost sure convergence of a new class of randomized operator-splitting schemes in separable Hilbert spaces; we prove that this class of algorithms convergences linearly in expectation when a weak regularity property holds;  we highlight connections to other algorithms; and we introduce a few new algorithms for large-scale optimization.}

\section{Introduction}

The guiding problems in optimization are evolving. While all optimization problems can be reduced to minimizing functions over sets, prototypical optimization problems, like $\min_{x \in C} \,f(x)$, hide structure that is found throughout modern applications, and this structure is useful for designing algorithms for large-scale problems (e.g., problems with gigabytes to terabytes of data). Among these structures, the finite sum is the most pervasive:
\begin{align}\label{eq:finsumfunc}
\minimize_{x \in C} \, \frac{1}{n}\sum_{i=1}^n f_i(x).
\end{align}
Large sums ($n \gg 0$) are common in statistical and machine learning, where the functions $f_i$ often, but not always, correspond one-to-one with points in a dataset. When the $f_i$ are summed together, the minimization problem~\eqref{eq:finsumfunc} grows in complexity, but the prevailing and realistic assumption in applications is that it is drastically simpler (in terms memory or computational complexity) to perform operations, like differentiation, on the $f_i$ than it is to perform the same operations on the entire sum. 

Large sums, like~\eqref{eq:finsumfunc}, come from large datasets, even low dimensional ones. But modern applications often involve high dimensional datasets. When the decision variable $x \in \RR^m$ is high-dimensional ($m \gg 0$), the prevailing and realistic assumption in applications is that it is drastically simpler to compute partial derivatives of the $f_i$, or other componentwise quantities, than it is to compute full derivatives of the $f_i$. 

These two structural assumptions, and the host of algorithms that adopt them, have led to big improvements in large-scale algorithm design (for examples, see ~\cite{doi:10.1137/100802001,approx,schmidt2013minimizing,strohmer2009randomized,shalev2014accelerated,defazio2014finito,johnson2013accelerating,defazio2014saga,liu2015asynchronous,peng2015arock,doi:10.1137/140971233,bertsekas2015incremental,combettes2015asynchronous,wang2013incremental,bertsekasincrementalproximal,doi:10.1137/S1052623499362111,bianchi2015ergodic,recht2011hogwild}). And we do not have to look far to find problems for which these assumptions make sense. A simple problem with $n \gg 0$ and $m \gg 0$ 
comes from regularized least squares problems with matrix $A = \begin{bmatrix} A_1^T, \ldots, A_n^T\end{bmatrix}^T\in \RR^{n\times m}$, vector $b \in \RR^n$, and regularizer induced by $K \in \RR^{m \times m}$:
\begin{align}\label{eq:linearequations}
\Min_{x\in \RR^m} \, \frac{1}{n}\sum_{i=1}^n \left(\frac{1}{2}\left\|A_i\begin{bmatrix}x_1 \\ \vdots \\ x_m \end{bmatrix} - b_i\right\|^2 + \frac{1}{2}\dotp{Kx, x}\right).
\end{align}
The two prevailing assumptions are clearly satisfied for this example when each row of $A_i$ is sparse: each term in the sum is simple and partial derivatives of the terms are easier to compute than full derivatives. Absent other special structure in the matrices $A$ and $K$ (like bandlimitedness), a conceptually simple operation like differentiating the objective function requires $O(\max\{mn, m^2\})$ operations. With problems of this scale, full gradient computations cannot lie at the heart of a practical iterative algorithm for solving~\eqref{eq:linearequations}. 

The algorithms that have sprung from these structural assumptions all look alike. They all form a sequence of points $\{x^k\}_{k \in \NN}$ that converge (in some appropriate sense of the word) to a minimizer of~\eqref{eq:finsumfunc}. They all make a series of choices at each time step: given $x^k$
\begin{enumerate}
\item choose (randomly or otherwise) indices $i \in \{1, \ldots, n\}$ and $j \in \{1, \ldots, m\}$;
\item get $x^{k+1}$ (in some undetermined way) from $x^{k}$ and $\nabla_j f_i(x^k)$ by only changing a single component of $x^k$ (possibly using some combination of the previous terms $\{x^l\}_{l \in \{0, \ldots, k\}}$).
\end{enumerate}
At this moment, the differences between these algorithms is immaterial. What matters is the contrast between these algorithms and the overwhelmingly costlier algorithms that form sequences $\{x^k\}_{k \in \NN}$ by getting $x^{k+1}$ from $x^k$ and $\nabla \left[ \sum_{i=1}^n f_i\right](x^k)$; we simply cannot afford to compute these full gradients for every $k \in \NN$. 

When looking at this 2-step prototypical algorithm, there is a looming temptation to generalize and to replace the partial derivatives of the $f_i$ with something else. To find a good generalization, we only need to look at the optimality conditions of~\eqref{eq:finsumfunc}:
\begin{align}\label{eq:gradientroot}
\text{Find $x^\ast$ such that:} \qquad \frac{1}{n}\sum_{i=1}^n \nabla f_i(x^\ast) = 0.
\end{align}
At once we reduce our cognitive load by changing notation and forgeting that there ever was a function called $f_i$. We simply replace every occurrence of $\nabla f_i$ with a mapping called $S_i$:\vspace{-1.5pt}
\begin{empheq}[box=\mybluebox]{align}\label{eq:Sroot}
\text{Find $x^\ast$ such that:} \qquad S(x^\ast) := \frac{1}{n}\sum_{i=1}^n S_i(x^\ast) = 0.
\end{empheq}
The $S_i$ take, as input, points in a space $\cH$, like $\RR^n$, and output points in $\cH$. We follow mathematical tradition and call the $S_i$ \emph{operators}. We do not, however, forget about our structural assumptions: for any given $x \in \cH$, we still expect $S_i(x)$ to be drastically simpler to compute than $n^{-1}\sum_{i=1}^n S_i(x)$, and we still expect the $j$th component, denoted by $(S_i(x))_j$, to be drastically simpler to compute than the full operator $S_i(x)$. We gain a lot of flexibility from this generalization. 

Each $S_i(x)$ can, in principle, be any mapping. But we would never be so optimistic. Instead, a coherence condition suggests itself: we assume that for each root $x^\ast$ of~\eqref{eq:Sroot}, we have
\begin{empheq}[box=\mybluebox]{align}\label{eq:coherence}\vspace{-5pt}
(\exists \beta_{ij} > 0) : (\forall x\in \cH)  \qquad \sum_{j=1}^m \sum_{i=1}^n \beta_{ij} \|(S_{i}(x))_{j} - (S_{i}(x^\ast))_j\|_j^2  \leq \dotp{S(x), x - x^\ast}.
\end{empheq}
This condition\footnote{The technical name might be quasi-cocoercivity; we could not find it discussed elsewhere.} is weaker than what can be guaranteed when each $S_i(\cdot)$ is the gradient of a convex function. In that case, the $x^\ast$ in~\eqref{eq:coherence} can even vary beyond the roots of $n^{-1}\sum_{i=1}^n S_i(\cdot)$ to any point in $\cH$; the $\beta_{ij}$ are akin to inverse Lipschitz constants of gradients. But~\eqref{eq:Sroot} is not limited to smooth optimization problems because~\eqref{eq:coherence} will be satisfied for $S_i$ that are, for example, compositions of proximal and gradient operators. Beyond the benefits of increased modeling power, though, we make this generalization because it is easy to design algorithms for~\eqref{eq:Sroot} that hide all of the complicated details that appear in algorithms which solve specific models, for example, it can be quite technical to design algorithms for~\eqref{eq:finsumfunc} when two of the terms, say $f_1$ and $f_2$, are nonsmooth. Instead, we treat $S_i$ just like we treat $\nabla f_i$. 

A need for increased modeling power and algorithms with smaller per-iteration complexity has led us to Problem~\eqref{eq:Sroot} and a class of algorithms that solve it. We want to parallelize these algorithms. But they are inherently sequential because each step updates just a single coordinate using just a single function or operator, and the other coordinates cannot be updated until the active one finishes its work. A big slowdown occurs if some partial derivatives or operator coordinates are much more complicated to evaluate than all the others because we will spend most of our time updating just a few coordinates and leaving the others fixed. A solution is to eliminate the stalling between coordinate updates and allow multiple processors to work at their own pace, updating whenever they complete their work.  We take this approach in this paper and allow \emph{asynchronous} updates in our algorithm.

In total, we have taken the basic template for incremental gradient and block-coordinate optimization algorithms and increased its modeling power by introducing the root finding Problem~\eqref{eq:Sroot} and an algorithm to solve it. This stochastic monotone aggregated root-finding (SMART)\footnote{We use the term \emph{monotone} because $S$ is a quasi-monotone operator.} algorithm is inspired by another algorithm called SAGA~\cite{defazio2014saga}. In this work we have taken SAGA and generalized it: the SMART algorithm applies to operators, allows block-coordinate updates, allows asynchronous updates, requires less memory than SAGA, and requires less computation per iteration than SAGA. The theoretical guarantees for SMART are also stronger: the sequence of points formed by the algorithm, and not just the function values, will converge with probability 1 when a solution exists (even in infinite dimensional separable Hilbert spaces, but in this case SMART converges weakly). Like SAGA, SMART converges linearly when~$n^{-1}\sum_{i=1}^n f_i$ is strongly convex, and beyond that, it converges linearly when $n^{-1} \sum_{i=1}^n S_i$ is essentially quasi-strongly monotone. It even converges linearly in the asynchronous, block-coordinate case. The rest of this paper describes SMART and proves that it converges. 

\section{Assumptions and Notation}\label{section:assump}

The SMART algorithm solves~\eqref{eq:Sroot} in a separable Hilbert space, like $\RR^n$, which we call $\cH$. We assume the Hilbert space $\cH = \cH_1 \oplus \cdots \oplus \cH_m$ is a direct sum of $m \in \NN$ other Hilbert spaces $\cH_1, \ldots, \cH_m$. Given a vector $x \in \cH$, we denote its $j$th component by $x_j \in \cH_j$. Given a sequence $\{x^k\}_{k\in \NN}$ and a vector $h \in \NN^m$, we define 
\begin{align*}
(\forall k \in \NN) \qquad  x^{k-h} =(x^{k-h_1}_1, \ldots, x^{k-h_m}_m)
\end{align*}
and use the convention that $x_j^k = x_j^0$ if $k \leq 0$. For $j \in \{1, \ldots, m\}$, we let $\dotp{ \cdot, \cdot}_j : \cH_j \times \cH_j  \rightarrow \RR$ denote the inner product on $\cH_j$, and we let $\|\cdot\|_j$ be the corresponding norm. For all $x, y \in \cH$, we let $\dotp{x, y}_\pr = \sum_{j=1}^m\dotp{x_j, y_j}_j$  and $\|x\|_{\pr}:= \sqrt{\smash[b]{\dotp{x, x}}}_{\pr}$ be the standard inner product and norm on $\cH$. We also fix an inner product $\dotp{\cdot, \cdot} : \cH \times \cH \rightarrow \RR$ and denote the corresponding norm by $\|\cdot\| $. We make one assumption about this norm:
\begin{align*}
\left(\exists  \lmu_j, \umu_j > 0\right) :\left( \forall x\in \cH\right) \qquad \sum_{i=1}^m\lmu_j \|x_j\|_j^2 \leq \|x\|^2 \leq \sum_{i=1}^m \umu_j\|x_j\|_j^2 .
\end{align*}
We often choose inner products associated to self-adjoint linear maps, $P$, which are defined for all $x, y \in \cH$, by $\dotp{x, y} = \dotp{x, y}_P := \dotp{Px, y}_\pr = \dotp{x, Py}_\pr$. We work with an underlying probability space denoted by $(\Omega, \cF, P)$, and we assume that the space $\cH$ is equipped with the Borel $\sigma$-algebra. We always let $\sigma(X) \subseteq \cF$ denote the sub $\sigma$-algebra generated by a random variable $X$. We use the shorthand $\as$ to denote almost sure convergence of a sequence of random variables. We also assume that the operators $S_i$ are measurable. We say that a map $T : \cH \rightarrow \cH$ is \emph{nonexpansive} if it is 1-Lipschitz continuous, i.e., $\|Tx - Ty\|\leq \|x-y\|$ for all $x, y \in \cH$. A map $S : \cH \rightarrow \cH$ is called \emph{demiclosed at $0$} if whenever a sequence $\{x^k\}_{k \in \NN}$ converges weakly to a point $x \in \cH$ and $S(x^k)$ converges strongly to $0\in \cH$, then $S(x) = 0$. Given any nonempty, closed, convex set $C$, let $P_C(x) := \argmin_{x^\ast \in C} \|x - x^\ast\|$ denote its projection operator, and let $d_C(x) := \min_{x^\ast\in C} \|x - x^\ast\|$ denote its distance function. For any closed convex set, $C \subseteq \cH$, we let $N_C$ denote the normal cone of $C$~\cite[Definition 6.37]{bauschke2011convex}. We define
\vspace{-7pt}\begin{align*}
S:= \frac{1}{n} \sum_{i=1}^n S_i \qquad \text{and} \qquad \cS := \zer(S) = \{ x \in \cH \mid S(x) = 0\}.
\end{align*}\vspace{-10pt}

Beyond~\eqref{eq:coherence}, we use a single regularity property
\begin{align}
&(\exists \mu > 0): (\forall x\in \cH) &\dotp{S(x), x - P_\cS(x)} &\geq \mu\|x- P_\cS(x)\|^2.\numberthis\label{eq:essstrongquasimono}
\end{align} 
Operators that satisfy~\eqref{eq:essstrongquasimono} are called \emph{essentially strongly quasi-monotone}. In order for our most general results to hold, $S$ need not satisfy~\eqref{eq:essstrongquasimono}, but when it is satisfied, our algorithm converges linearly. Although it is hardly comprehensive, one example is noteworthy: property~\eqref{eq:essstrongquasimono} holds if $S = A^\ast \circ \nabla f \circ A$ for a strongly convex function $f$ and a matrix $A$~\cite[p. 287]{liu2015asynchronous}. See~\cite{bolte2015error,zhang2015restricted,lai2013augmented} for information on convex error bounds.

Most of the concepts that we use in this paper can be found in~\cite{bauschke2011convex}. See Table~\ref{tab:symbols} for a list of symbols.
\vspace{-10pt}\section{The SMART Algorithm}

We develop an iterative algorithm that solves~\eqref{eq:Sroot}. The algorithm forms a sequence of \emph{primal variables} $\{x^k\}_{k \in \NN} \subseteq \cH$ that converges to a root of~\eqref{eq:Sroot}. The algorithm also maintains a sequence of \emph{dual variables}, one per operator, which is denoted by $\{(y_1^k, \ldots, y_n^k)\}_{k \in \NN}\subseteq \cH^n$. 

We assume that at least one of $m$, the number of components, and $n$, the number of operators, is large, and so it costs a lot to obtain $x^{k+1}$ from $x^k$ by updating all $m$ of its components, using all $n$ of the operators. To lower the cost, we introduce an IID sequence of $2^{\{1, \ldots, m\}}$-valued (subsets of $\{1, \ldots, m\}$) random variables $\{\sfS_k\}_{k \in \NN}$ that determines which components of $x^{k}$ we update at the $k$th iteration. The component-choosing variable $\mathsf{S}_k$ is coupled with an IID sequence of $\{1, \ldots, n\}$-valued random variables $\{i_k\}_{k \in \NN}$ that determine which one of the $n$ operators $S_1, \ldots, S_n$ are evaluated at the $k$th iteration. If $i_k = i$ and $j \in \sfS_k$, then the $j$th component of $x^k$ is updated at iteration $k$ using an evaluation of $(S_i(\cdot))_j$; the other operators and components are left alone. The user can choose $\{i_k\}_{k \in \NN}$ and $\{\sfS_k\}_{k \in \NN}$ however they like as long as 
\begin{quote}
\centering$q_j := P(j \in \sfS_0) > 0$ for all $j$; and $p_{ij} := P(i_0 = i | j \in \sfS_0) > 0$ exactly when $(S_i(\cdot))_j \not\equiv 0$.
\end{quote}

Unlike the point sequence $\{x^k\}_{k \in \NN}$, the dual variables need not be updated at every iteration. Instead, we introduce an IID sequence of $\{0, 1\}$-valued random variables $\{\epsilon_k\}_{k \in \NN}$ that determines if and when we update the dual variables. If $\epsilon_k = 1$, then the dual variables are updated at iteration $k$; otherwise, the dual variables are left alone. The user can choose $\{\epsilon_k\}_{k \in \NN}$ however they like as long as 
\begin{quote}
\centering $\rho := P(\epsilon_0 = 1) > 0$. 
\end{quote}

If $\epsilon_k = 1$, and thus the dual variables must update at iteration $k$, we only require, at the absolute minimum, that $y_{i_k}^k$ be updated; the rest of the variables may stay fixed. However, we allow the update of $y_{i_k}^k$ to \emph{trigger} the update of any subset of the other dual variables. The \emph{trigger graph} $G = (V, E)$ with vertices $V = \{1, \ldots, n\}$ uses the \emph{edge set} $E \subseteq \{1, \ldots, n\}^2$ to encode, for each $i \in V$, the set of dual variables that must be updated when $i_k = i$: 
\begin{quote}
\centering $(i,i') \in E$ if, and only if, $i_k = i$ triggers the update of dual variable $y_{i'}^k$.
\end{quote}
When $(i,i') \in E$, we simply say that $i$ \emph{triggers} $i'$. We require that for all $i$, $(i,i) \in E$, but otherwise there are no constraints on $G$; it can be absolutely any graph, from a completely disconnected graph, to a complete graph on $n$ vertices. And one quantity, figuring only in our linear rate of convergence, is important: the probability that $i$ is triggered \emph{and} coordinate $j$ is sampled simultaneously
\begin{quote}
\centering$p_{ij}^T := P((i_0, i) \in E, j \in \sfS_0) = P((i_0, i) \in E \mid j \in \sfS_0)q_j = \sum_{(i', i) \in E} p_{i'j}q_j$,
\end{quote}
which is easily computable, but it need not be known to guarantee convergence. 

Often, the matrix of optimal operator values 
\begin{quote}
\centering$\mathbf{S}^\ast := ((S_i(x^\ast))_j)_{ij}, \qquad x^\ast \in \cS$
\end{quote} has some entries which are zero (by~\eqref{eq:coherence},  $\mathbf{S}^\ast$ is independent of $x^\ast \in \cS$). For these zero entries, we simply set the corresponding dual variable to zero: $y_{i,j}^k \equiv 0$ if $\mathbf{S}^\ast_{ij} = 0$. In the extreme case that $\mathbf{S}^\ast = 0$, all operators are zero at the solution,~\eqref{eq:Sroot} reduces to the common zero problem, and the dual variables $\{(y^k_1, \ldots, y^k_n)\}_{k \in \NN}\subseteq \cH^n$ are all zero. Setting these particular dual variables to zero is not necessary, but by doing so, we use less memory than we otherwise would.

An algorithm that solves~\eqref{eq:Sroot} might use the following 3-step process: given $x^k$ 
\begin{enumerate}
\item\label{item:algNoAsyncSample} sample $\sfS_k,i_k, $ and $\epsilon_k$;
\item\label{item:algNoAsyncx}  get $x^{k+1}$ from $x^k$ using $\{(S_{i_k}(x^k))_{j}\mid j \in \sfS_k\}$ and $(y_{1}^k, \ldots, y_n^k)$;
\item\label{item:algNoAsyncphi} if $\epsilon_k = 1$, get $(y_{1}^{k+1}, \ldots, y_{n}^{k+1})$ using $x^{k}$ and $(y_{1}^k, \ldots, y_{n}^k)$; otherwise set $(y_{1}^{k+1}, \ldots, y_{n}^{k+1}) = (y_{1}^k, \ldots, y_{n}^k)$.
\end{enumerate}
After $i_k$ and $\sfS_k$ are sampled, the inactive operators and the inactive components stall until the active ones finish Steps~\ref{item:algNoAsyncx} and~\ref{item:algNoAsyncphi}. If we have access to a parallel computing device, stalling is wasteful, so in our algorithm we let all operators work in parallel and update $x^k$ whenever they finish their work. Mathematically, we form sequences of \emph{delays} 
$\{d_k\}_{k \in \NN} \subseteq \{0, 1, \ldots, \tau_p\}^{m}$ and $\{e^{i}_{k}\}_{k \in \NN} \subseteq \{0, 1, \ldots, \tau_d\}^m$, and we replace several instances of $x^{k}$ and $y_i^k$ with iterates from the past  whose coordinates were formed in the last $\tau_p$ or $\tau_d$ iterations, respectively.
The final algorithm is below:
\begin{mdframed}
\noindent\begin{algo}[SMART]\label{alg:main}
Let $\{\lambda_k\}_{k \in \NN}$ be a sequence of stepsizes. Choose $x^0\in \cH$ and $y_1^0, \ldots, y_{n}^0 \in \cH$ arbitrarily except that $y_{i,j}^0 = 0$ if $\mathbf{S}^\ast_{ij} = 0$. Then for $k \in \NN$, perform the following three steps: 
\begin{enumerate}
\item \textbf{Sampling:} choose a set of coordinates $\sfS_k$,  an operator index $i_k$, and dual update decision $\epsilon_k$.
\item \textbf{Primal update:} set
\begin{align*}
\left(\forall j \in \sfS_k\right) \qquad &x_j^{k+1} = x_j^k - \frac{\lambda_k}{q_jm}\left(\frac{1}{np_{ij}}(S_{i_k}(x^{k-d_k}))_j - \frac{1}{np_{ij}}y_{i_k, j}^{k-e^{i_k}_{k}} + \frac{1}{n}\sum_{i=1}^{n} y_{i,j}^{k - e^{i}_{k}}\right); \\
\left(\forall j \not\in \sfS_k\right) \qquad & x_j^{k+1} = x_j^k.
\end{align*}
\item \textbf{Dual update:} If $i_k$ triggers $i$, set
\begin{align*}
 \left(\forall j \in \sfS_k \text{ with } \mathbf{S}^\ast_{ij} \neq 0\right)  \qquad &y_{i,j}^{k+1} = y_{i,j}^k + \epsilon_k\left((S_i(x^{k-d_k}))_j - y_{i,j}^k\right); \\
\left(\forall  j \notin  \sfS_k \right) \qquad & y_{i,j}^{k+1} = y_{i,j}^k.
\end{align*}
Otherwise,  set $y_{i,j}^{k+1} = y_{i,j}^k.$\qquad \qed
\end{enumerate}
\end{algo}
\end{mdframed}

The $x^{k-d_k}$ iterate in SMART can be \emph{totally synchronous}, in which case $d_{k,j} = 0$ for all $j$; 
 it can be \emph{consistent-read asynchronous}, in which case $d_{k,j} = d_{k,j'}$ for all $j$ and $j'$; 
 or it can be \emph{inconsistent-read asynchronous}, in which case $d_{k,j} \neq d_{k,j'}$ for some $j$ and $j'$. 
Totally synchronous iterates are not delayed at all, so $x^{k-d_k} = x^k$ for all $k$; consistent-read asynchronous iterates are delayed, but all of their coordinates are delayed by the same amount, so $x^{k-d_k} \in \{x^k, x^{k-1}, \ldots, x^{k-\tau_p}\}$ for all $k$; inconsistent-read asynchronous iterates are delayed, and their coordinates can be delayed by different amounts, so $x^{k-d_{k,j}}_j \in \{ x^{k}_j, x^{k-1}_j, \ldots, x^{k-\tau_p}_j\}$ for all $j$, but $x^{k-d_k}$ is not necessarily an element of $\{x^k, x^{k-1}, \ldots, x^{k-\tau_p}\}$ for all $k$.

Likewise, the $y_i^{k-e_k^i}$ iterate in SMART can be totally synchronous, consistent-read asynchronous, or inconsistent-read asynchronous, corresponding to the cases $e_{k,j}^i = 0$ for all $j$, $e_{k,j}^i = e_{k,j'}^i$ for all $j$ and $j'$, or $e_{k,j}^i \neq e_{k,j'}^i$ for some $j$ and $j'$, respectively. 

The delays $\{d_k\}_{k \in \NN}$ and $\{e^{i}_k\}_{k \in \NN}$ come with no particular order. They can induce maximal delays at all times, e.g., $d_{k,j} = \tau_p$ for all $j$ and $k$, in which case the oldest information possible is used in every iteration; they can be cyclic, e.g., $d_{k,j} = k \mod (\tau_p+1)$ for all $j$ and $k$, in which case the same information is used for $\tau_p$ consecutive iterations in a row, and all the intermediate information is thrown away; but in general, the delays can be arbitrary. These delays are artificially imposed, but we can also incur delays that are beyond our control. 

Uncontrolled delays can occur in the $x^{k}$ iterate when a processor, called $\text{Proc}_1$, attempts to read coordinates $x_1^k, \ldots, x_m^k$ while another processor, called $\text{Proc}_2$, attempts to replace $x^k$ with $x^{k+1}$. It can happen that $\text{Proc}_1$ successfully reads coordinate $x_1^k$ before $\text{Proc}_2$ replaces it with $x_1^{k+1}$, but that $\text{Proc}_2$ replaces $x_2^k, \ldots, x_m^{k}$ with $x_2^{k+1}, \ldots, x_m^{k+1}$ before $\text{Proc}_1$ attempts to read this group of coordinates. When $\text{Proc}_1$ finishes reading, it will have the iterate $(x_1^k, x_2^{k+1}, \ldots, x_m^{k+1})$, which is not necessarily equal to any previous iterate $x^{t}$ with $t \leq k$. This effect is exacerbated if multiple processors attempt to write and read simultaneously, but in SMART, $x^{k-d_k}$ is inconsistent-read asynchronous, so these uncontrolled delays cause no trouble. 

In general, including inconsistent-read asynchronous iterates leads to tedious convergence proofs with many-term recursive identities and complicated stepsizes. The recursive identities are necessary to control algorithm progress in a chaotic environment. The stepsizes, on the other hand, can be be optimized and have a clear dependence on the delays, sampling probabilities, and problem data. In both the inconsistent-read and consistent-read cases, the algorithm will converge with the same range of parameters. However, the rates of convergence depend on a measure of inconsistency, which we call $\delta$
\begin{align}\label{eq:deltadef}
\delta := \sup_{\substack{k \in \NN \\ j, j' \in \{1, \ldots, m\}}} |d_{k,j} - d_{k,j'}|.
\end{align}
When $\delta = 0$, the $x^{k-d_k}$ are consistent-read asynchronous, and the convergence rates improve. Otherwise $\delta \in \{0, \ldots, \tau_p\}$ and the convergence rates degrade with increasing $\delta.$ Of course, when $\delta$ is not known explicitly, it can be replaced by its upper bound $\tau_p$.  

~\\~\\
\indent If all of the sampling variables are statistically independent and the stepsizes are chosen small enough, SMART converges:

\begin{theorem}[Convergence of SMART]\label{thm:newalgoconvergence}
For all $k \geq 0$, let $\cI_k = \sigma((i_k, \sfS_k))$, let $\cE_k = \sigma(\varepsilon_k)$, let 
$$\cF_k :=  \sigma(\{x^l\}_{l =0}^k\cup \{y_1^l, \ldots, y_n^l\}_{l = 0}^k),$$ and suppose that $\{ \cI_k, \cE_k, \cF_k\}$ are independent. Suppose that~\eqref{eq:coherence} holds. Finally, suppose that there are constants $\underline{\lambda}, \overline{\lambda} > 0$ such that $\{\lambda_k\}_{k\in \NN} \subseteq [\underline{\lambda}, \overline{\lambda}]$, and  
\begin{align}\label{eq:weaklambdabound}
\overline{\lambda}^2 < \begin{cases}
\min_{i,j}\left\{\frac{2\underline{\lambda} n^2p_{ij} \beta_{ij}}{\frac{3\umu_j\tau_p}{m\sqrt{\underline{q}}} + \frac{2\umu_j}{\underline{q}m}}\right\} & \text{if $\mathbf{S}^\ast = 0$}; \\
\min_{i,j}\left\{\frac{\underline{\lambda}n^2 p_{ij}\beta_{ij}}{\frac{\umu_j\tau_p}{m\sqrt{\underline{q}}}\sqrt{2(\tau_d+2)} + \frac{\umu_j(\tau_d + 2)}{m\underline{q}}}\right\} & \text{otherwise}.
\end{cases}
\end{align}
Then
\begin{enumerate}
 \item \textbf{Convergence of operator values.}\label{thm:newalgoconvergence:part:operatorvalues}  For $i \in \{1, \ldots, n\}$, the sequence of $\cH$-valued random variables $\{S_i(x^{k-d_k})\}_{k \in \NN} \as$ converges strongly to $S_i(x^\ast)$.
\item \textbf{Weak convergence.} \label{thm:newalgoconvergence:part:weak} 
Suppose that $S$ is demiclosed at $0$. Then the sequence of $\cH$-valued random variables $\{x^k\}_{k \in \NN} \as$ weakly converges to an $\cS$-valued random variable.
\item \textbf{Linear convergence.} \label{thm:newalgoconvergence:part:linear} Let $\eta <  \min_{i, j} \{\rho p_{ij}^T\},$ let $\alpha \in [0, 1)$, let $\underline{q} = \min_j\{q_j\}$, let $\umu = \min_{j} \{\umu_j\}$, and let 
\begin{align}\label{eq:asynclambdalinear}
\lambda &\leq \min_{i,j}\left\{\frac{2\eta(1-\alpha)n^2p_{ij}\beta_{ij}}{\frac{2\umu_j\eta(\tau_d + 2)}{q_jm}\left(1 + \frac{\delta\eta}{\tau_d + 1} + \frac{5\sqrt{2(\tau_d + 2)}\alpha^2\delta}{m\umu \left( \sqrt{2(\tau_d + 2)} + \tau_p \sqrt{\underline{q}} \right) }\right) + \frac{\umu_j\eta\tau_p\sqrt{2(\tau_d + 2)}}{m\sqrt{\underline{q}}} \left(2+\frac{\eta}{1-\eta}\right) + 4\mu(\tau_d+1)\alpha(1-\alpha)n^2p_{ij}\beta_{ij}} \right\}.
\end{align}
Then if~\eqref{eq:essstrongquasimono} holds, there exists a constant $C(z^0, \phi^0) \in \vR_{\geq 0}$ depending on $x^0$ and $\phi^0$ such that for all $k \in \vN$, 
\begin{align*}
\EE\left[d_{\cS}^2(x^k)\right] &\leq  \left(1 - \frac{2\alpha\mu\lambda}{\tau_p+1} \right)^{k/(\tau_p + 1)}\left(d_\cS^2(x^0) + C(x^0, \phi^0)\right).
\end{align*}
\end{enumerate}
\end{theorem}

The proof of Theorem~\ref{thm:newalgoconvergence} without delay (i.e., $d_k \equiv 0$ and $e^{ij}_{k}  \equiv 0$) is presented in Section~\ref{sec:nonasync}. The proof of the full theorem is in Appendix~\ref{sec:async}

\section{Connections with Other Algorithms}\label{sec:connections}

\begin{table}
\renewcommand{\arraystretch}{1.3}
\begin{center}
\begin{tabular}{c|c|c|c}\hline
\multirow{2}{*}{Algorithm} & \multicolumn{2}{c|}{Stepsize $\lambda$}  & \multirow{2}{*}{Rate for best $\lambda$} \\\cline{2-3}
  & that is largest possible & that gives best rate  &\\\hline\hline
 \multirow{2}{*}{SAGA~\eqref{eq:SAGABasic}} & \multirow{2}{*}{$\frac{1}{2L}$} & \multirow{2}{*}{$ \frac{1}{4L+ \mu N }$} & \multirow{2}{*}{$1 - \frac{\mu}{4L+ \mu N }$}\\
 & & & \\\hline
  SVRG~\eqref{eq:SVRGavgupdate} & \multirow{2}{*}{$\frac{1}{2L}$} & \multirow{2}{*}{$\frac{1}{4L + \mu \tau}$} & \multirow{2}{*}{$1-\frac{\mu}{4L + \mu \tau}$}\\
 average update frequency $\tau$& & & \\\hline
 SVRG~\eqref{eq:SVRGschedupdate} & \multirow{2}{*}{$\frac{1}{(\tau + 2)L}$} & \multirow{2}{*}{$\frac{1}{2L(\tau + 2) + \mu (\tau + 1)}$} & \multirow{2}{*}{$1-\frac{\mu}{2L(\tau + 2) + \mu (\tau + 1)}$}\\
 scheduled update frequency $\tau$& & & \\\hline
 \multirow{2}{*}{Finito~\eqref{eq:finitoclonestd}} & \multirow{2}{*}{$\frac{1}{2}; \; \gamma = \frac{2}{L}$} & \multirow{2}{*}{$\frac{1}{4}; \; \gamma = \frac{1}{L}$} & \multirow{2}{*}{$1- \frac{1}{4N}\left(1-\sqrt{1 - \frac{\hat{\mu}}{L}}\right)$}\\
  & & & \\\hline
  \multirow{2}{*}{SDCA~\eqref{eq:SDCAclone}} & \multirow{2}{*}{$\frac{3}{4}$} & \multirow{2}{*}{$\frac{3}{8}$} & \multirow{2}{*}{$1-\frac{3\mu_0}{8(L + \mu_0N)}$}\\
  & & & \\\hline 
   \multirow{2}{*}{Alternating Projections~\eqref{eq:RandomizedProjectionClone}}& \multirow{2}{*}{$1$} & \multirow{2}{*}{$\frac{1}{2}$} & \multirow{2}{*}{$1-\frac{\min\{1, \varepsilon^2L^{-2}\}}{2N\hat{\mu}}$}\\
    & & & \\\hline 
 \multirow{2}{*}{Kaczmarz~\eqref{eq:Kaczmarzclone}} & \multirow{2}{*}{$1$} & \multirow{2}{*}{$\frac{1}{2}$} & \multirow{2}{*}{$1-\frac{1}{2N\|A^{-1}\|_2^{2}}$}\\
  & & & \\\hline 
\end{tabular}\\[5pt]
\caption{The stepsizes and convergence rates for the special cases of SMART introduced in Section~\ref{sec:connections}}\label{table:stepsizeconnections}
\end{center}
\end{table}

\subsection{SAGA, SVRG, and S2GD}\label{sec:SAGA}
In the simplest case of~\eqref{eq:Sroot} \vspace{-10pt}
\begin{align}\label{eq:simplesmooth}
\Min_{x \in \cH_0}\, \frac{1}{N} \sum_{i=1}^N f_i(x) 
\end{align}
where each $f_i$ is convex and differentiable, and $\nabla f_i$ is $L$-Lipschitz, we set $S_i :=\nabla f_i$. (We also set $\cH = \cH_0$, use the canonical norm $\|\cdot\| = \|\cdot\|_0$, and ignore coordinates and partial derivatives for the moment.) By Proposition~\ref{eq:SAGAOP}, we can set $\beta_{ij} =(LN)^{-1}$ for all $i$ and $j$. With this choice of operators, the condition~\eqref{eq:essstrongquasimono} is, of course, implied by the $\mu$-strong convexity of $f : = N^{-1} \sum_{i=1}^N f_i$. But whether or not $f$ is strongly convex, if $\lambda$ is set according to Theorem~\ref{thm:newalgoconvergence} or Table~\ref{table:stepsizeconnections}, SMART will converge.

The SAGA Algorithm~\cite{defazio2014saga} applied to~\eqref{eq:simplesmooth} selects a function uniformly at random, performs a primal update, and updates a single dual variable (i.e., the trigger graph is completely disconnected). When, for all $i$, we set $y_i^0 = \nabla f_i(\phi_i^0)$ with $\phi_i^0 \in \cH$, SAGA takes the form:\footnote{Use $n = N, m = 1, d_{k} \equiv 0$, $ e^{i}_{k} \equiv 0$, $q_1 \equiv 1$, $\tau_p = \tau_d = 0$, $p_{ij} \equiv N^{-1}$, $\rho = 1$,  $E = \{(i, i) \mid i \in V\}$, $p_{ij}^T = N^{-1}$.}
\begin{align*}
 x^{k+1} &= x^k - \lambda \left(\nabla f_{i_k}(x^k) - y_{i_k}^k + \frac{1}{N} \sum_{i=1}^N y_i^k\right);\\
y_{i}^{k+1} &= \begin{cases} 
\nabla f_{i}(x^k)& \text{if $i_k = i$;} \\
 y_i^k &\text{otherwise.}
\end{cases}
\numberthis\label{eq:SAGABasic}
\end{align*}
In the SAGA algorithm, each dual variable is just a gradient, $y_i^k := \nabla f_i(\phi_i^k)$, for a past iterate $\phi_i^k$. The SAGA algorithm stores the stale gradients $\{ \nabla f_i(\phi_i^k) \mid i = 1, \ldots, N\}$, which, if $x \in \RR^d$, is the size of a $d\times N$ matrix. However, in logistic and least squares regression problems, the functions $f_i$ have a simple form $f_i(x) = \psi_i(\dotp{a_i, x})$ where the $\psi_i : \RR \rightarrow \RR$ are differentiable functions and the $a_i \in \RR^d$ are datapoints; in this case, $\nabla f_i(x) = \psi_i'(\dotp{a_i, x}) a_i$, so the cost of storing $\{\nabla f_i(\phi_i^k) \mid i = 1, \ldots, N\}$ can be reduced to that of $(\psi_1'(\dotp{a_1, \phi_1^k}), \ldots, \psi_N'(\dotp{a_N, \phi_N^k}))^T \in \RR^d$---a $d$-dimensional vector. But not all problems have this parametric form, so the SAGA algorithm is somewhat limited in scope.

The Stochastic Variance Reduced Gradient (SVRG) algorithm~\cite{johnson2013accelerating}, and the similar S2GD algorithm~\cite{konevcny2013semi}, solve~\eqref{eq:simplesmooth}, but in contrast to SAGA, these algorithms store just a single vector, namely $\nabla f(\widetilde{x}^k)$. As a consequence, SVRG and S2GD must make repeated, though infrequent, evaluations of the full gradient $\nabla f$---in addition to one extra evaluation of $\nabla f_{i_k}$ per-iteration:
\begin{align*}
x^{k+1} &= x^k - \lambda \left( \nabla f_{i_k}(x^k) - \nabla f_{i_k}(\widetilde{x}^k) + \frac{1}{N} \sum_{i=1}^N \nabla f_i(\widetilde{x}^k)\right); \\
\widetilde{x}^{k+1} &= \begin{cases}
x^{k - t_k}  &\text{if $k \equiv 0 \mod \tau-1$}; \\
\widetilde{x}^{k+1} & \text{otherwise;}
\end{cases}
\end{align*}
where $\{i_k\}_{k \in \NN}$ is an IID sequence of uniformly distributed $\{1, \ldots, N\}$-valued random variables and $\{t_k\}_{k \in \NN}$ is an IID sequence of uniformly distributed $\{0, \ldots, \tau-1\}$-valued random variables. The full gradient $\nabla f(\widetilde{x}^k)$ and the point $\widetilde{x}^k$ are only updated once every $\tau \in \NN$ iterations. 

The SVRG algorithm\footnote{From here on, we will ignore the distinction between SVRG and S2GD.} solves the SAGA storage problem, but it requires $f$ to be strongly convex, so it is also somewhat limited in scope. SMART can mimic SVRG---even without strong convexity---by selecting a function uniformly at random, performing a primal update, and updating \emph{all} dual variables with probability $\tau^{-1}$ (i.e., the trigger graph is the complete graph). When, for all $i$, we set $y_i^0 = \nabla f_i(\phi^0)$ with $\phi^0 \in \cH$, our SVRG clone takes the form:\footnote{Use $n = N, m = 1,  \umu_1 = 1, d_{k} \equiv 0$,  $e^{i}_{k} \equiv 0$, $q_j \equiv 1$, $\tau_p = \tau_d = 0$, $p_{ij} \equiv N^{-1}$, $\rho = \tau^{-1}$, $E = V \times V$, and $p_{ij}^T = 1$.}
\begin{align*}\vspace{-10pt}
x^{k+1} &= x^k - \lambda \left( \nabla f_{i_k}(x^k) -y_{i_k}^k + \frac{1}{N} \sum_{i=1}^N y_i^k\right); \\
\left(\forall i \right) \qquad y_i^{k+1} &= y_i^k + \epsilon_k(\nabla f_i(x^k) - y_i^k). \numberthis\label{eq:SVRGavgupdate} 
\end{align*}
As in SAGA, each dual variable is just a gradient $y_i^k = \nabla f_i(\phi^k)$, for a past iterate $\phi^k$, but unlike SAGA, the past iterate is the same for all $i$.  On average, all dual variables $y_i^k$ and, hence, the full gradient $\nabla f(\phi^k)$, are only updated once every $\tau$ iterations.  

Another clone of SVRG, this time with dual variables that update once every $\tau$ iterations, comes from SMART as applied in~\eqref{eq:SVRGavgupdate}, but with a cyclic uniform  dual variable delay $e_k^i = e_k := k \mod (\tau+1)$:\footnote{Use $n = N, m = 1,  \umu_1 = 1, d_{k} \equiv 0$,  $e^{i}_{k} = k \mod \tau$, $q_j \equiv 1$, $\tau_p = 0,  \tau_d = \tau$, $p_{ij} \equiv N^{-1}$, $\rho = 1$, $E = V \times V$, and $p_{ij}^T = 1$.}
\begin{align*}\vspace{-10pt}
x^{k+1} &= x^k - \lambda \left( \nabla f_{i_k}(x^k) -  y_{i_k}^{k - e_k} + \frac{1}{N} \sum_{i=1}^N y_i^{k-e_{k}}\right); \\
\left(\forall i \right) \qquad y_i^{k+1} &= \nabla f_{i}(x^k). \numberthis\label{eq:SVRGschedupdate}
\end{align*}
The dual variable $y_i^k = \nabla f_i(x^k)$, and hence the full gradient $\nabla f(x^k)$, is only updated once every $\tau$ iterations. 

\subsection{Finito}\label{sec:finito}

The Finito algorithm~\cite{defazio2014finito} solves~\eqref{eq:simplesmooth}, but unlike SAGA and SVRG, Finito stores one point and one gradient per function, or a superposition of the two: 
\begin{align*}
x_{i}^{k+1} &= \begin{cases}
 \frac{1}{N}\sum_{l=1}^N(x_l^k- \frac{1}{2\hat{\mu}}\nabla f_l(x_l^k)) & \text{if $i = i_k$;} \\
 x_i^k & \text{otherwise,}
 \end{cases}
\end{align*}
where each $f_i$ is $\hat{\mu}$-strongly convex. For each function $f_i$, Finito stores $x_i^k- (2\hat{\mu})^{-1}\nabla f_i(x_i^k)$, which can be substantially costlier than storing a matrix of gradients. In addition, only when each function $f_i(x)$ is $\hat{\mu}$-strongly convex and the bound $N \geq 2L\hat{\mu}^{-1}$ holds, is Finito known to converge. 

SMART can mimic SAGA and SVRG with multiple operators, but with one operator $S = S_1$, SMART recovers the Finito algorithm. To get Finito, recast~\eqref{eq:simplesmooth} into an equivalent form with duplicated variables
\begin{align*}
\Min_{(x_1, \ldots, x_N) \in \cH_0^N} \;\frac{1}{N} \sum_{i=1}^N f_i(x_i) \qquad\text{subject to:} \; x_1 = x_2 = \cdots = x_N,
\end{align*}
let $D := \{(x, \ldots, x) \in \cH_0^N \mid x \in \cH\}$ denote the \emph{diagonal} set, and define the operator (for a fixed $\gamma > 0$)
\begin{align*}
\left(\forall x \in \cH_0^N\right) \qquad S(x) := S_1(x) = x - P_D  \left(x_1 - \gamma \nabla f_1(x_1), \ldots, x_N - \gamma \nabla f_N(x_N)\right).
\end{align*}
Then the Finito algorithm, selects a coordinate (and hence a function) uniformly at random and performs a primal update; the sole dual variable is set to zero at all iterations:\footnote{Use $n = 1, m = N, \umu_j \equiv 1, d_{k}  \equiv 0$, $e^{i}_{k} \equiv 0 $, $\mathbf{S}^\ast = 0$, $q_j \equiv N^{-1}$, $\tau_p = 0$, $\tau_d = 0$, $p_{ij} \equiv 1$, $\rho = 1$, $E = \{(1, 1)\}$, and $p_{ij}^T = N^{-1}$.}
\begin{align}\label{eq:finitoclonestd}\vspace{-10pt}
x^{k+1}_j = \begin{cases}
(1-\lambda) x_j^k  +  \frac{\lambda}{N}\sum_{l=1}^N \left( x_l^k- \gamma \nabla f_l(x_l^k)\right) & \forall  j \in \sfS_k; \\ 
x_j^k & \text{otherwise.}
\end{cases}
\end{align}
Unlike the standard Finito algorithm~\cite{defazio2014finito}, Algorithm~\eqref{eq:finitoclonestd} converges with or without strong convexity.

The constants $\beta_{1j}$ depend on the constant $\gamma$ and the inner products that we place on $\cH_j := \cH_0$ and $\cH := \cH_0^N$. The projection operator $P_D$ also depends on these inner products. As long as\footnote{See Proposition~\ref{prop:finito_op}.} $\gamma\leq 2L^{-1}$, the operator $S$ satisfies~\eqref{eq:coherence} with~$\beta_{1j} \equiv 4^{-1}\gamma L$, and if $\mu := 1 - \sqrt{1-2\gamma\hat{\mu} + \gamma^2\hat{\mu} L}$, it is $\mu$-essentially strongly quasi-monotone---provided that each $f_i$ is $\hat{\mu}$-strongly convex and we make the choices $\dotp{x_j, z_j}_j := \dotp{x_j, z_j}_0$ and $\|\cdot\| = \|\cdot\|_\pr$; the operator $S$ need not be strongly monotone unless each $f_i$ is strongly convex.  See Proposition~\ref{prop:finito_op}.



The space $\cH = \cH_0^N$ is high-dimensional, so even storing a single vector $x^k$ is expensive. And in practice, the gradients should also be stored---unless they are simple to recompute. Thus, it is clear that Finito, like SAGA, but unlike SVRG, is  impracticable if $m$ is too large and memory is limited. Nevertheless, Finito performs well in practice, often better than other incremental gradient methods. 


\subsection{SDCA}\label{sec:SDCA}

The Stochastic Dual Coordinate Ascent (SDCA) algorithm~\cite{shalev2014accelerated} solves a problem different from~\eqref{eq:simplesmooth}:\footnote{In the standard SDCA problem, $f_i(z) := \psi_i(A_i^T z)$ for a convex, differentiable function $\psi_i$ and a matrix $A_i$. Furthermore, the squared 2-norm is replaced with a general strongly convex function $g$.}
\begin{align}\label{eq:SDCA_prob}
\Min_{z\in \cH_0} \frac{1}{N}\sum_{j=1}^N f_j(z) + \frac{\mu_0}{2}\|z\|^2
\end{align}
SDCA does not solve this primal problem directly; instead, SDCA solves the dual problem.
\begin{align*}
\Min_{(x_1, \ldots, x_N)\in \cH_0^N} \frac{1}{N}\sum_{j=1}^N f_j^\ast(-x_j) + \frac{\mu_0}{2}\left\|\frac{1}{\mu_0 N} \sum_{j=1}^N x_j \right\|^2.
\end{align*}
If we define $\{\sfS_k\}_{k \in \NN}$ as in SMART and restrict all $\sfS_k$ to have at most one element, SDCA repeatedly does the following:
\begin{align}\label{eq:SDCA}
x_{j}^{k+1} &= \begin{cases}
x_j^k + \argmin_{x_j\in \cH_0} \left\{ f_j^\ast(-x_j^k - x_j) + \frac{\mu_0 N}{2} \left\|\frac{1}{\mu_0 N}\sum_{l=1}^N x_l^k + \frac{1}{\mu_0 N} x_j\right\|^2\right\} & \text{if $j \in \sfS_k$;}\\
x_j^k & \text{otherwise.}
\end{cases}
\end{align}
For each function $f_j$, SDCA stores the vector $x_j^k$. This can be cheap because, as the optimality conditions of~\eqref{eq:SDCA} show, there is always a point $z_j^k$ such that $x_j^k \in \partial f(x_j^k)$. The drawback to \eqref{eq:SDCA} is that only when each $f_i$ is differentiable, is SDCA known to converge. In addition, the form~\eqref{eq:SDCA} is opaque, but with a bit of polishing, it can be made transparent.

Evidently, the first line of~\eqref{eq:SDCA} is a \emph{forward-backward} step: 
$$ 
x_j^{k+1} = \prox_{\mu_0N f_j^\ast(-\cdot)}\left(x_j^k - \mu_0N \left[\nabla g(x^k)\right]_j\right)
$$ where $g(x) := 2^{-1}\mu_0N\left\|(\mu_0N)^{-1}\sum_{l=1}^N x_l\right\|^2$. Thus, SDCA is an instance of SMART with $f^\ast :=  \sum_{j=1}^N f_j^\ast(x_j)$,
\begin{align*}
\left(\forall x\in \cH_0^N\right) \qquad S(z) = S_1(x) := x - \prox_{ \mu_0N f^\ast(-\cdot)}(x - \mu_0N\nabla g(x)),
\end{align*}
that selects, at each iteration, a single coordinate (hence, a function) uniformly at random and performs a primal update; the sole dual variable is set to zero at all iterations:\footnote{Use $n = 1, m=N, \umu_j \equiv 1, d_{k} \equiv 0$,  $e^{i}_{k} \equiv 0$, $\mathbf{S}^\ast = 0$, $q_j \equiv N^{-1}$, $\tau_p = \tau_d = 0$, $p_{ij} \equiv 1$, $\rho = 1$, $E = \{(1, 1)\}$, and $p_{ij}^T = N^{-1}$.}
\begin{align}\label{eq:SDCAclone}
x^{k+1}_j = \begin{cases}
 x_j^k  - \lambda\left(x_j^k -  \prox_{\mu_0N f_j^\ast(-\cdot)}\left(x_j^k - \mu_0N \left[\nabla g(x^k)\right]_j\right)\right)& \forall j \text{ with } j \in \sfS_k\\ 
x_j^k & \text{otherwise.}
\end{cases}
\end{align}
Unlike~\eqref{eq:SDCA}, Algorithm~\eqref{eq:SDCAclone} converges whether or not any $f_j$ is differentiable.

The constants $\beta_{1j}$ depend on the inner products that we place on $\cH_j := \cH_0$ and $\cH := \cH_0^N$; as we did for Finito, we set $\dotp{x_j, z_j}_j := \dotp{x_j, z_j}_0$, and $\|\cdot\| = \|\cdot\|_\pr$.  With this choice, the operator $S$ satisfies~\eqref{eq:coherence} with $\beta_{1j} \equiv 3/4$, and it is $\mu_0N(\mu_0N + L)^{-1}$-strongly quasi-monotone. See Proposition~\ref{prop:SDCA_op}.

Although the space $\cH$ is high-dimensional, the condition $x_j^0 \in\text{range}(\partial f_j)$ ensures that each component $x_j^k$ will lie in the linear span of $\text{range}(\partial f_j)$, which is often one-dimensional.\footnote{The Moureau identity implies that for all $x_j$, we have $\prox_{N\mu_0 f_j^\ast(-\cdot)}\left(x_j\right) \in \text{range}(\partial f_j). $}

\subsection{Randomized Projection Algorithms}\label{sec:RPA}

Besides minimization problems, SMART also solves feasibility problems:
\begin{align}\label{eq:feasibilityproblem}
\text{Find } x \in \bigcap_{i=1}^{s_1} C_i \qquad \text{ subject to: } \qquad f_i(x) \leq 0 \text{ for } i = 1, \ldots, s_2, \qquad (N = s_1 + s_2)
\end{align}
where $C_1, \ldots, C_{s_1} \subseteq \cH$ are closed, convex sets, and $f_i : \cH \rightarrow (-\infty, \infty)$ are continuous, convex functions. To align~\eqref{eq:feasibilityproblem} with~\eqref{eq:Sroot}, each set and each function are assigned an operator. The sets $C_i$ are assigned the familiar projection operator $P_{C_i}$. The functions are assigned a \emph{subgradient projector}
$$
\left(\forall x \in \cH\right) \qquad G_{f_i}(x)  :=\begin{cases} x - \frac{f_i(x)}{\|g_i(x)\|^2} g_i(x)& \text{if $f_i(x) > 0$}\\
x& \text{otherwise.}
\end{cases}
$$
where $g_i : \cH \rightarrow \cH$ is a \emph{measurable subgradient selector}; i.e., for all $x \in \cH$, $g_i(x) \in \partial f_i(x)$. We assume that $0 \notin \partial f_i(\{x \mid f_i(x) > 0\})$ so that the projector is well-defined.

Then Problem~\eqref{eq:feasibilityproblem} is an instance of~\eqref{eq:Sroot} with $n = s_1 + s_2$ and 
\begin{align*}
\left(\forall x \in \cH\right) \qquad S_i(x) := \begin{cases}
x - P_{C_i}x & \text{if } i = 1, \ldots, s_1; \\
x - G_{f_{i-s_1}}(x) & \text{if } i = s_1+1, \ldots, s_1 + s_2. 
\end{cases}
\end{align*}
With SMART, we can select a function or a set uniformly at random and perform a primal update; all operators are zero at points in $\zer(S)$, so all dual variables are zero:\footnote{Use $n = N, m = 1, \umu_1 \equiv 1, d_{k} \equiv 0$,  $e^{i}_{k} \equiv 0$, $\mathbf{S}^\ast = 0$, $q_j \equiv 1$, $\tau_p = \tau_d = 0$, $p_{ij} \equiv N^{-1}$, $\rho = 1$, $E = \{(i,i) \mid i \in V\}$, and $p_{ij}^T = N^{-1}.$}
\begin{align} \label{eq:RandomizedProjectionClone}\vspace{-10pt}
x^{k+1} := x^k - \lambda 
\begin{cases}
x^k - P_{C_{i_k}}x^k & \text{if } i_k \in \{1, \ldots, s_1\}; \\
x^k - G_{f_{i_k-s_1}}(x^k) & \text{if } i_k \in \{s_1+1, \ldots, s_1 + s_2\}. 
\end{cases}
\end{align}
In general, the operator $S$ in~\eqref{eq:feasibilityproblem} has several nice properties: $\zer(S)$ is the set of solutions to~\eqref{eq:feasibilityproblem}; $S$ satisfies~\eqref{eq:coherence} with $\beta_{i1} \equiv N^{-1}$; $S$ is demiclosed; and if 
\begin{enumerate}
\item $\{C_i \mid i = 1, \ldots, s_1\} \cup \{\{x \mid f_i(x) \leq 0\} \mid i = 1, \ldots, s_2\}$ are $\hat{\mu}$-linearly regular,\footnote{A set family $\{D_1, \ldots, D_N\}$ is $\hat{\mu}$-linearly regular if $\forall x\in \cH$, $d_{D_1\cap\cdots \cap D_N}(x) \leq \hat{\mu}\max\{d_{D_1}(x), \ldots, d_{D_N}(x)\}$.}
\item there is an $\varepsilon > 0$ such that $f_i(x) \geq \varepsilon d_{\{ f_i(x) \leq 0\}}(x)$ for all $x \in \cH$, 
\item and there is an $L > 0$ such that $\|g_i(x)\| \leq L$ for all $x \in \cH$,
\end{enumerate}
then, with $\mu := (N\hat{\mu})^{-1}\max\{1, \varepsilon^2L^{-2}\}$, the operator $S$ is $\mu$-essentially strongly quasi-monotone. Thus, $x^k$ converges linearly, which is a new result for~\eqref{eq:RandomizedProjectionClone}.\footnote{See Proposition~\ref{prop:RPA}.}

The randomized Kaczmarz algorithm~\cite{strohmer2009randomized,liu2014asynchronous}, which solves overdetermined linear systems $Ax = b$, is a special case of~\eqref{eq:RandomizedProjectionClone}: if $a_1, \ldots, a_{N}$ are the rows of $A$, which we assume, without loss of generality, are normalized  and $C_i := \{ x \mid \dotp{a_i,x} = b_i\}$, then the projection is $P_{C_i}(x) = x + (b_i - \dotp{a_i, x})a_i$, and 
\begin{align}\label{eq:Kaczmarzclone}
x^{k+1} = x^k + \lambda(b_{i_k} - \dotp{a_{i_k}, x})a_{i_k}
\end{align}
The Kaczmarz operator is $(\|A^{-1}\|_2^{-2}N^{-1})$-essentially strongly quasi-monotone, and so $x^k$ linearly converges to a solution of $Ax = b$.\footnote{We define $\|A^{-1}\|_2 := \inf \{M \mid \left(\forall x \in \cH\right) \; M\|Ax \|_2 \geq \|x\|_2 \}$; see Corollary~\ref{cor:Kaczmarz}.}

\subsection{ARock}

When $n = 1$, and hence, $S = S_1$, SMART recovers the ARock Algorithm~\cite{peng2015arock}, which at iteration $k$, samples a coordinate $j$ and updates
\begin{align*}
x^{k+1}_j &= x^k_j - \frac{\lambda}{q_j m} (S(x^{k-d_k}))_j; \\
\left(\forall j' \neq j\right)\qquad x_j^{k+1} &= x_j^k.
\end{align*}
When we specialize SMART to this simple single operator case, which is the only case that ARock applies to, there is no difference between ARock and SMART; but Theorem~\ref{thm:newalgoconvergence} guarantees that ARock will converge under conditions weaker than those presented in~\cite{peng2015arock}.

In~\cite{peng2015arock}, the underlying Hilbert space $\cH = \cH_1 \times \cdots \times \cH_m$ is a product of several other Hilbert spaces, and for ARock to converge, the norm $\|\cdot\|$ on $\cH$ must be equal to the standard product norm: $\|x\|_\pr^2 = \sum_{j=1}^m \|x_j\|_j^2$. We removed this assumption, and this significantly extends the problems that ARock can solve;  in Sections~\ref{sec:TropicSMART},~\ref{sec:ProxSMART},~and~\ref{sec:ProxSMART+} we present new algorithms that use nonstandard norms. 

In~\cite{peng2015arock}, the operator $S$ must satisfy 
$$
\left(\forall x \in \cH\right), \left(\forall y\in \cH\right) \qquad  \dotp{S(x) - S(y), x- y} \geq \frac{1}{2} \|S(x) - S(y)\|^2,
$$
where the inner product on left and the norm on the right are, again, both the standard ones on $\cH$. This cocoercivity condition is plainly stronger than~\eqref{eq:coherence}, and for example, fails for the subgradient projector of Section~\ref{sec:RPA}.

ARock converges linearly only when 
\begin{align*}
&(\exists \mu > 0): (\forall x\in \cH), \left(\forall x^\ast \in \cS\right) &\dotp{S(x), x - x^\ast} &\geq \mu\|x- x^\ast\|^2,
\end{align*}
which is, of course, plainly stronger than~\eqref{eq:essstrongquasimono}, and for example, requires that $\cS$ is a singleton.

But the biggest limitation of ARock, a limitation that we remove in SMART, is that to solve simple problems, such as~\eqref{eq:simplesmooth}, an extra primal variable must be introduced for each smooth term, and these primal variables are not low dimensional, unlike the dual variables of~SMART, which tend to be gradients of functions of the form $f(\dotp{a_i, x})$; this difference is comparable to the difference between the low memory methods SAGA/SVRG and and high memory method Finito.
\section{What's New: Improving Existing Algorithms}

In Section~\ref{sec:connections} we introduced a few algorithms and described how to recover them with SMART, but we did not discuss new features which are obtainable from SMART. We do that now.

\subsection{Weaker Conditions for Convergence} 

Only when the objectives are differentiable and strongly convex are SVRG, Finito, and SDCA known to converge. But with SMART, strong convexity can be dropped in all cases.  Furthermore, in SDCA, the objectives, $f_j$, need only be convex, proper, and closed. 

SAGA and SVRG are known to converge linearly when $f$ is strongly convex. By Theorem~\ref{alg:main}, they converge linearly when $N^{-1}\sum_{i=1}^N\nabla f_i(x_i)$ is essentially strongly quasi-monotone, which occurs, for example, when there is a strongly convex function $g_i$ and a linear map $A_i$ such that $f_i(x) = g_i(A_ix)$.

Finito is also known to converge linearly when $f_j$ is strongly convex. Again, by Theorem~\ref{alg:main}, Finito converges linearly when its operator $S$ is essentially strongly quasi-monotone; the weakest conditions under which this occurs appears to be an open problem in the study of error bounds.

SDCA is known to converge linearly when each $f_j$ is strongly convex and differentiable with Lipschitz continuous gradient. SDCA will still converge linearly if each $f_j$ is just differentiable, but not necessarily strongly convex; see Lemma~\ref{prop:SDCA_op} for a proof of this simple fact. 

\subsection{Proximable Terms}\label{sec:proximable_SAGA}

SAGA and SVRG also solve problems in which a single nonsmooth term is added to~\eqref{eq:simplesmooth}:
\begin{align}\label{eq:simplesmooth_one_nonsmooth}
\Min_{x \in \cH_0}\, g(x) + \frac{1}{N} \sum_{i=1}^N f_i(x) ,
\end{align}
where $g : \cH_0 \rightarrow [-\infty, \infty)$ is closed, proper, and convex. In this composite case, the update rule is only slightly changed, for example, in SAGA~\eqref{eq:SAGABasic}, we replace the primal update with
$$
 x^{k+1} = \prox_{\lambda g}\left(x^k - \lambda \left(\nabla f_{i_k}(x^k) - y_{i_k}^k + \frac{1}{N} \sum_{i=1}^N y_i^k\right)\right).
 $$
This update rule is not a special case of SMART, but with almost no extra work, we can extend our proof of convergence in the synchronous case (Theorem~\ref{thm:syncnewalgoconvergence}) to show that the update rule works. However, when SMART is asynchronous, we hit a wall; progress seems unlikely. 

Instead of pursuing asynchronous versions of this update, we introduce a new update rule: Let
\begin{align*}
S_{i} = \frac{1}{LN}\nabla f_i\circ \prox_{L^{-1} g} \qquad i = 1, \ldots, N && \text{and} && S_{N+1} = (I-\prox_{L^{-1}g}),
\end{align*}
Then, for all $k \in \NN$, we get the proximal SAGA update:\footnote{Use $n = N, m = 1, d_{k} \equiv 0$, $ e^{i}_{k} \equiv 0$, $q_1 \equiv 1$, $\tau_p = \tau_d = 0$, $p_{ij} \equiv (2N)^{-1}$ for $i < N+1$ and $p_{(N+1)1} = 2^{-1}$, $\rho = 1$,  $E = \{(i, i) \mid i \in V\} \cup \{(i, N+1) \mid i \in V\} $, $p_{i1}^T = (2N)^{-1}$ for $i < N+1$, $p_{(N+1)1}^T = 1.$} select $\lambda < (N+1)8^{-1}$ and iterate
\begin{align*}
 x^{k+1} &= x^k - \lambda \begin{cases}
 \left( \frac{2}{L(N+1)}\nabla f_i( \prox_{L^{-1} g}(x)) - \frac{2N}{N+1}y_{i}^k + \frac{1}{N+1} \sum_{i=1}^{N+1} y_i^k\right) & \text{if $i_k = i$ and $i < N+1$};\\
 \left( \frac{2}{N+1}(I-\prox_{L^{-1} g})(x^k) - \frac{2}{N+1}y_{N+1}^k + \frac{1}{N+1} \sum_{i=1}^{N+1} y_i^k\right) & \text{if $i_k = N+1$}; 
 \end{cases}\\
y_{i}^{k+1} &= \begin{cases} 
\frac{1}{LN}\nabla f_{i}(\prox_{L^{-1} g}(x^k))& \text{if $i_k=i$ and $i < N+1$;} \\
\left(I - \prox_{L^{-1} g}\right)(x^k) & \text{if $i = N+1$}; \\
 y_i^k &\text{otherwise.}
\end{cases}
\numberthis\label{eq:ProximalSAGA}
\end{align*}
This special case of SMART has a trigger graph that is not completed disconnected, as it is in the standard SAGA algorithm; instead all vertices $i$ in the graph contain a directed edge starting at $i$ and ending at $N+1$. In a similar fashion, proximal SVRG algorithms arise from our choice of $S_i$, but as in~\eqref{eq:SVRGavgupdate} and~\eqref{eq:SVRGschedupdate}, the dual variable update probability should be $\rho = \tau^{-1}$ and the trigger graph should be the complete graph on $N+1$ vertices.

The zero set $\cS$ of $S = (N+1)^{-1}\sum_{i=1}^{N+1} S_i$ does not contain the set of minimizers of $f+g$, but it is related to the minimizers (if any exist) through the proximal operator of $g$:
\begin{quote}
\centering
For all $x^\ast \in \cS$, the point $\prox_{\gamma g}(x^\ast)$ minimizes $f+g$. 
\end{quote}
Unlike in SAGA and SVRG, the operators $S_i$ do not satisfy $\dotp{S_i(x) - S_i(x^\ast), x - x^\ast} \geq L^{-1} \|S_i(x) - S_i(x^\ast)\|^2$; but the sum $S$ satisfies~\eqref{eq:coherence} with 
\begin{align*}
\beta_{i1} = \frac{N}{2(N+1)} \qquad i = 1, \ldots, N && \text{and} && \beta_{(N+1)1} =  \frac{1}{2(N+1)};
\end{align*}
in Proposition~\ref{prop:SAGAopprop}, we give different possibilities for different $\gamma$. But with these parameters, we have
\begin{align*}
\lambda &\leq \frac{(N+1)(1+ \mu_gL^{-1})}{(16+2N)(1 + \mu_gL^{-1}) - 2N(\sqrt{1 - \mu_f L^{-1}})}\; \implies \text{linear rate: } \; \; \; 1 - \frac{1 + \mu_gL^{-1} - \sqrt{1 - \mu_f L^{-1}}}{(8+N)(1 + \mu_gL^{-1}) - N(\sqrt{1 - \mu_f L^{-1}})},
\end{align*}
where $N^{-1}\sum_{i=1}^N f_i$ is $\mu_f$-strongly convex and $g$ is $\mu_g$-strongly convex.\footnote{In this case, $S$ is $((N+1)(1 + \mu_gL^{-1}))^{-1}(1 + L^{-1} \mu_g - \sqrt{1 - L^{-1}\mu_f})$ essentially strongly quasi-monotone.}

\subsection{Coordinate Updates}

We recover Finito and SDCA with SMART by performing coordinate updates on an operator $S$. These operators are block-separable, but all algorithms in Section~\ref{sec:connections} will still converge if we perform finer, nonseparable updates.

For example, SAGA will converge all the same if we replace all full derivatives $\nabla f_i$ with partial derivatives $\nabla_j f_i$ (where $f_i$ is now viewed as a function on a space $\cH_1 \times \cdots \times \cH_m$, and the partial derivative is taken with respect to the coordinates in $\cH_j$, which could be an infinite dimensional space): given $x^k$, sample $j_k \in \{1, \ldots, m\}$ uniformly at random, and set
\begin{align*}
 x_{j}^{k+1} &= \begin{cases}
 x_{j}^k - \lambda \left(\nabla_{j} f_{i_k}(x^k) - y_{i_k, j}^k + \frac{1}{N} \sum_{i=1}^N y_{i,j}^k\right) & \text{if $j = j_k$;}\\
 x_j^k & \text{otherwise.}
 \end{cases}\\
y_{i, j}^{k+1} &= \begin{cases} 
\nabla_j f_{i}(x^k)& \text{if $i_k = i$ and $j_k = j$;} \\
 y_{i,j}^k &\text{otherwise.}
\end{cases}
\numberthis\label{eq:SAGACoordinate}
\end{align*}
This algorithm converges for the same range of $\lambda$ as~\eqref{eq:SAGABasic} (namely, for $\lambda < (2L)^{-1}$), but in the strongly convex case, the step size $\lambda$ that gives the best rate of convergence rate changes to $\lambda = m(4Lm+\mu N)^{-1}$, and correspondingly, the convergence rate changes to $1 - \mu(4Lm+\mu N)^{-1}$.

The story is similar for all algorithms in Section~\ref{sec:connections}.

\subsection{Importance Sampling and Better $\beta_{ij}$}

Until now, the constants $\beta_{ij}$ have been constant and equal to $L^{-1}$. But by making the finer distinction that $\nabla f_i$ is $L_i$-Lipschitz continuous, can choose larger $\lambda$ by sampling $i_k$ nonuniformly.

For example, in the SAGA algorithm~\eqref{eq:SAGABasic}, we have  
\begin{align}
p_{i1} = P(i_k = i) = \frac{L_i}{\sum_{i=1}^n L_i} \implies \lambda < \frac{1}{\frac{2}{n}\sum_{i=1}^n L_i}.\label{eq:importance}
\end{align}
Compared to $(2\max\{L_i\})^{-1}$, which is the stepsize obtained with uniform $i_k$ sampling, the above stepsize can be much larger. Similarly
\begin{align*}
\text{$i_k$ sampled according to \eqref{eq:importance} }\qquad  & \implies\qquad \text{linear rate: } \; \; \;  1 - \frac{\mu}{ \frac{4}{N}\sum_{i=1}^NL_i + \mu N};\\
\text{$i_k$ sampled uniformly } \qquad &\implies \qquad\text{linear rate: } \; \; \;  1 - \frac{\mu}{ 4\max_i\{L_i\} + \mu N}.
\end{align*}
Thus, importance sampling replaces maximums of Lipschitz constants by averages of Lipschitz constants. 

For smooth functions, the constants $\beta_{ij}$ appearing in~\eqref{eq:coherence} are related to inverse coordinatewise Lipschitz constants, which are the minimal values $L_{ij}$ such that
\begin{align*}
f_i(x + \hat{y_j}) \leq f_i(x) + \dotp{ \nabla f_i(x), \hat{y_j}} + \frac{L_{ij}}{2} \|\hat{y_j}\|_j^2; \qquad \hat{y_j} = (0, \ldots, 0, y_j, 0, \ldots, 0);
\end{align*}
for all $x \in \cH$ and $y_j \in \cH_j$. From the above inequality, comes the relationship\footnote{See Proposition~\ref{prop:coordinate-BH} for a proof in our general case, and~\cite[Lemma 4]{konevcny2014semi} for a proof in the case that $\cH_j = \RR$.} 
$$
\frac{1}{L_{ij}}\|\nabla_j f_i(x) - \nabla_j f_i(y)\|^2_j \leq \dotp{\nabla f_i(x) - \nabla f_i(y), x- y},
$$
which is not a favorable one because the right hand side depends on the full derivative $\nabla f_i$. 

However, the relationship improves if, say, at most $s \ll n$ partial derivatives $\nabla_j f(x_1, \ldots, x_n)$ depend on each coordinate; from that assumption, comes the relationship
$$
\sum_{i=1}^n \frac{1}{sL_{ij}}\|\nabla_j f_i(x) - \nabla_j f_i(y)\|^2_j \leq \dotp{\nabla f_i(x) - \nabla f_i(y), x- y}.
$$
But notice that if $L_i$ is the minimal Lipschitz constant $\nabla f_i$, then $L_i^{-1} \geq \min_j\{(sL_{ij})^{-1}\}$.

\subsection{Mini Batching} 

The \emph{pre-update} mini batching method adjusts problem~\eqref{eq:simplesmooth} by grouping functions together according to $\cB \subseteq 2^{\{1, \ldots, N\}}$:
$$
\Min_{x \in \cH} \frac{1}{N} \sum_{B \in \cB} \sum_{i \in B} \frac{1}{N(i)} f_i; \qquad \left(\forall i\right) \; N(i) := |\{ B \in \cB \mid i \in B\}|.
$$
Then it runs one of the algorithms from Section~\ref{sec:connections} with the $n = |\cB|$ functions $\sum_{i\in B} N(i)^{-1} f_i$. 

The parameter $n$ and hence, the number of dual variables in the pre-update mini batching method can be impractically large. To save memory, only our SVRG clone~\eqref{eq:SVRGavgupdate} should be used with the pre-update mini batching method; with this method, a mini batch $B_{i_k} \in \cB$ of gradients is computed at every iteration, and on average all of the dual variables are, and hence, the full gradient is, updated once per $\tau^{-1}$ iterations: label the elements of $\cB = \{B_1, \ldots, B_n\}$, and iterate\footnote{The pre-update mini batching SVRG clone is similar to the algorithm in~\cite{konevcny2014ms2gd}.}
\begin{align*}\vspace{-10pt}
x^{k+1} &= x^k - \lambda \left( \sum_{l\in B_{i_k}} (\nabla f_{l}(x^k) -\nabla f_{l}(\phi^k)) + \frac{1}{N} \sum_{i=1}^N \nabla f_i (\phi^k)\right); \\
\left(\forall i \right) \qquad \phi^{k+1} &= \phi^k + \epsilon_k(x^k - \phi^k). 
\end{align*}
In this SVRG clone, we eliminate all of the dual variables, and as a result, we save a lot of memory.  

Unlike the pre-update mini batching method, the \emph{post-update} mini batching method does not adjust problem~\eqref{eq:simplesmooth}; it only adjusts the trigger graph. Consequently, there is no grouping $\cB$ which pairs functions together. In place of a grouping, for each function, the trigger graph $G$ \emph{triggers} a gradient computation for some other set of functions: 
\begin{align*}
 x^{k+1} &= x^k - \lambda \left(\nabla f_{i_k}(x^k) - y_{i_k}^k + \frac{1}{N} \sum_{i=1}^N y_i^k\right);\\
y_{i}^{k+1} &= \begin{cases} 
\nabla f_{i}(x^k)& \text{if $i_k$ triggers $i$;} \\
 y_i^k &\text{otherwise.}
\end{cases}
\end{align*}
Here the mini batching is arbitrary, but we have already seen three specific examples of trigger graphs: the completely disconnected graph (SAGA), the completely connected graph (SVRG), and an internally directed star graph (Proximal SAGA).

Convergence rates improve with post-update mini batching because the parameter $\eta := \min_{i, j} \{\rho p_{ij}^T\},$ increases. For example, if every node $i \in \{1, \ldots, N\}$ in the trigger graph is triggered by $N_T$ other nodes, then $p_{ij}^T := N^{-1} N_T$, and the convergence rate of SAGA improves:
\begin{align*}
\lambda &\leq \frac{1}{4L + \frac{8 \mu N}{N_T}}\; \implies \text{linear rate: } \; \; \; 1 - \frac{\mu}{4L + \frac{8 \mu N}{N_T}}.
\end{align*}
Compared to $1 - \mu(4L+ \mu N)^{-1}$, which is the convergence rate for standard SAGA, the above convergence rate can be much better. See~\cite{neighorhoodwatch} for similar convergence rate improvements from mini batching.

\subsection{Asynchronous Updates}

Besides ARock, all algorithms in Section~\ref{sec:connections} are synchronous because the primal and dual updates are not delayed. On paper, delaying is just a matter of changing indices, as in the \emph{asynchronous SAGA} algorithm:
\begin{align*}
 x^{k+1} &= x^{k} - \lambda \left(\nabla f_{i_k}(x^{k-d_k}) - y_{i_k}^{k-e_{k}^i} + \frac{1}{N} \sum_{i=1}^N y_i^{k-e_k^i}\right);\\
y_{i}^{k+1} &= \begin{cases} 
\nabla f_{i}(x^{k-d_k})& \text{if $i_k = i$;} \\
 y_i^k &\text{otherwise.}
\end{cases}
\end{align*}
However, on a computer, implementing an asynchronous algorithm can be difficult. For a brief discussion on implementation issues, see~\cite{peng2015arock}.

\section{What's New: Algorithms}\label{sec:new_opt_alg}

On the surface, the shape of an operator plainly resembles the shape of a gradient, and this is already a powerful observation, leading to the SAGA and SVRG algorithms. But operators also model nonsmooth optimization problems and even monotone inclusions; choosing these operators just requires a bit of experience, which anyone can acquire, for example, by understanding the examples presented in this section. 

The examples presented here isolate common problems, reformulate these problems with operators $S_i$, and then solve these problems with SMART. Once we choose the operators $S_i$, we can apply SMART in exponentially many ways, for example, by choosing arbitrary delays, sampling probabilities, and trigger graphs; in this section, we avoid endless customization, and instead, we apply SMART in a simple, arbitrary manner. 

\subsection{LinSAGA and LinSVRG}\label{sec:LinSAGA}

LinSAGA and LinSVRG add a linear constraint to the proximal SAGA and proximal SVRG problems: 
\begin{align*}
\Min_{x \in \cH}&\; g(x) + \frac{1}{N}\sum_{i=1}^{N} f_i(x); \\
\text{subject to:} &\;  x \in V, \numberthis\label{eq:linSAGA}
\end{align*}
where the function $g : \cH \rightarrow (-\infty, \infty]$ is closed, proper, and convex; the functions $f_i : \cH \rightarrow (-\infty, \infty)$ are differentiable and the gradients $P_V \circ \nabla f_i \circ P_V $ are $\hat{L}_i$-Lipschitz continuous;\footnote{The operator $ P_V \circ \nabla f_i  \circ P_V$ is the gradient of the convex function $f_i \circ P_V$, and its Lipschitz constant, which we denote by $\hat{L}_i$, is generally smaller than the Lipschitz constant of $\nabla f_i$ (see the discussion surrounding~\cite[Lemma 1.3]{myFDRS}.)} and the set $V \subseteq \cH$ is a vector space.\footnote{Affine constraints, say, $Ax = b$ can replace the linear constraint $ x\in V$ provided we choose any $c \in \cH$ such that $A c = b$, change $f(x)$ and $g(x)$ to $f(x + c)$, and $g(x+c)$, respectively, and set $V = \ker(A)$.} We assume that $P_V$ and $\prox_{\gamma g}$ are both easy to evaluate.

We model this problem with $N+1$ operators (for some $\gamma > 0$):
\begin{align*}
\left(\forall i < N +1\right) \qquad& S_i := \frac{\gamma}{N} P_V \circ \nabla f_i  \circ P_V \circ \prox_{\gamma g};\\
&S_{N+1} := (I- 2P_V)\circ \prox_{\gamma g} + P_V. 
\end{align*} 
The roots of $S := (N+1)^{-1}\sum_{i=1}^{N+1} S_i$ are not solutions of~\eqref{eq:linSAGA}, but in general, 
\begin{align*}
x\in \zer(S) \implies \prox_{\gamma g} (x) \text{ solves \eqref{eq:linSAGA}}, 
\end{align*}
and $\zer(S) \neq \emptyset$ if, and only if, $\zer(\partial g + N^{-1}\sum_{i=1}^N \nabla f_i  + N_V) \neq \emptyset$.

The following iterative algorithm is a special case of SMART. 

\begin{mdframed}
\noindent\begin{algo}[LinSAGA/LinSVRG]\label{alg:LinSAGA}
Choose initial points $x^0, y_1^0, \ldots, y_{N+1}^0 \in \cH$. Choose stepsizes satisfying
\begin{align*}
\lambda  \leq \frac{(N+1)}{8} && \text{and} && \gamma = \frac{N}{\sum_{i=1}^N \hat{L}_i}.
\end{align*}
Then for $k \in \NN$, perform the following four steps:
\begin{enumerate}
\item \textbf{Sampling:} Choose dual update decision $\epsilon_k \in \{0,1\}$. Choose an index $i_k \in \{1, \ldots, N+1\}$ with distribution 
$$
P(i_k = i) = \begin{cases}
\frac{1}{2} &\text{if $i_k = N+1$}\\
\frac{\hat{L}_i}{2\sum_{i=1}^N\hat{L}_i}& \text{otherwise.}
 \end{cases}
 $$
\item \textbf{Primal update (gradient case):} if $i_k < N+1$, set
\begin{align*}
x^{k+1} &= x^k - \lambda\left(\frac{\gamma}{N(N+1)p_{i_k1}}P_V \nabla f_i (P_V\prox_{\gamma g}(x^k)) - \frac{1}{(N+1)p_{i_k1}}y_{i_k}^{k} + \frac{1}{N+1}\sum_{i=1}^{N+1} y_i^k\right).
\end{align*}
 \item \textbf{Primal update (proximal case):} if $i_k = N+1$, set
\begin{align*}
x^{k+1} &= x^k - \lambda\left(\frac{1}{(N+1)p_{i_k1}}\left((I - 2P_V)\prox_{\gamma g}(x^k) + P_Vx^k\right) - \frac{1}{(N+1)p_{i_k1}}y_{i_k}^{k} + \frac{1}{N+1}\sum_{i=1}^{N+1} y_i^k\right).
\end{align*}
\item \textbf{Dual update:} set
\begin{align*}
y_i^{k+1} &= \begin{cases}
y_i^k + \epsilon_k\left(\frac{\gamma}{N}P_V \nabla f_i (P_V\prox_{\gamma g}(x^k)) - y_i^k\right) & \text{if $i_k$ triggers $i$ and $i < N+1$.}\\
y_{N+1}^k + \epsilon_k\left((I - 2P_V)\prox_{\gamma g}(x^k) - y_{N+1}^k\right) & \text{if $i_k$ triggers $i$ and $i=N+1$.}
\end{cases}\qquad \qed
\end{align*}
\end{enumerate}
\end{algo}
\end{mdframed}

\paragraph{~~} As in Section~\ref{sec:SAGA}, the difference between LinSAGA and LinSVRG lies in the trigger graph, and how often the dual variables are updated: LinSAGA uses the directed star trigger graph in which every node $i < N+1$ connects to $N+1$; and LinSVRG  uses the completely connected trigger graph with dual variable update frequency $ P(\epsilon_k = 1) = \tau^{-1} $ for some $\tau > 0$.

\paragraph{The Properties of $S$.}
The operator $S$ satisfies the coherence condition~\eqref{eq:coherence} with
\begin{align*}
\left(\forall i \leq N \right) \qquad \beta_{i1} =  \frac{N}{2\gamma \hat{L}_i (N+1)} && \text{and} && \beta_{(N+1)1} =  \frac{1}{N+1}\left(1 - \frac{1}{2N} \sum_{i=1}^N\gamma \hat{L}_i\right),
\end{align*}
and with
\begin{align*}
\mu = \frac{1}{N+1}\left(1 - \left(\frac{1}{(1+(\gamma L_g)^{-1})} + \frac{\sqrt{1- 2\gamma\mu_f +  \gamma^2 L\mu_f}}{(1+\gamma \mu_g)}\right)\right)
\end{align*}
the operator $S$ is essentially strongly quasi-monotone (if $\gamma \leq 2L^{-1}$ and $\mu > 0$), where $L = N^{-1} \sum_{i=1}^N \hat{L}_i$, the function $N^{-1}\sum_{i=1}^N f_i$ is $\mu_f$-strongly convex, the function $g$ is differentiable and $\mu_g$-strongly convex, and the gradient $\nabla g$ is $L_g$-Lipschitz continuous.

\subsection{SuperSAGA and SuperSVRG}

SuperSAGA and SuperSVRG add a few nonsmooth terms to the smooth problem~\eqref{eq:simplesmooth}:
\begin{align*}
\Min_{x \in \cH_1} \sum_{j=1}^M g_j(z) + \frac{1}{N} \sum_{i=1}^{N} f_i(z), \numberthis\label{eq:super_SAGA_1}
\end{align*}
where the functions $g_j : \cH_1 \rightarrow (-\infty, \infty]$ are closed, proper, and convex; and the functions $f_i:\cH \rightarrow (-\infty, \infty)$ are differentiable and the gradients $\nabla f_i $ are $L_i$-Lipschitz continuous.

It is more convenient for us to work with the reformulated problem: 
\begin{align*}
\Min_{z\in \cH} \; & g(x) + \frac{1}{N} \sum_{i=1}^{N} f_i(x_1) \\
\text{subject to:} \; & x \in D \numberthis\label{eq:super_SAGA_2}
\end{align*}
where we define the spaces $\cH := \cH_1^M$ and $D := \{ x \in \cH \mid x_1 = \cdots = x_M\}$; and for all $x \in \cH$, we let $g(x) := \sum_{j=1}^M g(x_i)$. 

Problem~\eqref{eq:super_SAGA_2} is evidently a special case of Problem~\eqref{eq:linSAGA}, so we choose the operators that worked well there: 
\begin{align*}
\left(\forall i < N +1\right) \qquad& S_i := \frac{\gamma}{N} P_D \circ \nabla f_i  \circ P_D \circ \prox_{\gamma g};\\
&S_{N+1} := (I- 2P_D)\circ \prox_{\gamma g} + P_D. 
\end{align*} 
Provided that $M$ is relatively small, and for each $j$, the operator $\prox_{\gamma g_j}$ is easy to evaluate, the operators $S_i$ are easy to evaluate. For example, for all $x \in \cH$ and $i < N+1$, we have
\begin{align*}
 \qquad P_Dx = \frac{1}{M}\sum_{j=1}^M x_j; &&  \prox_{\gamma g}(x) = (\prox_{\gamma g_j}(x_j))_{j = 1}^M; && S_i(x) = \left(\frac{\gamma}{NM} \nabla f_i\left(\frac{1}{M} \sum_{l=1}^M\prox_{\gamma g_l}(x_l)\right)\right)_{j= 1}^M.
\end{align*}
And as before, the roots of $S := (N+1)^{-1}\sum_{i=1}^{N+1} S_i$ are not solutions of~\eqref{eq:super_SAGA_1}, but in general, 
\begin{align*}
x^\ast\in \zer(S) \implies \left(\forall j \right) \;  \prox_{\gamma g_j} (x_j^\ast) \text{ solves \eqref{eq:super_SAGA_1}}, 
\end{align*}
and $\zer(S) \neq \emptyset$ if, and only if, $\zer(\sum_{j=1}^M\partial g_j + N^{-1}\sum_{i=1}^N \nabla f_i) \neq \emptyset$.

The following iterative algorithm is a special case of SMART. 

\begin{mdframed}
\noindent\begin{algo}[SuperSAGA/SuperSVRG]\label{alg:superSAGA}
Choose initial points $x^0 \in \cH$ and $\overline{y}_1^0, \ldots, \overline{y}_{N+1}^0 \in \cH_1$. Choose stepsizes satisfying
\begin{align*}
\lambda  \leq \frac{(N+1)}{8} && \text{and} && \gamma = \frac{MN}{\sum_{i=1}^N L_i}.
\end{align*}
Then for $k \in \NN$, perform the following four steps:
\begin{enumerate}
\item \textbf{Sampling:}  Choose dual update decision $\epsilon_k \in \{0,1\}$. Choose an index $i_k \in \{1, \ldots, N+1\}$ with distribution 
$$
P(i_k = i) = \begin{cases}
\frac{1}{2} &\text{if $i = N+1$}\\
\frac{L_i}{2\sum_{i=1}^NL_i}& \text{otherwise.}
 \end{cases}
 $$
\item \textbf{Primal update (gradient case):} if $i_k < N+1$, set\begin{align*}
\left( \forall j\right) \hspace{10pt}& x_j^{k+1} = x_j^k - \lambda \left(\frac{\gamma}{MN(N+1)p_{i_k1}}\nabla f_i\left(\frac{1}{M}\sum_{j=1}^Mw_j^k\right)- \frac{1}{(N+1)p_{i_k1}}\overline{y}_{i_k}^k + \frac{1}{N+1} \sum_{i=1}^k \overline{y}_{i}^k \right); \\
&w_j^{k+1} = \prox_{\gamma g_j}(x^{k+1}_j).
\end{align*}
\item \textbf{Primal update (proximal case):} if  $i_k = N+1$, set
\begin{align*}
\left( \forall j \right) \hspace{10pt}& x_j^{k+1} = x_j^k - \lambda\left(\frac{1}{(N+1)p_{i_k1}}\left(w_j^k - \frac{1}{M}\sum_{j=1}^M (2w_j^k - x_j^k)\right) - \frac{1}{(N+1)p_{i_k1}}\overline{y}_{i_k}^k +\frac{1}{N+1} \sum_{i=1}^k \overline{y}_i^k \right); \\
&w_j^{k+1} = \prox_{\gamma g_j}(x^{k+1}_j).
\end{align*}
\item \textbf{Dual update:} set
\begin{align*}
\overline{y}_i^{k+1} &= \begin{cases}
\overline{y}_i^k + \epsilon_k\left(\frac{\gamma}{MN}\nabla f_i\left(\frac{1}{M}\sum_{j=1}^Mw_j^k\right) - \overline{y}_i^k\right) &\text{if $i_k$ triggers $i$ and $i < N+1$;}\\
\overline{y}_{i}^k + \epsilon_k\left(\frac{1}{M}\sum_{j=1}^M (x_j^k - w_j^k) - \overline{y}_{i}^k\right) &\text{if $i_k$ triggers $i$ and $i = N+1$.} 
\end{cases}\qquad \qed
\end{align*}
\end{enumerate}
\end{algo}
\end{mdframed}

\paragraph{~~} The difference between SuperSAGA and SuperSVRG, again, lies in the trigger graph and the dual variable update frequency; see the comments following Algorithm~\ref{alg:LinSAGA}.

\paragraph{The Properties of $S$.}
The operator $S$ is an instance of the LinSAGA/LinSVRG operator defined in Section~\ref{sec:LinSAGA}, so the two operators satisfy the coherence condition~\eqref{eq:coherence} with the same constants\footnote{Clearly, $\hat{L}_i = L_iM^{-1}$.}  
\begin{align*}
\left(\forall i \leq N \right) \qquad \beta_{i1} =  \frac{NM}{2\gamma L_i (N+1)} && \text{and} && \beta_{(N+1)1} =  \frac{1}{N+1}\left(1 - \frac{1}{2N} \sum_{i=1}^N \frac{\gamma L_i}{M}\right),
\end{align*}
and both operators are essentially strongly monotone under the same conditions. 

In fact, the SuperSAGA/SuperSVRG algorithms and the LinSAGA/LinSVRG algorithms differ in just one way: for SuperSAGA/SuperSVRG, the dual variables $\overline{y}_i^k$ are vectors in $\cH_0$ rather than vectors $y_i^k$ in $\cH = \cH_0^M$, and that saves some memory. We can use $\overline{y}_i^k$ in place of $y^k_i$ by viewing $S$ as an operator not on the space $\cH^M$, but as an operator on the orthogonal decomposition $\cH = D \oplus D^\perp$, which has just two components. Then, the identity
$$
S_{N+1} = P_{D^\perp} \prox_{\gamma g} + P_D (I - \prox_{\gamma g})
$$
allows us to reparamterize the components of $S_{i}$ into a $D$ component and $D^\perp$ component: for all $i \leq N$,
\begin{align*}
(S_i)_D &= \frac{\gamma}{N} P_D \circ \nabla f_i  \circ P_D \circ \prox_{\gamma g}; && & (S_i(x))_{D^\perp} &\equiv 0; \\
(S_{N+1})_D &= P_D (I - \prox_{\gamma g}); && & (S_{N+1})_{D^\perp} &= P_{D^\perp} \prox_{\gamma g}.
\end{align*}
Then, with this new decomposition, we have $(S_i(x^\ast))_{D^\perp} = 0$ for all $i$, which makes saving the full vector $y_i^k$ superfluous; we need only store the component in $D$, which is precisely the vector $\overline{y}_i^k$. 

The above approach saves memory when all $M$ components of $x$ are updated at every iteration \textit{and} the coordinates are consistent-read asynchronous or totally synchronous. However, if only some of the coordinates $x_j$ are updated at each iteration \textit{or} inconsistent-read updates are performed, Theorem~\ref{thm:newalgoconvergence} ceases to apply; in either of those cases, we must use the full dual variables $y_i^k$.

\subsection{TropicSMART: Randomized Smoothly Coupled Monotropic Programming}\label{sec:TropicSMART}

The TropicSMART problem is different from all problems we have seen so far: 
\begin{align*}
\Min_{x_j \in \cH_j} &\; \sum_{j=1}^M g_j(x_j) + f(x_1, \ldots, x_M); \\
\text{subject to:} &\; \sum_{j=1}^M A_j x_j = b, \numberthis\label{eq:RSCMP}
\end{align*}
where the sets $\cH_j$ ($j = 1, \ldots, M+1$) are Hilbert spaces; the sets $\cH' : = \cH_1 \times \cdots \times \cH_M$ and $\cH := \cH_1 \times \cdots \times \cH_{M+1}$ are product spaces; the functions $g_j : \cH_j \rightarrow (-\infty, \infty]$ are closed, proper, and convex (we also let $g(x) := \sum_{j=1}^M g(x_j)$); the function $f : \cH' \rightarrow (-\infty, \infty)$ is differentiable and $\nabla f$ is $L$-Lipschitz continuous; the linear maps $A_j : \cH_j \rightarrow \cH_{M+1}$ are continuous; and $b \in \cH_{M+1}$.

There is only one TropicSMART operator: for all $x \in \cH$, define
\begin{align*}
(S(x))_j := \begin{cases}
x_j - \prox_{\gamma_j g_j}\left(x_j - \gamma_j  A_j^\ast \left( x_{M+1} + 2\gamma_{M+1} \left(\sum_{l=1}^M A_l x_l - b\right)\right) - \gamma_j \nabla_j f(x)\right) & \text{if $j < M+1$};\\
 -\gamma_{M+1}\left( \sum_{l=1}^M A_l x_l - b\right) &\text{if $j = M+1$.}
\end{cases}
\end{align*}
The roots of $S := (N+1)^{-1}\sum_{i=1}^{N+1} S_i$ are not solutions of~\eqref{eq:RSCMP}, but in general, 
\begin{align*}
x^\ast\in \zer(S) \implies   (x_1^\ast, \ldots, x_M^\ast) \text{ solves \eqref{eq:super_SAGA_1}}, 
\end{align*}
and $\zer(S) \neq \emptyset$ if, and only if, $ \zer\left(\partial g + \nabla f  + N_{\{x \in \cH \mid \sum_{j=1}^M A_j x_j = b\}}\right) \neq \emptyset$.

The following iterative algorithm is a special case of SMART. 

\begin{mdframed}
\noindent\begin{algo}[TropicSMART]\label{alg:RSCMP}
Choose initial points $x^0 \in \cH$. Choose $\delta \in (0, 1)$, and choose stepsizes satisfying
\begin{align*}
\gamma_{M+1}\left( \sum_{j=1}^M \gamma_j\|A_j\|^2\right)\leq \delta;  && \max_j\{\gamma_j\} \leq \frac{2(1-\sqrt{\delta})}{L}; && \lambda  \leq \frac{L\max_j\{\gamma_j\}}{4(1+\sqrt{\delta})}.
\end{align*}
Then for $k \in \NN$, perform the following two steps:
\begin{enumerate}
\item \textbf{Sampling:} Choose a coordinate $j_k \in \{1, \ldots, M+1\}$ uniformly at random and set $\sfS_k = \{j_k\}$.
\item \textbf{Primal update:} set
\begin{align*}
 \qquad & \overline{x}_{M+1}^{k+1} = x_{M+1}^k + \gamma_{M+1}\left( \sum_{l=1}^M A_l x_l^k - b\right);\\
\left(\forall j \in \sfS_k\backslash \{M+1\}\right) \qquad & \overline{x}_j^{k+1} = \prox_{\gamma_j g_j}\left(x_j^k - \gamma_j  A_j^\ast (2\overline{x}_{M+1}^{k+1} - x_{M+1}^k) - \gamma_j \nabla_j f(x^k)\right); \\\\
\left( \forall j \in \mathsf{S}_k\right)\qquad& x_j^{k+1} = x_j^k - \lambda\left(x_j^k - \overline{x}_j^{k+1}\right); \\
\left( \forall j \notin \mathsf{S}_k\right) \qquad&
x_j^{k+1} = x_j^k. \qquad \qed
\end{align*}
\end{enumerate}
\end{algo}
\end{mdframed}

\paragraph{The Properties of $S$.}
The operator $S$ satisfies the coherence condition~\eqref{eq:coherence} with
\begin{align*}
\left(\forall j \right) \qquad \beta_{1j} =  \frac{L\max_j\{\gamma_j\}}{4\gamma_j}.
\end{align*}
But the TropicSMART operator $S$ does not satisfy the coherence condition in the standard metric on $\cH$; instead there is a strongly positive self-adjoint linear operator $P$ (e.g., with respect to $\dotp{\cdot, \cdot}_\pr$, a symmetric positive definite matrix) so that
\begin{align*}
\left( \forall x \in \cH\right), \left( \forall x^\ast \in \zer(S)\right)  \qquad \dotp{ S(x), x - x^\ast}_P \geq \sum_{j=1}^{M+1} \beta_{1j}\|(S(x))_j\|^2_j,
\end{align*}
and this linear operator $P$ satisfies $\sum_{j=1}^{M+1} \underline{M}_j\|x_j\|^2_j \leq \|x\|_P^2 \leq   \sum_{j=1}^{M+1} \overline{M}_j\|x_j\|^2_j$,
where for all $j $, we have 
\begin{align*}
\underline{M}_j := \frac{1-\sqrt{\delta}}{\gamma_j} && \text{and} && \overline{M}_j := \frac{1+\sqrt{\delta}}{\gamma_j}.
\end{align*}

The weakest conditions under which $S$ is essentially strongly quasi-monotone are not known.

\subsection{ProxSMART: Randomized Proximable Optimization}\label{sec:ProxSMART}

The ProxSMART algorithm solves the following proximable optimization problem:
\begin{align*}
\Min_{z\in \cH_1} \; g_1(z) + \sum_{j=2}^{M} g_j(A_jz), \numberthis\label{eq:RNPD}
\end{align*}
where the sets $\cH_j$ ($j = 1, \ldots, M+1$) are Hilbert spaces; the set $\cH:= \cH_1 \times \cdots \times \cH_M$ is a product space; the functions $g_j : \cH_j \rightarrow (-\infty, \infty]$ are closed, proper, convex, and nonsmooth; and the linear maps $A_j :\cH_1 \rightarrow \cH_j$ are continuous. 

There is only one ProxSMART operator: for all $x \in \cH$, define
\begin{align*}
\left(S(x) \right)_j
:= \left(S_1(x) \right)_j= 
\begin{cases}
x_1 - \prox_{\gamma_1 g_1}\left(x_1 - \gamma_1\sum_{j=2}^M A_j^\ast x_j \right)  &\text{if $j = 1$;} \\
x_j - \prox_{\gamma_j g_j^\ast}\left(x_j + \gamma_j A_j\left(2\overline{x}_1 - x_1\right)\right) &\text{otherwise;}
\end{cases}
\end{align*}
where $\overline{x}_1 = \prox_{\gamma_1 g_1}\left(x_1 - \gamma_1\sum_{j=2}^M A_j^\ast x_j \right)$. The roots of $S := (N+1)^{-1}\sum_{i=1}^{N+1} S_i$ are not solutions of~\eqref{eq:RNPD}, but in general, \vspace{-5pt}
\begin{align*}
x^\ast\in \zer(S) \implies   x_1 \text{ solves \eqref{eq:RNPD}}, 
\end{align*}
and $\zer(S) \neq \emptyset$ if, and only if, $\zer(\partial g_1(x) + \sum_{j=2}^M A_j^\ast\partial g_j\circ A_j) \neq \emptyset$.

The following iterative algorithm is a special case of SMART. 

\begin{mdframed}
\noindent\begin{algo}[ProxSMART]\label{alg:ProxSMART}
Choose initial point $x^0 \in \cH$. Choose $\delta \in (0, 1)$, and choose stepsizes satisfying
\begin{align*}
\gamma_{1}\left( \sum_{j=2}^M \gamma_j\|A_j\|^2\right)\leq \delta  && \text{and} && \lambda  \leq \frac{1-\sqrt{\delta}}{1+\sqrt{\delta}}.
\end{align*}
Then for $k \in \NN$, perform the following two steps:
\begin{enumerate}
\item \textbf{Sampling.} Choose a set of functions (their indices) $\mathsf{S_k}$.
\item \textbf{Primal update}: set 
\begin{align*}
\qquad&\overline{x}_1^{k+1} = \prox_{\gamma_1 g_1}\left(x_1^k - \gamma_1\sum_{j=2}^M A_j^\ast x_j^k\right); \\
\left( \forall j \in \mathsf{S}_k\backslash \{1\}\right)\qquad & \overline{x}_j^{k+1} = \prox_{\gamma_j g_j^\ast} \left( x_j^k + \gamma_jA_j\left(2\overline{x}_1^{k+1} - x_1^k\right)\right); \\\\
\left( \forall j \in \mathsf{S}_k\right)\qquad& x_j^{k+1} = x_j^k - \frac{\lambda}{q_j M }\left(x_j^k - \overline{x}_j^{k+1}\right); \\
\left( \forall j \notin \mathsf{S}_k\right) \qquad&
x_j^{k+1} = x_j^k. \qquad \qed
\end{align*}
\end{enumerate}
\end{algo}
\end{mdframed}

The variables $x_j$ for which $j > 1$ are different from the dual variables in SMART because here $n = 1$. However, the $x_j$ dual variables play a similar role to the $y$ dual variables, and like the dual variables in SMART, they are often low dimensional---if $A_i$ is a row vector, then $x_j$ is a scalar. Another similarity, arising in the case $\lambda = 1$, is that for $j > 1$, $x_j^k$ is a subgradient, much like the dual variables in SMART are gradients in SAGA and SVRG.\footnote{If $z^+ = \prox_{\gamma_j g^\ast}(z)$, then $z^+ \in \partial g(\gamma^{-1}(z - z^+))$} But even if $A_i$ is the identity map, and $\lambda$ is arbitrary, $x_j^k$ is always in the linear span of $\partial f$, which may be low dimensional.

\paragraph{The Properties of $S$.} The operator $S$ satisfies the coherence condition~\eqref{eq:coherence} with
\begin{align*}
\left(\forall j \right) \qquad \beta_{1j} =  \frac{(1-\sqrt{\delta})}{\gamma_j}.
\end{align*}
But  like the TropicSMART operator from Section~\ref{sec:TropicSMART}, the ProxSMART operator $S$ does not satisfy the coherence condition in the standard metric on $\cH$; instead there is a strongly positive self-adjoint linear operator $P$ so that\vspace{-5pt}
\begin{align*}
\left( \forall x \in \cH\right), \left( \forall x^\ast \in \zer(S)\right)  \qquad \dotp{ S(x), x - x^\ast}_P \geq \sum_{j=1}^M \beta_{1j}\|(S(x))_j\|^2_j,
\end{align*}
and this linear operator $P$ satisfies $\sum_{j=1}^{M} \underline{M}_j\|x_j\|^2_j \leq \|x\|_P^2 \leq   \sum_{j=1}^{M} \overline{M}_j\|x_j\|^2_j$,
where for all $j > 0$, we have 
\begin{align*}
\underline{M}_j := \frac{1-\sqrt{\delta}}{\gamma_j} && \text{and} && \overline{M}_j := \frac{1+\sqrt{\delta}}{\gamma_j}.
\end{align*}

essentially strongly quasi-monotone, provided that $g_1$ is $\mu_1$-strongly convex and each function $g_j$, $j> 1$, is differentiable, the gradient $\nabla g_j$ is $L_j$-Lipschitz continuous, and the constants $\gamma_j$ are chosen small enough; we omit the proof of this fact. The weakest conditions under which $S$ is essentially strongly quasi-monotone are not known.

\vspace{-10pt}
\subsection{ProxSMART+: Randomized Composite Optimization}\label{sec:ProxSMART+}

The ProxSMART+ algorithm solves the following composite optimization problem:
\begin{align*}
\Min_{z\in \cH_1} \; \sum_{j=2}^{M} g_j(A_jz) + \frac{1}{N}\sum_{i=1}^{N} f_i(z).\numberthis\label{eq:RCPD}
\end{align*}
where the sets $\cH_j$ ($j = 1, \ldots, M+1$) are Hilbert spaces; the set $\cH:= \cH_1 \times \cdots \times \cH_M$ is a product space; the functions $g_j : \cH_j \rightarrow (-\infty, \infty]$ are closed, proper, and convex; the functions $f_i : \cH_1 \rightarrow (-\infty, \infty)$ are differentiable and $\nabla f_i$ is $L_i$-Lipschitz continuous; the maps $A_j :\cH_1 \rightarrow \cH_j$ are continuous linear maps.

We model the ProxSMART+ problem with $N+1$ operators: for all $x \in \cH$, define
\begin{align*}
\left(\forall i < N+1\right)\qquad \left(S_i(x)\right)_j 
&:= \begin{cases}
\frac{\gamma_1}{N}\nabla f_i\left(x_1 - 2\gamma_1\sum_{j=2}^M A_j^\ast x_j\right) &\text{if $j = 1$;}\\
0 &\text{otherwise.}
\end{cases}\\
\left(S_{N+1}(x) \right)_j
&:= 
\begin{cases}
 \gamma_1\sum_{j=2}^M A_j^\ast x_j  &\text{if $j = 1$;} \\
x_j - \prox_{\gamma_j g_j^\ast}\left(x_j + \gamma_j A_j\left(x_1 - 2\gamma_1\sum_{j=2}^M A_j^\ast x_j\right)\right) &\text{otherwise.}
\end{cases}
\end{align*}
The roots of $S := (N+1)^{-1}\sum_{i=1}^{N+1} S_i$ are not solutions of~\eqref{eq:RCPD}, but in general, \vspace{-10pt}
\begin{align*}
x^\ast\in \zer(S) \implies  x_1^\ast - 2\gamma_1\sum_{j=2}^M A_j^\ast x_j^\ast  \text{ solves \eqref{eq:RCPD}}
\end{align*}
and $\zer(S) \neq \emptyset$ if, and only if, $\zer(\sum_{j=2}^M A_j^\ast\partial g_j\circ A_j + N^{-1}\sum_{i=1}^N \nabla f_i) \neq \emptyset$.

The following iterative algorithm is a special case of SMART. 

\begin{mdframed}
\noindent\begin{algo}[ProxSMART+]\label{alg:ProxSMART+}
Choose initial points $x^0 \in \cH$ and $y_{11}, \ldots, y_{(N+1)1} \in\cH_1$. Choose $\delta \in (0, 1)$, and choose stepsizes satisfying
\begin{align*}
\gamma_{1}\left( \sum_{j=2}^M \gamma_j\|A_j\|^2 + \frac{1}{2N} \sum_{i=1}^N L_i\right)\leq \delta  && \text{and} && \lambda  \leq \frac{(N+1)(1-\sqrt{\delta})M}{2 \left(1 + \sqrt{\delta} -  \frac{1}{2\sqrt{\delta}N} \sum_{i=1}^N L_i\right)\left(\frac{M-1}{\frac{\gamma_1}{N} \sum_{i=1}^{N} L_i} + 1\right)}.
\end{align*}
Then for $k \in \NN$, perform the following four steps:
\begin{enumerate}
\item \textbf{Sampling.} Choose dual update decision $\epsilon_k \in \{0,1\}$. Choose $j_k \in \{1, \ldots, m\}$ with distribution ($j > 1$)
\begin{align*}
q_1  = P(j_k = 1)  = \frac{1}{\frac{M-1}{\frac{\gamma_1}{N} \sum_{i=1}^{N} L_i} + 1}, &&  && q_j = P(j_k = j) = \frac{1 - q_{1}}{M-1}, 
\end{align*}
and set $\sfS_k = \{j_k\}$. Also, given $j_k$, choose $i_k \in \{1, \ldots, N+1\}$ with distribution ($i < N+1$ and $j > 1)$
\begin{align*}
 p_{i1} &= \frac{\gamma_1 L_i}{N\left( \frac{\gamma_1}{N} \sum_{i=1}^N L_i + 1\right)};    &&& p_{(N+1)1} & =1 - \sum_{i=1}^N p_{i1}; \\
p_{ij} &= 0; &&& p_{(N+1)j} &= 1. 
\end{align*}
\item \textbf{Primal update (gradient case):} if $i_k < N+1$ and $\hat{x}_1^k = x_1^k - 2\gamma_1\sum_{j=2}^M A_j^\ast x_j^k$, set
\begin{align*}
&\overline{x}_1^{k+1} = x_1^k - \frac{\lambda}{q_1M}\left(\frac{\gamma_1}{N(N+1) p_{i_k1}}\nabla f_{i_k}\left(\hat{x}_1^k\right) - \frac{1}{(N+1) p_{i_k1}} y_{i_k, 1}^k + \frac{1}{N+1} \sum_{i=1}^{N+1} y_{i, 1}^k  \right); \\
\left( \forall j > 1\right) \qquad&
x_j^{k+1} = x_j^k.
\end{align*}
\item \textbf{Primal update (proximal case):} if $i_k = N+1$ and $\hat{x}_1^k = x_1^k - 2\gamma_1\sum_{j=2}^M A_j^\ast x_j^k$, set 
\begin{align*}
\qquad&\overline{x}_1^{k+1} = x_1^k - \left(\frac{\gamma_1}{(N+1) p_{i_k1}}\sum_{j=2}^M A_j^\ast x_j^k - \frac{1}{(N+1) p_{i_k1}} y_{i_k, 1}^k + \frac{1}{N+1} \sum_{i=1}^{N+1} y_{i, 1}^k  \right); \\
\left( \forall j \in \mathsf{S}_k\backslash \{1\}\right)\qquad & \overline{x}_j^{k+1} = \prox_{\gamma_j g_j^\ast} \left( x_j^k + \gamma_jA_j\hat{x}_1^k\right); \\\\
\left( \forall j \in \mathsf{S}_k\right)\qquad & x_j^{k+1} = x_j^k - \frac{\lambda}{q_jM}\left(x_j^k - \overline{x}_j^{k+1}\right); \\
\left( \forall j \notin \mathsf{S}_k\right) \qquad&
x_j^{k+1} = x_j^k.
\end{align*}
\item \textbf{Dual update:} if $\hat{x}_1^k = x_1^k - 2\gamma_1\sum_{j=2}^M A_j^\ast x_j^k$, set
\begin{align*}
y_{i, 1}^{k+1} &= \begin{cases}
y_{i, 1}^k + \epsilon_k\left(\frac{\gamma_1}{N}\nabla f_i\left(\hat{x}_1^k\right) - y_{i, 1}^k\right) & \text{if $i_k$ triggers $i$ and $i <  N+1$;}\\
y_{i, 1}^k + \epsilon_k\left(\gamma_1\sum_{j=2}^M A_j^\ast x_j^k - y_{i, 1}^k\right) &\text{if $i_k$ triggers $i$ and $i= N+1$.}
\end{cases} \qquad \qed
\end{align*} 
\end{enumerate}
\end{algo}
\end{mdframed}

As in the ProxSMART+ algorithm, the dual variables $x_j$ with $j > 1$, are subgradients of $g_j$, and they often live in low-dimensional spaces; see the comments immediately following Algorithm~\ref{alg:ProxSMART+}. 

We save a bit of memory on the $y_i$ dual variables, too,  because for all $j > 1$ and for all $i$, $\mathbf{S}^\ast_{ij} = 0$; thus, we only maintain the first component of these dual variables. 

\paragraph{The Properties of $S$.}
The operator $S$ satisfies the coherence condition~\eqref{eq:coherence} with 
\begin{align*}
\left(\forall 1\leq  i < N+1\right) \qquad \beta_{i1}  = \frac{N(1-\sqrt{\delta})}{2(N+1)\gamma_1^2 L_i}  &&\text{and} && \left(\forall j\right) \qquad\beta_{(N+1)j} = \frac{1-\sqrt{\delta}}{\gamma_1}.
\end{align*}
But like the TropicSMART and ProxSMART operators from Sections~\ref{sec:TropicSMART} and~\ref{sec:ProxSMART}, the ProxSMART+ operator $S$ does not satisfy the coherence condition in the standard metric on $\cH$; instead there is a strongly positive self-adjoint linear operator $P$ so that\begin{align*}
\left( \forall x \in \cH\right), \left( \forall x^\ast \in \zer(S)\right)  \qquad \dotp{ S(x), x - x^\ast}_P \geq \sum_{i=1}^N\sum_{j=1}^M \beta_{ij}\|(S_i(x))_j - (S_i(x^\ast))_j\|^2_j,
\end{align*}
and this linear operator $P$ satisfies $\sum_{j=1}^{M} \underline{M}_j\|x_j\|^2_j \leq \|x\|_P^2 \leq   \sum_{j=1}^{M} \overline{M}_j\|x_j\|^2_j$
where for all $j > 0$, we have 
\begin{align*}
\underline{M}_j := \frac{1-\sqrt{\delta}}{\gamma_j} && \text{and} && \overline{M}_j := \frac{1+\sqrt{\delta}}{\gamma_j}.
\end{align*}

The operator $S$ is essentially strongly quasi-monotone, provided that $N^{-1}\sum_{i=1}^N f_i$ is $\mu_1$-strongly convex, each function $g_j$, $j> 1$, is differentiable, each gradient $\nabla g_j$ is $\overline{L}_j$-Lipschitz continuous, and the constants $\gamma_j$ are chosen small enough; we omit the proof of this fact. The weakest conditions under which $S$ is essentially strongly quasi-monotone are not known.


\subsection{Monotone Inclusions and Saddle-Point Problems}\label{sec:general_monotone_inclusion}

SMART solves the monotone inclusion problem (see~\cite{bauschke2011convex})
\begin{align}
\text{Find $x \in \cH$ such that $0 \in Ax + \frac{1}{N}\sum_{i=1}^N B_ix$,} \label{eq:general_monotone_inclusion}
\end{align}
where $A : \cH \rightarrow 2^\cH$ is a monotone operator and $B_i : \cH \rightarrow \cH$ are $L_i$-Lipschitz continuous, as long as  
\begin{enumerate}[label=(\alph*)]
\item \label{cond:mono_operator_cocoercive} each operator $B_i$ is $L_i^{-1}$-cocoercive; or
\item \label{cond:mono_operator_Lipschitz} the operator $B := N^{-1}\sum_{i=1}^n B_i$ is monotone and $A$ is $\mu_A$-strongly monotone.
\end{enumerate}
SMART will converge (weakly) under condition~\ref{cond:mono_operator_cocoercive} and will converge linearly under condition~\ref{cond:mono_operator_Lipschitz}. A straightforward modification of the results in Section~\ref{sec:proximable_SAGA} deals with case~\ref{cond:mono_operator_cocoercive}, so here we focus on case~\ref{cond:mono_operator_Lipschitz}.

To solve~\eqref{eq:general_monotone_inclusion}, we create $N+1$ SMART operators: ($J_{\gamma A} = (I+ \gamma A)^{-1}$)
\begin{align*}
\left(\forall i < N +1\right) \qquad& S_{i} = \frac{\gamma}{N} B_i\circ J_{\gamma A} && \text{and}&& S_{N+1} = (I-J_{\gamma A}), \numberthis\label{eq:SMART_mono_operators}
\end{align*}
The operator $S = (N+1)^{-1}\sum_{i=1}^{N+1} S_i$ satisfies all the conditions needed for us to invoke SMART (see Proposition~\eqref{prop:MONOopprop}): solutions of~\eqref{eq:general_monotone_inclusion} are obtainable from zeros of $S$ via the resolvent mapping $J_{\gamma A}(\zer(S)) = \zer(A + B)$;  the operator $S$ is $\mu$-essentially strongly quasi-monotone with 
\begin{align}\label{eq:mu_saddle}
 \mu := \frac{(1+\gamma \mu_A)^2 - (1 + \gamma^2 (\overline{L})^2)}{(N+1)(1 + \gamma \mu_A)^2} \qquad\qquad \text{where }\overline {L} = \frac{1}{N}\sum_{i=1}^N L_i;
\end{align}
the operators $S_i$ satisfy the coherence condition~\eqref{eq:coherence} with particular constants  
\begin{align*}
\beta_{i1} :=   \frac{\mu N^2}{L_i^2 \gamma^2(N+1)} \quad \left(\forall i \leq N\right) && \text{and} && \beta_{(N+1)1} := \frac{\mu}{(N+1)}; \numberthis\label{eq:mono_op_conditions}
\end{align*}
and the operator $S$ is demi-closed at $0$. 
Of course, we should choose $\gamma$ small enough that $\beta_{i1}$ and $\mu$ are positive.\footnote{If $\overline{L} \leq \mu$, then $\gamma$ is arbitrary; if $\overline{L} > \mu$, then choose $\gamma < 2\mu ((\overline{L})^2 - \mu_A^2)^{-1}$.}

The following iterative algorithm is a special case of SMART. 
\begin{mdframed}
\noindent\begin{algo}[SMART for Monotone Inclusions]\label{alg:monoSMART}
Fix $\lambda  \in \RR_{++}$ according to Table~\ref{table:stepsizeconnections}. Let $n = N+1$. Choose $x^0\in \cH$ and $y_1^0, \ldots, y_{n}^0 \in \cH$. Then for $k \in \NN$, perform the following three steps: 
\begin{enumerate}
\item \textbf{Sampling:} choose an operator index $i_k$ and a dual update decision variable $\epsilon_k$.
\item \textbf{Primal update:} if $i_k < n$, set
\begin{align*}
&x^{k+1} = x^k - \frac{\lambda}{n}\left(\frac{\gamma}{p_{i_k1}N}B_{i_k}( J_{\gamma A}(x^k)) - \frac{1}{p_{i_k1}}y_{i_k}^k + \sum_{i=1}^{n} y_i^k\right); 
\end{align*}
otherwise, set
\begin{align*}
& x^{k+1} = x^k - \frac{\lambda}{n} \left(\frac{1}{p_{n1}}(I-J_{\gamma A})(x^k) - \frac{1}{p_{n1}}y_{n}^k + \sum_{i=1}^{n} y_i^k\right).
\end{align*}
\item \textbf{Dual update:} set 
\begin{align*}
\qquad & y_{n}^{k+1} = y_{n}^k + \epsilon_k\left((I-J_{\gamma A})(x^k) - y_{N+1}^k\right); \\
\text{if $i_k$ triggers $i < n$} \qquad & y_{i}^{k+1}  = y_{i}^k + \epsilon_k \left(\frac{\gamma}{N}B_{i}( J_{\gamma A}(x^k)) - y_{i}^k\right).
\end{align*}
\end{enumerate}
\end{algo}
\end{mdframed}

For this problem, proper trigger graph selection and importance sampling (i.e., nonuniform $p_{ij}$) are essential for getting good performance with SMART. No matter which operator $S_i$ is selected at iteration $k$, we must compute the resolvent $J_{\gamma A}$, and we can then use that resolvent value to update the dual variable $y_{N+1}$ corresponding to $S_{N+1}$. So the trigger graph $G$ that we choose will always include a star subgraph, with directed edges emanating from each $i = 1, \ldots, N+1$ and ending at the node $N+1$. This subgraph requirement acknowledges the nonuniform structure of $S_1, \ldots, S_{N+1}$ by refusing to throw away useful information. Their nonuniform structure can also be addressed by sampling $i_k$ from a nonuniform distribution, for example, by sampling each operator with a probability proportional to the inverse of $\beta_{i1}$: $P(i_k = i) \propto \beta_{i1}^{-1}.$ Nonuniform sampling not only improves rates of convergence, but also maximizes the range of allowable step sizes $\lambda$ in~\eqref{eq:weaklambdabound} and~\eqref{eq:asynclambdalinear}.

With such constraints imposed on the parameters of SMART, we obtain the update rules in Algorithm~\ref{alg:monoSMART}.\footnote{Use $n = N+1, m = 1, d_{k} \equiv 0$, $ e^{i}_{k} \equiv 0$, $q_1 \equiv 1$, $\tau_p = \tau_d = 0$, $p_{i1} \propto \beta_{i1}$, $\rho  \in (0, 1)$,  $E \supseteq \{(i, N+1) \mid i \in V\}$, $p_{(N+1)1}^T =1$.} We showed, in Section~\ref{sec:SAGA}, that the difference between SAGA and SVRG lies with the dual update decision variable $\epsilon_k \in \{0, 1\}$ and the trigger graph $G$; the same tricks can be played here, resulting in variants of SAGA and SVRG for solving the monotone inclusion problem~\eqref{eq:general_monotone_inclusion}. Further asynchronous or coordinate-update variants of these methods readily follow, too: as soon as we define the operators~\eqref{eq:SMART_mono_operators}, the full power of SMART is available to us.

Convex-concave saddle-point problems
\begin{align*}
\min_{w \in \cG_1} \max_{z \in \cG_2} \left\{M(w, z) + K(w, z)\right\},
\end{align*}
where $\cG_1, \cG_2$ are Hilbert spaces and $M, K : \cG_1 \times \cG_2 \rightarrow (-\infty, \infty]$ are convex in $w \in \cG_1$, concave in $z \in \cG_2$, fall under~\eqref{eq:SMART_mono_operators} through through the assignments $\cH := \cG_1 \times \cG_2$ and,  for all $x = (w, z) \in \cH$, 
\begin{align*}
 Ax := \left(\partial_w M(w, z), \partial_z (- M) (w, z)\right) && \text{and} &&  Bx := \left(\partial_w K(w, z), \partial_z(-K)(w,z)\right),
\end{align*}
as long as $B$ splits into the sum of Lipschitz operators $B_i$. With this notation, $A$ is strongly monotone whenever $M$ is strongly convex in $w$ and strongly concave in $z$. (See~\cite{rockafellar1970monotone} for precise conditions which ensure the maximal monotonicity of $A$ and $B$.)

Among the many possible Saddle-Point problems, the following one, which falls under assumption~\ref{cond:mono_operator_Lipschitz}, deserves special attention:
\begin{align*}
M(w, z) = g_1(w) - g_2(z) && \text{and} &&  K(w, z) = \frac{\dotp{Lw, z} + \sum_{i=1}^{N-1} (f_i(w) - h_i(z))}{N},
\end{align*} 
where the functions $g_1 : \cG_1 \rightarrow (-\infty, \infty]$ and $g_2 : \cG_2 \rightarrow (-\infty, \infty]$ are proper, closed, convex, the linear map $L : \cG_1 \rightarrow \cG_2$ is bounded, and the functions $f_i : \cG_1 \rightarrow (-\infty, \infty)$ and $h_i : \cG_2 \rightarrow (-\infty, \infty)$ are differentiable, convex, and have Lipschitz continuous gradients. Finer splittings of $K(w, z)$ and, hence, $B$ are available by splitting the linear map $L$ into finer, additive pieces. In this case 
$$
\left( \forall (w, z) \in \cH\right)\qquad J_{\gamma A}(w, z) = (\prox_{\gamma g_1}(w), \prox_{\gamma g_2}(z)),
$$
where $\prox_{\gamma g_i}(v) = \argmin_{v' \in \cG_i} \{g_i(v') + (2\gamma)^{-1}\|v' - v\|^2\}$, whereas multiplications by $B_i$ correspond to multiplications by $L$ and $L^\ast$ or evaluations of the gradient mappings $\nabla f_i$ and $\nabla h_i$.

\newpage\section{Proof in the Synchronous Case}\label{sec:nonasync}

The proof that SMART converges is illuminating when the algorithm is totally synchronous, i.e., when $d_k \equiv 0$ and $e_{k}^i \equiv 0$. The proof of the following theorem should be read as a warmup before moving onto the proof of the full theorem, which is presented in Appendix~\ref{sec:async}.
\begin{theorem}[Convergence of SMART in the Synchronous Case]\label{thm:syncnewalgoconvergence}
For all $k \geq 0$, let $\cI_k = \sigma((i_k, \sfS_k))$, let $\cE_k = \sigma(\varepsilon_k)$, let 
$$\cF_k :=  \sigma(\{x^l\}_{l =0}^k\cup \{y_1^l, \ldots, y_n^l\}_{l = 0}^k),$$ and suppose that $\{ \cI_k, \cE_k, \cF_k\}$ are independent. Suppose that~\eqref{eq:coherence} holds. Finally assume that $d_k \equiv 0$, $e^{ij}_{k} \equiv 0$,  and $\lambda_k := \lambda$ satisfies
\begin{align*}
\lambda < \begin{cases}
 \min_{i,j}\left\{\frac{n^2p_{ij} \beta_{ij}q_j m}{\umu_j} \right\} & \text{if $\mathbf{S}^\ast = 0$}; \\
  \min_{i,j}\left\{\frac{n^2p_{ij} \beta_{ij}q_j m}{2\umu_j}\right\} & \text{otherwise}.
 \end{cases}
\end{align*}
Then
\begin{enumerate}
\item \textbf{Convergence of operator values.}\label{thm:syncnewalgoconvergence:part:operatorvalues}  For $i \in \{1, \ldots, n\}$, the sequence of $\cH$-valued random variables $\{S_i(x^k)\}_{k \in \NN} \as$ converges strongly to $S_i(x^\ast)$.
\item \textbf{Weak convergence.} \label{thm:syncnewalgoconvergence:part:weak} 
Suppose that $S$ is demiclosed at $0$. Then the sequence of $\cH$-valued random variables $\{x^k\}_{k \in \NN} \as$ weakly converges to an $\cS$-valued random variable.
\item \textbf{Linear convergence.} \label{thm:syncnewalgoconvergence:part:linear} Let $\eta :=  \min_{i, j} \{\rho p_{ij}^T\},$ let $\alpha \in [0, 1)$, and let
\begin{align}\label{eq:nonasynclambdalinear}
\lambda &\leq \begin{cases}
\min_{i,j}\left\{\frac{(1-\alpha)\beta_{ij} n^2 p_{ij}q_jm}{\umu_j}  \right\} & \text{if $\mathbf{S}^\ast = 0$};\\
\min_{i,j}\left\{\frac{\eta(1-\alpha)\beta_{ij} n^2 p_{ij}q_jm}{2\umu_j\eta + 2\mu\alpha(1-\alpha)\beta_{ij}n^2p_{ij}q_j}\right\} & \text{otherwise.}
\end{cases}
\end{align}
Then if~\eqref{eq:essstrongquasimono} holds, there exists a constant $C(z^0, \phi^0) \in \vR_{\geq 0}$ depending on $x^0$ and $\phi^0$ such that for all $k \in \vN$, 
\begin{align*}
\EE\left[d_{\cS}^2(x^k)\right] &\leq  \left(1 - \frac{2\alpha\mu\lambda}{m} \right)^{k}\left(d_\cS^2(x^0)+ C(x^0, \phi^0)\right).
\end{align*}
\end{enumerate}
\end{theorem}
\begin{proof}
\indent\textbf{Notation.} Let $x^\ast \in \cS$. The following quantities are often repeated---in both the synchronous and asynchronous settings.
\begin{enumerate}
\item \textbf{The $t_{k}$ variable.} For all $k \geq 0$, let $t_k$ be the $\{0, (q_1m)^{-1}\} \times \cdots \times  \{0, (q_m m)^{-1}\}$-valued random variable, which for all $j \in \{1, \ldots, m\}$, satisfies 
\begin{align*}
t_{k,j} := \begin{cases}
\frac{1}{q_jm} & \text{if $j \in \sfS_k$;} \\
0 & \text{otherwise.}
\end{cases}
\end{align*}
\item \textbf{The primal and dual $Q_i$ operators}. For all $i$, define two new operators $Q_i^p, Q_i^d: \cH \rightarrow \cH$: for all $x, y \in \cH$, set
\begin{align*}
(Q_i^p(x))_j := \begin{cases} 
\frac{1}{np_{ij}}(S_i(x))_j & \text{if } p_{ij} \neq 0;\\
0 & \text{otherwise;} 
\end{cases} && \text{and} && \qquad(Q_i^d(y))_j := \begin{cases} 
\frac{1}{np_{ij}}y_{i,j} & \text{if } p_{ij} \neq 0;\\
0 & \text{otherwise.}
\end{cases}
\end{align*}
For all $i$, define $Q_i^\ast := Q_i^p(x^\ast).$  By~\eqref{eq:coherence}, $Q_i^\ast$ is independent of the choice of $x^\ast \in \zer(S)$.
\item \textbf{The primal and dual $r_{ij}^p$ functions.} For each $i$ and $j$, define two functions $r_{ij}^p, r_{ij}^d : \cH \rightarrow \RR_+$: for all $x, y \in \cH$, let
\begin{align*}
  r_{ij}^p(x) := \|(Q_i^p(x))_j - (Q_i^\ast)_j\|^2_j && \text{and} && r_{ij}^d(y) := \|(Q_i^d(y))_j -  (Q_i^\ast)_j\|^2_j.
\end{align*}
\item \textbf{The $\mkQ^k$ random vector.} If 
$$
\mkQ^k =  t_k \odot\left( Q_{i_k}^p(x^{k}) - Q_{i_k}^d(y_{i_k}^k) + \frac{1}{n}\sum_{i=1}^n y_i^k\right),
$$ then $x^{k+1} = x^k - \lambda  \mkQ^k$ and  $\EE\left[\mkQ^k\mid \cF_k\right] = m^{-1}S(x^{k})$.
\item \textbf{The $\gamma_{ij}$ constants.} For all $i $ and $ j$, choose any constants that satisfy\footnote{The constraint~\eqref{eq:gammabound1} has a solution by the assumptions of the theorem.} 
\begin{align}\label{eq:gammabound1}
\frac{2\umu_j\lambda^2}{q_jm^2} < \gamma_{ij} <  \frac{2\lambda\beta_{ij}n^2p_{ij}}{m} -   \frac{2\umu_j\lambda^2}{q_jm^2}.
\end{align}
\item \textbf{The  $R_{l, ij}$ constants.} For all $i $ and $ j$, we define two \emph{positive constants}\footnote{By the choice of~$\gamma_{ij}$, there is a constant $b > 0$ such that for all $i,j$ and $k$, we have $R_{1, ij} > b$ and $R_{2, ij}> b$.}
\begin{align*}
R_{1, ij} := \gamma_{ij} -   \frac{2\umu_j\lambda^2}{q_jm^2} && \text{and} && R_{2, ij} := \frac{2\lambda\beta_{ij}n^2p_{ij}}m -   \frac{2\umu_j\lambda^2}{q_jm^2} - \gamma_{ij}.
\end{align*}
\end{enumerate}
\begin{assumption}
We only address the case in which $\mathbf{S}^\ast$ is not necessarily the zero matrix. From this case, and with the help of the identity $r_{ij}^d(y_i^{k}) \equiv 0$, it is straightforward to retrieve the case in which $\mathbf{S}^\ast = 0$.
\end{assumption}

\textbf{Parts~\ref{thm:syncnewalgoconvergence:part:operatorvalues} and~\ref{thm:syncnewalgoconvergence:part:weak}}: Two essential elements feature in our proof. The indispensable \emph{supermartingale convergence theorem}~\cite[Theorem 1]{robbins1985convergence}, with which we show that a pivotal sequence of random variables converges, is our hammer for nailing down the effect of randomness in SMART:
\begin{theorem}[Supermartingale convergence theorem]
Let $(\Omega, \cF, P)$ be a probability space. Let $\mathfrak{F} := \{\cF_k\}_{k \in \NN}$ be an increasing sequence of sub $\sigma$-algebras of $\cF$ such that $\cF_k \subseteq \cF_{k+1}$. Let $\{X_k\}_{k \in \NN}$ and $\{Y_k\}_{k \in \NN}$ be sequences of $[0, \infty)$-valued random variables such that for all $k \in \NN$, $X_k$ and $Y_k$ are $\cF_k$ measurable, and 
\begin{align*}
(\forall k \in \NN) \qquad \EE\left[X_{k+1}\mid \cF_k\right] + Y_{k} \leq X_k.
\end{align*}
Then $\sum_{k =0}^\infty Y_k < \infty \as$ and $X_k \as$ converges to a $[0, \infty)$-valued random variable.
\end{theorem}
The other equally indispensable element of our proof is the next inequality, which, when taken together with the supermartingale convergence theorem, will ultimately show that SMART converges: with
\begin{align*}
\kappa_{k} &:= \sum_{j=1}^m \sum_{i=1}^n \frac{\gamma_{ij}p_{ij}}{\rho p_{ij}^T}r^d_{ij}(y_i^k); &&  \text{ and } &&Y_k := \lambda ^2\sum_{j=1}^m \frac{\umu_j}{m^2}\|(S(x^{k}))_j\|_j^2 + 2\lambda ^2\sum_{j=1}^m  \frac{\umu_j}{q_jm^2}\left\|\frac{1}{n} \sum_{i=1}^n y_{i,j}^k  \right\|^2_j \\
&&&&&\hspace{40pt}+ \sum_{j=1}^m\sum_{i=1}^n  p_{ij}R_{1, ij}r_{ij}^d(y_i^k) +\sum_{j=1}^m \sum_{i=1}^{n_1} p_{ij}R_{2, ij} r_{ij}^p(x^{k}),
\end{align*}
the supermartingale inequality holds
\begin{align}\label{eq:syncfejerinequality}
\left(\forall k \in \NN\right)\left(\forall x^\ast \in \cS\right) \qquad \EE\left[\|x^{k+1} - x^\ast\|^2 + \kappa_{k+1}\mid \cF_{k}\right] + Y_k \leq  \|x^{k} - x^\ast\|^2 + \kappa_k.
\end{align}
So, by the supermartingale convergence theorem, the sequence $Y_k$ is $\as$ summable and $X_k := \|x^{k} - x^\ast\|^2 + \kappa_k \as$ converges to a $[0, \infty)$-valued random variable, and these conclusions hold for any element $x^\ast \in \cS$. 

At this point, Part~\ref{thm:newalgoconvergence:part:operatorvalues} of the theorem follows because $\sum_{k=0}^\infty\sum_{j=1}^m \sum_{i=1}^n p_{ij} R_{2, ij} r_{ij}^p(x^k) \leq \sum_{k=0}^\infty Y_k < \infty \as$, and hence, $\|S_i(x^k) - S_i(x^\ast)\| = (np_{ij})^2r_{ij}^p(x^k) \as$ converges to $0$.

Moreover, the road to almost sure weak convergence of $x^k$ is only three steps long:
\begin{enumerate}
\item Because $\sum_{k=0}^\infty Y_k < \infty \as$,  we conclude that $\|S(x^k)\|$ and $r_{ij}^d(y_i^k) \as$ converge to $0$.
\item Because each $r_{ij}^d(y_i^k) \as$ converges to $0$, we conclude that $\kappa_{k} \as$ converges to $0$. 
\item Because $\kappa_k \as$ converges to $0$ and $\|x^k - x^\ast\|^2 + \kappa_k \as$ converges to a $[0, \infty)$-valued random variable, we conclude that $\|x^k - x^\ast\|^2 \as$ converges to a $[0, \infty)$-valued random variable.
\end{enumerate}
At last, $x^k$ weakly converges to an $\cS$-valued random vector:
\begin{proposition}[Weak convergence assuming demiclosedness]\label{eq:weakconvergence}
Suppose that $S$ is demiclosed at $0$. Let $\{z^k\}_{k \in \NN}$ be a sequence of random vectors such that, for all $z^\ast \in \zer(S)$, the sequence  $\{\|z^k - z^\ast\|^2\}_{k \in \NN} \as$ converges to a $[0, \infty)$-valued random variable. In addition, assume that $\|S(z^k)\| \as$ converges to $0$. Then $z^k \as$ converges to an $\cS$-valued random variable.
\end{proposition}
\begin{proof}
The set $\cS$ is closed\footnote{Let $y^k \in \cS$ and suppose that $y^k \rightarrow y \in \cH$. We claim that $ y \in \cS$. Indeed, from~\eqref{eq:coherence}, there exists a constant $\beta > 0$ such that $\beta\|S(y)\|^2 \leq \dotp{S(y), y-y^k} \rightarrow 0.$ Thus, $S(y) = 0$, and $y \in \cS$.}, the space $\cH$ is separable, and for all $z^\ast \in \cS$, the sequence $\{\|z^k - z^\ast\|\}_{k \in \NN} \as$ converges, so the exact argument in~\cite[Prop. 2.3(iii)]{doi:10.1137/140971233} shows that there exists $\hat{\Omega} \subseteq \Omega$ such that $P(\hat{\Omega}) = 1$ with the property that for all $z^\ast \in \zer(S)$ and for all $\omega \in \hat{\Omega}$, the sequence $\{\|z^k(\omega) - z^\ast\|^2\}_{k \in \NN}$ converges. Furthermore, by assumption, there exists $\widetilde{\Omega}\subseteq \hat{\Omega}$ such that $P(\widetilde{\Omega}) = 1$ with the property that $\|S(z^k(\omega))\| \as$ almost surely converges to $0$. 

Finishing the proof with demiclosedness, let $\omega \in \widetilde{\Omega}$, and let $z$ be a weak sequential cluster point of $\{z^k(\omega)\}_{k \in \NN}$, say $z^{k_j}(\omega) \rightharpoonup z$ (cluster points exist because $\{z^k\}_{k \in \NN}$ is bounded). Then $\|S(z^{k_j}(\omega))\| \rightarrow 0$ so by the demiclosedness of $S$, the limit is a zero: $S(z) = 0$. In summary, for all $z^\ast \in \cS$, the sequence $\{\|z^k(\omega) - z^\ast\|\}_{k \in \NN}$ converges and every weak sequential cluster point of $\{z^k\}_{k \in \NN}$ is an element of $\cS$. Therefore, by~\cite[Lemma 2.39]{bauschke2011convex}, the sequence $\{z^k(\omega)\}_{k \in \NN}$ weakly converges to an element of $\cS$. Because $\omega \in \widetilde{\Omega}$ is arbitrary and $P(\widetilde{\Omega}) = 1$, the sequence $\{z^k\}_{k \in \NN} \as$ weakly converges, and by the classic result~\cite[Corollary 1.13]{pettis1938integration}, the weak limit of $\{z^k\}_{k \in \NN}$ is measurable. \qed
\end{proof}

Our task is now clear: we must prove~\eqref{eq:syncfejerinequality}. But to do so, we must first bound several random variables. We present the three bounds we need now and defer their proofs until later.
\begin{lemma}[Variance bound]\label{eq:syncvariance}
For $j \in \{1, \ldots, m\}$, let $\eta_{j} > 0$ be a positive real number. Then for all $s \in \NN$, we have 
\begin{align*}
\sum_{j = 1}^m \eta_{j}\EE\left[\|\mkQ_j^s\|_j^2 \mid \cF_s\right] &\leq 2\sum_{j = 1}^m \sum_{i=1}^n \frac{p_{ij}}{q_jm^2}\eta_{j}r_{ij}^p(x^{s})  +2\sum_{j = 1}^m \sum_{i=1}^n \frac{p_{ij}}{q_jm^2} \eta_{j}r_{ij}^d(y_i^s)\\
&\hspace{20pt} - 2\sum_{j=1}^m \frac{1}{q_jm^2}\eta_j\left\|\frac{1}{n} \sum_{i=1}^n y_{i,j}^s \right\|^2_j- \sum_{j=1}^m \frac{\eta_j}{m^2}\|(S(x^{s}))_j\|_j^2.
\end{align*}
\end{lemma}

\begin{lemma}[Using~\eqref{eq:coherence}]\label{eq:kappasync2}
For all $s \in \NN$ and $\lambda > 0$, we have
\begin{align*}
2\lambda \dotp{S(x^{s}), x^s - x^\ast} &\geq  \sum_{j=1}^m\sum_{i=1}^n 2\lambda\beta_{ij}n^2p_{ij}^2r_{ij}^p(x^{s}).
\end{align*}
\end{lemma}

\begin{lemma}[Recursive $\kappa_{k}$ bound]\label{lem:kappasync1}
For all $s \in \NN$, we have
\begin{align*}
\EE\left[ \kappa_{s+1} \mid \cF_k\right] \leq \kappa_s -\sum_{j=1}^m\sum_{i=1}^{n}  p_{ij} \gamma_{ij}r_{ij}^d(y_i^s)  + \sum_{j=1}^m \sum_{i=1}^{n} p_{ij} \gamma_{ij}r_{ij}^p(x^{s}).
\end{align*} 
\end{lemma}

Lemmas in hand, we can now prove~\eqref{eq:syncfejerinequality}
\begin{lemma}[Proof of~\eqref{eq:syncfejerinequality}]
Equation~\eqref{eq:syncfejerinequality} holds for all $s \in \NN$.
\end{lemma}
\begin{proof}
Fix $s \in \NN$. Then
\begin{align*}
&\EE\left[\|x^{s+1} - x^\ast\|^2 + \kappa_{s+1} \mid \cF_s\right] \\
&=\|x^s - x^\ast\|^2 - 2\lambda\EE\left[\dotp{\mkQ^s, x^s - x^\ast}\mid \cF_s\right]+ \lambda^2 \EE\left[\|\mkQ^s\|^2 \mid \cF_s\right] + \EE\left[ \kappa_{s+1}   \mid \cF_s\right] \\
&=\|x^s - x^\ast\|^2 - \frac{2\lambda}{m}\dotp{S(x^{s}), x^s - x^\ast}+ \lambda^2 \EE\left[\|\mkQ^s\|^2 \mid \cF_s\right] + \EE\left[ \kappa_{s+1}   \mid \cF_s\right].
\end{align*}
where the second equality follows from the linearity of expectation. Now, apply Lemmas~\ref{eq:syncvariance},~\ref{eq:kappasync2}, and~\ref{lem:kappasync1}:
\begin{align*}
&\leq \|x^s - x^\ast\|^2 + \kappa_s + \lambda^2 \EE\left[\sum_{j=1}^m \umu_j \|\mkQ_j^s\|_j^2 \mid \cF_s\right] \\
&\hspace{20pt}  -\sum_{j=1}^m\sum_{i=1}^n  p_{ij}\gamma_{ij}r_{ij}^d(y_i^s)  - \sum_{j=1}^m \sum_{i=1}^n p_{ij}\left( \frac{2\lambda\beta_{ij}n^2p_{ij}}{m} - \gamma_{ij}\right)r_{ij}^d(x^{s})  \\
&\leq \|x^s - x^\ast\|^2 + \kappa_s - \lambda^2\sum_{j=1}^m \frac{\umu_j}{m^2}\|(S(x^{s}))_j\|_j^2 - 2\lambda^2\sum_{j=1}^m \frac{\umu_j}{q_jm^2}\left\|\frac{1}{n} \sum_{i=1}^n y_{i,j}^s  \right\|^2_j\\
&\hspace{20pt} - \sum_{j=1}^m\sum_{i=1}^n  p_{ij}\left(\gamma_{ij} - \frac{2\umu_j\lambda^2}{q_jm^2} \right)r_{ij}^d(y_i^s) -\sum_{j=1}^m \sum_{i=1}^n p_{ij}\left(\frac{2\lambda\beta_{ij}n^2p_{ij}}m -   \frac{2\umu_j\lambda^2}{q_jm^2} - \gamma_{ij}\right)r_{ij}^p(x^{s}). \\
&=\|x^s - x^\ast\|^2 + \kappa_s - Y_s. \qquad\qed 
\end{align*}
\end{proof}

We finish the proof of Parts~\ref{thm:syncnewalgoconvergence:part:operatorvalues} and~\ref{thm:syncnewalgoconvergence:part:weak} by proving Lemmas~\ref{eq:syncvariance},~\ref{eq:kappasync2}, and~\ref{lem:kappasync1}.
\begin{proof}[of Lemma~\ref{eq:syncvariance} (variance bound)]
In the next sequence of inequalities, we use the following variance identity three times: For any random vector $X : \Omega \rightarrow \cH_j$ and a sub $\sigma$-algebra $\cX$ on $\Omega$, we have  $\EE\left[\|X - \EE\left[X \mid\cX\right]\|_j^2 \mid \cX\right]  = \EE\left[ \|X\|_j^2\mid \cX\right] - \|\EE\left[X \mid\cX\right]\|_j^2 $. 
By the law of iterated expectation: 
\begin{align*}
&2 \EE\left[ \|t_{s,j}((Q_{i_s}^d(y_{i_s}^s))_j - (Q_{i_s}^\ast)_j - \frac{1}{n} \sum_{i=1}^n y_{i,j}^s ) \|^2_j\mid \cF_s\right]\\
&= 2 \EE\left[ \EE\left[\|t_{s,j}((Q_{i_s}^d(y_{i_s}^s))_j - (Q_{i_s}^\ast)_j - \frac{1}{n} \sum_{i=1}^n y_{i,j}^s ) \|^2_j\mid \sigma(t_{s,j}, \cF_s) \right]\mid \cF_s\right] \\
&= 2 \EE\left[  \EE\left[\|t_{s,j}((Q_{i_s}^d(y_{i_s}^s))_j - (Q_{i_s}^\ast)_j) \|^2_j\mid \sigma(t_{s,j}, \cF_s) \right] - \left\|t_{s,j}\frac{1}{n} \sum_{i=1}^n y_{i,j}^s  \right\|_j^2 \mid \cF_s\right] \\
&=2  \sum_{i=1}^n \frac{p_{ij}}{q_jm^2}\|((Q_{i_s}^d(y_{i}^s))_j - (Q_{i}^\ast)_j)   \|_j^2 -\frac{2}{q_jm^2} \left\|\frac{1}{n} \sum_{i=1}^n y_{i,j}^s  \right\|^2_j.
\end{align*}
Therefore, 
\begin{align*}
&\EE\left[\|\mkQ_j^s\|_j^2 \mid \cF_s\right] \\
&= \EE\left[\|\mkQ_j^s - \frac{1}{m}(S(x^{s}))_j\|_j^2 \mid \cF_s \right] + \frac{1}{m^2}\|(S(x^{s}))_j\|_j^2 \\
&= \EE\biggl[ \|t_{s,j} \left[(Q_{i_s}^d(y_{i_s}^s))_j - (Q_{i_s}^\ast)_j - \frac{1}{n} \sum_{i=1}^n y_{i,j}^s \right]  \\
&\hspace{20pt} -t_{s,j}\left[ (Q_{i_s}^p(x^{s}))_j  - (Q_{i_s}^\ast)_j\right] + \frac{1}{m}(S(x^{s}))_{j}   \|_j^2 \mid \cF_s \biggr]  + \frac{1}{m^2}\|(S(x^{s}))_j\|_j^2 \\
&\leq 2 \EE\left[ \|t_{s,j}((Q_{i_s}^d(y_{i_s}^s))_j - (Q_{i_s}^\ast)_j - \frac{1}{n} \sum_{i=1}^n y_{i,j}^s ) \|_j^2\mid \cF_s\right] \\
&\hspace{20pt} + 2\EE\left[ \| t_{s,j}((Q_{i_s}^p(x^{s}))_j  - (Q_{i_s}^\ast)_j)  - \frac{1}{m}(S(x^{s}))_j  \|_j^2 \mid \cF_s\right] + \frac{1}{m^2}\|(S(x^{s}))_j\|_j^2  \\
&\leq 2  \EE\left[ \|t_{s,j}((Q_{i_s}^d(y_{i_s}^s))_j - (Q_{i_s}^\ast)_j)   \|_j^2\mid \cF_s\right] + 2\EE\left[ \|t_{s,j}((Q_{i_s}^p(x^{s}))_j - (Q_{i_s}^\ast)_j ) \|_j^2\mid \cF_s\right] \\
&\hspace{20pt} - \frac{2}{q_jm^2}\left\|\frac{1}{n} \sum_{i=1}^n y_{i,j}^s  \right\|^2_j- \frac{1}{m^2}\|(S(x^{s}))_j\|_j^2. \\
&\leq 2  \sum_{i=1}^n \frac{p_{ij}}{q_jm^2}\|((Q_{i}^d(y_{i}^s))_j - (Q_{i}^\ast)_j)   \|_j^2 + 2\sum_{i=1}^n \frac{p_{ij}}{q_jm^2}\|((Q_{i}^p(x^{s}))_j - (Q_{i}^\ast)_j ) \|_j^2 \\
&\hspace{20pt} - \frac{2}{q_jm^2}\left\|\frac{1}{n} \sum_{i=1}^n y_{i,j}^s \right\|^2_j- \frac{1}{m^2}\|(S(x^{s}))_j\|_j^2.
\end{align*}
To get the claimed identity, multiply this inequality by $\eta_j$ for each $j $ and sum.
\qed
\end{proof}

\begin{proof}[of Lemma~\ref{eq:kappasync2} (using~\eqref{eq:coherence})]
From~\eqref{eq:coherence}, 
\begin{align*}
 \dotp{S(x^{s}), x^{s} - x^\ast} \geq  \sum_{j=1}^m\sum_{i=1}^n\beta_{ij}\|(S_{i}(x^{s}))_j - (S_{i}(x^\ast))_j\|_j^2  \geq  \sum_{j=1}^m\sum_{i=1}^n \beta_{ij}n^2p_{ij}^2\|(Q_{i}^p(x^{s}))_j - (Q_{i}^\ast)_j\|_j^2.
\end{align*}
\end{proof}

\begin{proof}[of Lemma~\ref{lem:kappasync1} (recursive $\kappa_{k}$ bound)]
For any $i, j$ with $\mathbf{S}^\ast_{ij} \neq 0$ and $\gamma_{ij} > 0$, 
\begin{align*}
\EE\left[ r_{ij}^d(y_i^{s+1}) \mid \cF_k\right] &= \left(1- \rho p_{ij}^T\right) r_{ij}^d(y_i^{s}) + \rho p_{ij}^T r_{ij}^p(x^{s})
\end{align*} 
because $r_{ij}^d(y_i^{s+1})$ depends only on $y_{i,j}^{s+1}$, not on its other components, and the probability of update for $(y_i^k)_j$ is $P((i_k, i) \in E, t_{k,j} \neq 0, \epsilon_k = 1) = \rho p_{ij}^T$.  When $\mathbf{S}^\ast_{ij} = 0$, the left hand side of the equation is zero, so the above equation holds as an inequality. Thus, to get the claimed inequality, multiply both sides of the above equation by $p_{ij}\gamma_{ij}(\rho p_{ij}^T)^{-1}$ and sum over $i$ and $j$. \qed
\end{proof}

\textbf{Part~\ref{thm:syncnewalgoconvergence:part:linear}}: In this part, we no longer work with arbitrary zeros $x^\ast \in \cS$. Instead we work with the sequence of zeros $P_{\cS}(x^k)$, which, by definition, satisfy 
\begin{align*}
d_{\cS}^2(x^{k+1}) = \|x^{k+1} - P_{\cS}(x^{k+1})\|^2 \leq \|x^{k+1} - P_{\cS}(x^{k})\|^2.
\end{align*}
But notice that, because they are independent of the fixed zero $x^\ast$ from Parts~\ref{thm:syncnewalgoconvergence:part:operatorvalues} and~\ref{thm:syncnewalgoconvergence:part:weak}, we can freely use the inequalities in Lemmas~\ref{eq:syncvariance},~\ref{eq:kappasync2}, and~\ref{lem:kappasync1}. 

In this part, we further constrain the constants $\gamma_{ij}$. (We defer the proof of the lemma for a moment.)
\begin{lemma}[Choosing $\gamma_{ij}$]\label{eq:synclinearconstant}
Let $\xi := 2m^{-1}\mu\alpha\lambda$. Then for all $i$ and $j$, there exists $\gamma_{ij} >0$ such that 
\begin{align*}
\frac{2\umu_j\lambda^2}{q_jm^2\left(1-\frac{\xi}{\rho p_{ij}^T}\right)}  \leq \gamma_{ij} \leq \frac{2(1-\alpha)\lambda\beta_{ij}n^2 p_{ij}}{m}  - \frac{2\umu_j\lambda^2}{q_jm^2}.
\end{align*}
\end{lemma}

Then we finish off the proof by appealing to the lower bound
$$\dotp{S(x^{k}), x^{k} - P_\cS(x^k)} \geq \mu\|x^{k} -  P_{\cS}(x^k))\|^2
$$
in the following sequence of inequalities
\begin{align*} 
&\EE\left[d_\cS^2(x^{k+1}) + \kappa_{k+1} \mid \cF_k\right] \\
&\leq \EE\left[\|x^{k+1} - P_\cS(x^{k})\|^2 + \kappa_{k+1} \mid \cF_k\right] \\
&=\|x^k - P_{\cS}(x^k)\|^2 - 2\lambda\EE\left[\dotp{\mkQ^k, x^k - P_{\cS}(x^k)}\mid \cF_k\right]+ \lambda^2 \EE\left[\|\mkQ^k\|^2 \mid \cF_k\right] + \EE\left[ \kappa_{k+1}   \mid \cF_k\right] \\
&=\|x^k -  P_\cS(x^{k})\|^2 - \frac{2\lambda}{m}\dotp{S(x^{k}), x^k - P_\cS(x^{k})}+ \lambda^2 \EE\left[\|\mkQ^k\|^2 \mid \cF_k\right] + \EE\left[ \kappa_{k+1}   \mid \cF_k\right]  \\
&\leq \left(1-\frac{2\alpha\mu\lambda}{m}\right)\|x^k - P_\cS(x^{k})\|^2 - \frac{2\lambda(1-\alpha)}{m}\dotp{S(x^k), x^k -  P_\cS(x^{k})}+ \lambda^2 \EE\left[\sum_{j=1}^m \umu_j \|\mkQ_j^k\|_j^2 \mid \cF_k\right] + \EE\left[ \kappa_{k+1}   \mid \cF_k\right] \\
&\leq \left(1-\frac{2\alpha\mu\lambda}{m}\right)\|x^k -  P_\cS(x^{k})\|^2 + \kappa_k \qquad \text{(Lemmas~\ref{eq:syncvariance},~\ref{eq:kappasync2}, and~\ref{lem:kappasync1} are used below)} \\
&\hspace{20pt} - \sum_{j=1}^m\sum_{i=1}^n  p_{ij}\left(\gamma_{ij} - \frac{2\umu_j\lambda^2}{q_jm^2} \right)r_{ij}^d(y_i^k)  -\sum_{j=1}^m \sum_{i=1}^n p_{ij}\left(\frac{2(1-\alpha)\lambda\beta_{ij}n^2p_{ij}}{m} -   \frac{2\umu_j\lambda^2}{q_jm^2} - \gamma_{ij}\right)r_{ij}^p(x^k) \\
&\leq (1-\xi)\left(\|x^k -   P_\cS(x^{k})\|^2 + \kappa_k\right) - \sum_{j=1}^m\sum_{i=1}^n  p_{ij}\left(\gamma_{ij}\left(1 - \frac{\xi}{\rho p_{ij}^T}\right) -  \frac{2\umu_j\lambda^2}{q_jm^2} \right)r_{ij}^d(y_i^k)  \\
&\leq(1-\xi)\left(d_\cS^2(x^k) + \kappa_k\right)
\numberthis\label{eq:linearbound}
\end{align*}
(Apply the law of iterated expectations to get the linear convergence rate (with $C(z^0, \phi^0) := \kappa_0$).)

The only loose end, which we now tie up, is the proof of Lemma~\ref{eq:synclinearconstant}.
\begin{proof}[of Lemma~\ref{eq:synclinearconstant} (choosing $\gamma_{ij}$)]
We have assumed that
\begin{align*}
\lambda &\leq 
\min_{i,j}\left\{\frac{\eta(1-\alpha)\beta_{ij} n^2 p_{ij}q_jm}{2\umu_j\eta + 2\mu\alpha(1-\alpha)\beta_{ij}n^2p_{ij}q_j}\right\},
\end{align*}
and consequently, if
$$w_{ij} :=  \frac{\frac{(1-\alpha)\beta_{ij}n^2p_{ij}q_jm}{2\umu_j}}{\frac{m\eta}{2\alpha\mu} + \frac{(1-\alpha)\beta_{ij}n^2p_{ij}q_jm}{2\umu_j}} < 1,$$
then
\begin{align*}
\lambda \leq  \frac{w_{ij}m\eta}{2\alpha\mu} = \frac{(1-\alpha)\beta_{ij}n^2p_{ij}q_jm(1-w_{ij})}{2\umu_j} \leq \frac{(1-\alpha)\beta_{ij}n^2p_{ij}q_jm(1-w_{ij})}{(2-w_{ij})\umu_j} .
\end{align*}
In particular, $\xi \leq w_{ij}\eta < 1$. Now rearrange the above inequality to get
\begin{align*}
 \frac{\umu_j\lambda^2}{q_jm^2\left(1-w_{ij}\right)} + \frac{\umu_j\lambda^2}{q_jm^2} = \frac{\umu_j\lambda^2(2-w_{ij})}{q_jm^2\left(1-w_{ij}\right)}  \leq \frac{2\umu_j\lambda^2}{q_jm^2\left(1-w_{ij}\right)} \leq \frac{(1-\alpha)\lambda\beta_{ij}n^2p_{ij}}{m}.
\end{align*}
The last bound proves the claimed bound because $\left(1-\frac{\xi}{\rho p_{ij}^T}\right)^{-1} \leq \left(1 - \frac{w_{ij}\eta}{\rho p_{ij}^T}\right)^{-1}  \leq \left(1-w_{ij}\right)^{-1}$.
\qed\end{proof}
\qed\end{proof}

\section{Future Work}

The SMART algorithm calls for five avenues of future work: numerical experiments, especially in the asynchronous setting; the creation of new operators $S$; the characterization of the error bound condition which we call essential strong quasi-monotonicity; the convergence rate analysis of the SMART algorithm without assuming essential strong quasi-monotonicity; and nonconvex extensions. All of these avenues present challenges.

The SAGA, SVRG, Finito, SDCA, and randomized projection algorithms presented in Section~\ref{sec:connections} perform well in the synchronous setting. Asynchronous algorithms often obtain a linear speedup (in the number of computing cores) over their synchronous implementations, so we expect these algorithms to perform even better in the asynchronous setting. However, implementing asynchronous algorithms still requires a bit of programming expertise, so we expect that good experimental results will take some time to acquire.

Any operator that satisfies the coherence condition~\eqref{eq:coherence} can be plugged into the SMART algorithm; as such, the power of SMART increases with each new operator discovered. We presented many examples of operators in this paper, and we expect there to be more in the future. 

The essential strong quasi-monotonicity property~\eqref{eq:essstrongquasimono} appears to be the weakest possible condition under which a first-order algorithm will converge linearly. This deep, difficult to characterize property is related to the Hoffman bound~\cite{hoffman1952approximate}, the linear regularity assumption, and the Kurdyka-{\L}ojasiewicz property~\cite{bolte2015error}. We look forward to a calculus of operations that preserve this property.

Even if $S$ is not essentially strongly quasi-monotone, we suspect that SMART converges sublinearly; to show these rates, the techniques in~\cite{myPDpaper,davis2014convergenceFaster,myFDRS,davis2014convergence} should be adapted to the stochastic and asynchronous settings.

We showed that SMART converges when the coherence condition is satisfied, and this tethers our results to convex optimization problems. But SMART can and should be applied to nonconvex problems. In the nonconvex case, convergence guarantees will certainly be weaker than those presented here, but SMART will likely perform well on these problems.

\paragraph{\textbf{Acknowledgements:}} We thank Brent Edmunds and Professors Patrick Combettes and Wotao Yin for helpful comments.

\bibliographystyle{spmpsci}
\bibliography{bibliography}

\begin{thebibliography}{10}
\providecommand{\url}[1]{{#1}}
\providecommand{\urlprefix}{URL }
\expandafter\ifx\csname urlstyle\endcsname\relax
  \providecommand{\doi}[1]{DOI~\discretionary{}{}{}#1}\else
  \providecommand{\doi}{DOI~\discretionary{}{}{}\begingroup
  \urlstyle{rm}\Url}\fi

\bibitem{baillon1977quelques}
Baillon, J.B., Haddad, G.: Quelques propri{\'e}t{\'e}s des op{\'e}rateurs
  angle-born{\'e}s et {$n$}-cycliquement monotones.
\newblock Israel Journal of Mathematics \textbf{26}(2), 137--150 (1977)

\bibitem{bauschke2011convex}
Bauschke, H.H., Combettes, P.L.: {Convex Analysis and Monotone Operator Theory
  in Hilbert Spaces}, 1st edn.
\newblock Springer Publishing Company, Incorporated (2011)

\bibitem{subgradient_projectors}
Bauschke, H.H., Wang, C., Wang, X., Xu, J.: {On Subgradient Projectors}.
\newblock SIAM Journal on Optimization \textbf{25}(2), 1064--1082 (2015)

\bibitem{bertsekasincrementalproximal}
Bertsekas, D.P.: Incremental proximal methods for large scale convex
  optimization.
\newblock Mathematical Programming \textbf{129}(2), 163--195 (2011)

\bibitem{bertsekas2015incremental}
Bertsekas, D.P.: {Incremental Aggregated Proximal and Augmented Lagrangian
  Algorithms}.
\newblock arXiv preprint arXiv:1509.09257  (2015)

\bibitem{bianchi2015ergodic}
Bianchi, P.: A stochastic proximal point algorithm: convergence and application
  to convex optimization.
\newblock In: {Computational Advances in Multi-Sensor Adaptive Processing
  (CAMSAP), 2015 IEEE 6th International Workshop on}

\bibitem{bolte2015error}
Bolte, J., Nguyen, T.P., Peypouquet, J., Suter, B.: From error bounds to the
  complexity of first-order descent methods for convex functions.
\newblock arXiv preprint arXiv:1510.08234  (2015)

\bibitem{combettes2015asynchronous}
Combettes, P.L., Eckstein, J.: {Asynchronous Block-Iterative Primal-Dual
  Decomposition Methods for Monotone Inclusions}.
\newblock arXiv preprint arXiv:1507.03291  (2015)

\bibitem{doi:10.1137/140971233}
Combettes, P.L., Pesquet, J.C.: {Stochastic Quasi-Fej{\'e}r Block-Coordinate
  Fixed Point Iterations with Random Sweeping}.
\newblock SIAM Journal on Optimization \textbf{25}(2), 1221--1248 (2015)

\bibitem{combettes2014compositions}
Combettes, P.L., Yamada, I.: Compositions and convex combinations of averaged
  nonexpansive operators.
\newblock Journal of Mathematical Analysis and Applications \textbf{425}(1), 55
  -- 70 (2015).
\newblock \doi{http://dx.doi.org/10.1016/j.jmaa.2014.11.044}.
\newblock
  \urlprefix\url{http://www.sciencedirect.com/science/article/pii/S0022247X14010865}

\bibitem{myPDpaper}
Davis, D.: {Convergence Rate Analysis of Primal-Dual Splitting Schemes}.
\newblock SIAM Journal on Optimization \textbf{25}(3), 1912--1943 (2015)

\bibitem{myFDRS}
Davis, D.: {Convergence Rate Analysis of the Forward-Douglas-Rachford Splitting
  Scheme}.
\newblock SIAM Journal on Optimization \textbf{25}(3), 1760--1786 (2015)

\bibitem{davis2014convergenceFaster}
Davis, D., Yin, W.: {Faster convergence rates of relaxed Peaceman-Rachford and
  ADMM under regularity assumptions}.
\newblock arXiv preprint arXiv:1407.5210v3  (2014)

\bibitem{davis2014convergence}
Davis, D., Yin, W.: Convergence rate analysis of several splitting schemes.
\newblock In: R.~Glowinski, S.~Osher, W.~Yin (eds.) {Splitting Methods in
  Communication and Imaging, Science and Engineering}. Springer International
  Publishing, New York (to appear)

\bibitem{defazio2014saga}
Defazio, A., Bach, F., Lacoste-Julien, S.: {SAGA: A Fast Incremental Gradient
  Method With Support for Non-Strongly Convex Composite Objectives}.
\newblock In: Advances in Neural Information Processing Systems, pp. 1646--1654
  (2014)

\bibitem{defazio2014finito}
Defazio, A., Domke, J., Caetano, T.: {Finito: A Faster, Permutable Incremental
  Gradient Method for Big Data Problems}.
\newblock In: Proceedings of the 31st International Conference on Machine
  Learning (ICML-14), pp. 1125--1133 (2014)

\bibitem{approx}
Fercoq, O., Richt{\'a}rik, P.: {Accelerated, Parallel, and Proximal Coordinate
  Descent}.
\newblock SIAM Journal on Optimization \textbf{25}(4), 1997--2023 (2015)

\bibitem{hoffman1952approximate}
Hoffman, A.J.: {On Approximate Solutions of Systems of Linear Inequalities}.
\newblock Journal of Research of the National Bureau of Standards
  \textbf{49}(4), 263--265 (1952)

\bibitem{neighorhoodwatch}
Hofmann, T., Lucchi, A., Lacoste-Julien, S., McWilliams, B.: {Neighborhood
  Watch: Stochastic Gradient Descent with Neighbors}.
\newblock arXiv preprint arXiv:1506.03662v3  (2015)

\bibitem{johnson2013accelerating}
Johnson, R., Zhang, T.: {Accelerating Stochastic Gradient Descent Using
  Predictive Variance Reduction}.
\newblock In: Advances in Neural Information Processing Systems, pp. 315--323
  (2013)

\bibitem{konevcny2014ms2gd}
Kone{\v{c}}n{\`y}, J., Liu, J., Richt{\'a}rik, P., Tak{\'a}{\v{c}}, M.: {mS2GD:
  Mini-Batch Semi-Stochastic Gradient Descent in the Proximal Setting}.
\newblock arXiv preprint arXiv:1410.4744  (2014)

\bibitem{konevcny2014semi}
Kone{\v{c}}n{\`y}, J., Qu, Z., Richt{\'a}rik, P.: Semi-stochastic coordinate
  descent.
\newblock arXiv preprint arXiv:1412.6293  (2014)

\bibitem{konevcny2013semi}
Kone{\v{c}}n{\`y}, J., Richt{\'a}rik, P.: {Semi-Stochastic Gradient Descent
  Methods}.
\newblock arXiv preprint arXiv:1312.1666  (2013)

\bibitem{lai2013augmented}
Lai, M.J., Yin, W.: {Augmented \$$\ell\_1$\$ and Nuclear-Norm Models with a
  Globally Linearly Convergent Algorithm}.
\newblock SIAM Journal on Imaging Sciences \textbf{6}(2), 1059--1091 (2013)

\bibitem{liu2015asynchronous}
Liu, J., Wright, S.J., R{\'e}, C., Bittorf, V., Sridhar, S.: {An Asynchronous
  Parallel Stochastic Coordinate Descent Algorithm}.
\newblock Journal of Machine Learning Research \textbf{16}, 285--322 (2015)

\bibitem{liu2014asynchronous}
Liu, J., Wright, S.J., Sridhar, S.: {An Asynchronous Parallel Randomized
  Kaczmarz Algorithm}.
\newblock arXiv preprint arXiv:1401.4780  (2014)

\bibitem{doi:10.1137/S1052623499362111}
Nedic, A., Bertsekas, D.P.: {Incremental Subgradient Methods for
  Nondifferentiable Optimization}.
\newblock SIAM Journal on Optimization \textbf{12}(1), 109--138 (2001)

\bibitem{doi:10.1137/100802001}
Nesterov, Y.: {Efficiency of Coordinate Descent Methods on Huge-Scale
  Optimization Problems}.
\newblock SIAM Journal on Optimization \textbf{22}(2), 341--362 (2012)

\bibitem{peng2015arock}
Peng, Z., Xu, Y., Yan, M., Yin, W.: {ARock: an Algorithmic Framework for
  Asynchronous Parallel Coordinate Updates}.
\newblock arXiv preprint arXiv:1506.02396  (2015)

\bibitem{pettis1938integration}
Pettis, B.J.: {On Integration in Vector Spaces}.
\newblock Transactions of the American Mathematical Society \textbf{44}(2),
  277--304 (1938)

\bibitem{recht2011hogwild}
Recht, B., Re, C., Wright, S., Niu, F.: {Hogwild: A Lock-Free Approach to
  Parallelizing Stochastic Gradient Descent}.
\newblock In: Advances in Neural Information Processing Systems, pp. 693--701
  (2011)

\bibitem{robbins1985convergence}
Robbins, H., Siegmund, D.: {A Convergence Theorem for Non Negative Almost
  Supermartingales and Some Applications}.
\newblock In: Herbert Robbins Selected Papers, pp. 111--135. Springer (1985)

\bibitem{rockafellar1970monotone}
Rockafellar, R.: {Monotone Operators Associated with Saddle-functions and
  Minimax Problems}.
\newblock Nonlinear functional analysis \textbf{18}(part 1), 397--407 (1970)

\bibitem{schmidt2013minimizing}
Schmidt, M., Roux, N.L., Bach, F.: {Minimizing finite sums with the stochastic
  average gradient}.
\newblock arXiv preprint arXiv:1309.2388  (2013)

\bibitem{shalev2014accelerated}
Shalev-Shwartz, S., Zhang, T.: Accelerated proximal stochastic dual coordinate
  ascent for regularized loss minimization.
\newblock Mathematical Programming pp. 1--41

\bibitem{strohmer2009randomized}
Strohmer, T., Vershynin, R.: {A randomized Kaczmarz algorithm with exponential
  convergence}.
\newblock Journal of Fourier Analysis and Applications \textbf{15}(2), 262--278
  (2009)

\bibitem{wang2013incremental}
Wang, M., Bertsekas, D.P.: {Stochastic First-Order Methods with Random
  Constraint Projection}.
\newblock SIAM Journal on Optimization \textbf{26}(1), 681--717 (2016)

\bibitem{zhang2015restricted}
Zhang, H.: The restricted strong convexity revisited: analysis of equivalence
  with error bound and quadratic growth.
\newblock arXiv preprint arXiv:1511.01635  (2015)

\end{thebibliography}

\section*{Appendix}
\appendix
\section{Proof in the Asynchronous Case}\label{sec:async}

\begin{proof}[Convergence of SMART]
~\\
\indent\textbf{Notation.} Let $x^\ast \in \cS$. Some of our notation, in particular, the definitions of $t_k, Q_i^p, Q_i^d$, $r_{ij}^p$, and $r_{ij}^d$, remains the same as in the proof of the nonasynchronous case (Theorem~\ref{thm:syncnewalgoconvergence}). Below, we adapt the rest of the notation to the present asynchronous case:
\begin{enumerate}
\item \textbf{Integer indexed sequences.} We extend every sequence indexed by natural numbers, say $\{z^k\}_{k \in \NN}$, to all of $\ZZ$ by setting  $z^{k} := z^0$ for all $k \leq 0$.
\item \textbf{The $\mkQ^k$ random vector.} If 
$$
\mkQ^k =  t_{k} \odot\left( Q_{i_k}^p(x^{k-d_k}) - Q_{i_k}^d(y_i^{k-e_{k}^i}) + \frac{1}{n} \sum_{i=1}^n y_{i}^{k-e_k^i}\right),
$$ then $x^{k+1} = x^k - \lambda  \mkQ^k$ and  $\EE\left[\mkQ^k\mid \cF_k\right] = m^{-1}S(x^{k-d_k})$.
\item \textbf{The $C$ constant.} Define the positive constant
$$C := \frac{\sqrt{2(\tau_d + 2)}}{m\sqrt{\underline{q}}}.$$
\item \textbf{The $\gamma_{ij}$ constants.} For all $i $ and $ j$, choose any constants that satisfy\footnote{The constraint~\eqref{eq:gammabound1async} has a solution by the assumptions of the theorem.}
\begin{align*}
\frac{2\umu_j(\tau_d+1)}{q_jm^2}\left(1 + \frac{\tau_p}{mC}\right)\lambda_k^2&\leq \frac{2\umu_j(\tau_d+1)}{\underline{q}m^2}\left(1 + \frac{\tau_p}{mC}\right)\overline{\lambda}^2\\
&< \gamma_{ij} \\
&<  \frac{2\underline{\lambda} n^2p_{ij}\beta_{ij}}{m} -  \umu_j\left(\frac{2}{\underline{q}m^2} + \tau_p\left(\frac{2}{\underline{q}m^3C} + \frac{C}{m}\right)\right)\overline{\lambda}^2 \\
&\leq\frac{2\lambda_kn^2p_{ij}\beta_{ij}}{m} -  \umu_j\left(\frac{2}{q_jm^2} + \tau_p\left(\frac{2}{q_jm^3C} + \frac{C}{m}\right)\right)\lambda_k^2.\numberthis\label{eq:gammabound1async}
\end{align*}
\item \textbf{The  $R_{l, ij}^k$ constants.} For all $i $, $ j$, and $k$, we define two \emph{positive constants}\footnote{By the choice of~$\gamma_{ij}$, there is a constant $b > 0$ such that for all $i,j$ and $k$, we have $R^k_{1, ij} > b$ and $R^k_{2, ij}> b$.}
\begin{align*}
R_{1, ij}^k &:= \frac{\gamma_{ij}}{\tau_d+1} - \frac{2\umu_j}{q_jm^2}\left(1 + \frac{\tau_p}{mC}\right)\lambda_k^2 ; \text{ and} \\
R_{2, ij}^k &:=  \frac{2\lambda_k n^2p_{ij}\beta_{ij}}{m} - \umu_j\left(\frac{2}{q_jm^2} + \tau_p\left(\frac{2}{q_jm^3C} + \frac{C}{m}\right)\right)\lambda_k^2 - \gamma_{ij}.
\end{align*}
\end{enumerate}

\begin{assumption}
We only address the case in which $\mathbf{S}^\ast$ is not necessarily the zero matrix. From this case, and with the help of the identity $r_{ij}^d(y_i^{k-e_{k}^i}) \equiv 0$, it is straightforward to retrieve the case in which $\mathbf{S}^\ast = 0$.
\end{assumption}

\textbf{Parts~\ref{thm:newalgoconvergence:part:operatorvalues} and~\ref{thm:newalgoconvergence:part:weak}}: In the synchronous case, the supermartingale convergence theorem and a supermartingale inequality featured; the same is true in the present asynchronous case, but the supermartingale inequality now accounts for the use of delayed iterates and operator values: with
\begin{align*}
 \kappa_{k} &:= \underbrace{\sum_{j=1}^m \sum_{i=1}^n \frac{\gamma_{ij}p_{ij}}{\rho p_{ij}^T}r_{ij}^d(y_i^k) + \sum_{h = 0}^{\tau_d - 1} \sum_{j=1}^m \sum_{i=1}^n \frac{(h+1)p_{ij}\gamma_{ij}}{\tau_d + 1}r_{ij}^d(y_i^{k - \tau_d + h})}_{:=\kappa_{1,k}} +\underbrace{\sum_{h=k-\tau_p+1}^{k}\sum_{j=1}^m \frac{\umu_j(h-k+\tau_p)}{mC}\|x_j^{h} - x_j^{h-1}\|_j^2}_{:=\kappa_{2,k}}; \\
 Y_k &:= \lambda_k^2\sum_{j=1}^m\frac{\umu_j}{m^2}\left(1 + \frac{\tau_p}{mC}\right)\|(S(x^{k-d_k}))_j\|_j^2  + 2\lambda_k^2\sum_{j=1}^m  \frac{\umu_j}{q_jm^2}\left(1 + \frac{\tau_p}{mC}\right)\left\|\frac{1}{n} \sum_{i=1}^n y_{i,j}^{k-e^{i}_{k}}  \right\|^2_j \\
& \sum_{j=1}^m\sum_{i=1}^n  p_{ij}R_{1, ij}^kr_{ij}^d(y_i^{k-e^{i}_{k}}) +  \sum_{j=1}^m \sum_{i=1}^n p_{ij}R_{2, ij}^k r_{ij}^p(x^{k-d_k}) +  \sum_{j=1}^m \sum_{i=1}^n  \sum_{\substack{h=0\\ h\neq \tau_d - e^{i}_{k, j} }}^{\tau_d}\frac{p_{ij}\gamma_{ij}}{\tau_d + 1}r_{ij}^d(y_i^{k - \tau_d +h}), 
\end{align*}
the supermartingale inequality holds
\begin{align}\label{eq:asyncfejerinequality}
\left(\forall k \in \NN\right)\left(\forall x^\ast \in \cS\right) \qquad \EE\left[\|x^{k+1} - x^\ast\|^2 + \kappa_{k+1}\mid \cF_{k}\right] + Y_k \leq  \|x^{k} - x^\ast\|^2 + \kappa_k.
\end{align}
So, by the supermartingale convergence theorem, the sequence $Y_k$ is $\as$ summable and $X_k := \|x^{k} - x^\ast\|^2 + \kappa_k \as$ converges to a $[0, \infty)$-valued random variable, and these conclusions hold for any element $x^\ast \in \cS$. 

At this point, as in the synchronous case, Part~\ref{thm:newalgoconvergence:part:operatorvalues} of the theorem follows because $\sum_{k=0}^\infty\sum_{j=1}^m \sum_{i=1}^n p_{ij} R_{2, ij}^k r_{ij}^p(x^{k-d_k}) \leq \sum_{k=0}^\infty Y_k < \infty \as$, and hence, $\|S_i(x^{k-d_k}) - S_i(x^\ast)\| = (np_{ij})^2r_{ij}^p(x^{k-d_k}) \as$ converges to $0$.

The road to almost sure weak convergence of $x^{k}$, now 5 steps long, is slightly more complicated than in the synchronous case:
\begin{enumerate}
\item Because $\sum_{k=0}^\infty Y_k < \infty \as$,  we conclude that each term\footnote{Here we invoke the assumed equivalence of the two norms $\|x\|$ and $\|x\|_{\pr} = \sqrt{\sum_{j=1}^m \|x_j\|^2_j}$.}
\begin{align*}
\|S(x^{k-d_k})\|; && \left\|\left(\frac{1}{n} \sum_{i=1}^n y_{i,j}^{k-e^{i}_{k}}\right)_{j=1}^m  \right\|; && &r_{ij}^d(y_i^{k-e^{i}_{k}});\\
 r_{ij}^p(x^{k-d_k}); &&  && &  \sum_{\substack{h=0\\ h\neq \tau_d - e^{i}_{k, j} }}^{\tau_d}\frac{p_{ij}\gamma_{ij}}{\tau_d + 1}r_{ij}^d(y_i^{k - \tau_d +h}),
\end{align*}
$\as$ converges to $0$.
\item Because $r_{ij}^d(y_i^{k-e^{i}_{k}})$ and $ \sum_{\substack{h=0; h\neq \tau_d - e^{ij}_{k} }}^{\tau_d}\frac{p_{ij}\gamma_{ij}}{\tau_d + 1}r_{ij}^d(y_i^{k - \tau_d +h}) \as$ converge to $0$, we conclude that $\kappa_{1, k} \as$ converges to $0$.
\item Because $r_{ij}^d(y_i^{k-e^{i}_{k}})$, $\|n^{-1} \sum_{i=1}^n (y_i^{k-e^{i}_{k}})_j\|_j $, and $r_{ij}^p(x^{k-d_k}) \as$ converge to $0$ and because for all $h \in \{0, \ldots, \tau_p\}$, we have 
\begin{align*}
x_j^{k-h + 1} - x_j^{k-h} &= -\lambda_{k-h}t_{k-h,j}\left( (Q_{i_{k-h}}^p(x^{k - h -d_{k-h}}))_j - (Q_{i_{k-h}}^d(y_{i_k}^{k-h - e^{i_{k-h}}_{k-h}}))_{j} + \frac{1}{n} \sum_{i=1}^n y_{i,j}^{k-h- e^{i}_{k-h}}\right) \\
&= -\lambda_{k-h}t_{k-h,j}\biggl( ((Q_{i_{k-h}}^p(x^{k-h - d_{k-h}}))_j - (Q_{i_{k-h}}^\ast)_j) \\
&\hspace{70pt}+ ((Q_{i_{k-h}}^\ast)_j- (Q_{i_{k-h}}^d(y_{i_k}^{k-h - e^{i_{k-h}}_{k-h}}))_{j}) + \frac{1}{n} \sum_{i=1}^n y_{i,j}^{k-h- e^{i}_{k-h}}\biggr) \stackrel{k \rightarrow \infty}{\rightarrow} 0 \as,
\end{align*}
we conclude $\kappa_{2, k}$ and $\|x^{k-d_k} - x^k\|^2\as$ converge to $0$: 
\item Because $\kappa_k \as$ converges to $0$ and $\|x^k - x^\ast\|^2 + \kappa_k \as$ converges to a $[0, \infty)$-valued random variable, we conclude that $\|x^k - x^\ast\|^2 \as$ converges to a $[0, \infty)$-valued random variable.
\item Because $\|x^k - x^\ast\|^2 \as$ converges to a $[0, \infty)$-valued random variable, because $\|x^k - x^{k-d_k}\|\as$ converges to $0$, and because
$$
\|x^{k-d_k} - x^\ast\|^2 - \|x^{k} - x^\ast\|^2 = \|x^k - x^{k-d_k}\|^2 + 2\dotp{x^{k-d_k} - x^k, x^k - x^\ast}  \stackrel{k \rightarrow \infty}{\rightarrow} 0 \as,
$$
we conclude that $\|x^{k-d_k} - x^\ast\|^2 \as$ converges to a $[0, \infty)$-valued random variable.
\end{enumerate}
With the facts listed here, Proposition~\ref{eq:weakconvergence} immediately renders the sequence $x^{k-d_k} \as$ weakly convergent to an $\cS$-valued random variable. And at last, due to the limit $\|x^k - x^{k-d_k}\|\rightarrow 0\as$, the sequence $x^k \as$ weakly converges to an $\cS$-valued random variable.

Our task is now clear: we must prove~\eqref{eq:asyncfejerinequality}. But to do so, we must first bound several random variables. We present the two bounds we need now and defer their proofs until later. 

\begin{lemma}[Variance bound]\label{lem:AsyncFPRbound}
For $j \in \{1, \ldots, m\}$, let $\eta_{j} > 0$ be a positive real number. Then for all $s \in \NN$, we have 
\begin{align*}
\sum_{j = 1}^m \eta_{j}\EE\left[\|\mkQ_j^s\|_j^2 \mid \cF_s\right] &\leq 2\sum_{j = 1}^m \sum_{i=1}^n \frac{p_{ij}\eta_{j}}{q_jm^2}r_{ij}^p(x^{s-d_s}) +2\sum_{j = 1}^m \sum_{i=1}^n \frac{p_{ij} \eta_{j}}{q_jm^2}r_{ij}^d(y_i^{s-e^{i}_{s}})\\
&\hspace{20pt} - 2\sum_{j=1}^m \frac{\eta_j}{q_jm^2}\left\|\frac{1}{n} \sum_{i=1}^n y_{i,j}^{s-e^{i}_{s}} \right\|^2_j- \sum_{j=1}^m \frac{\eta_j}{m^2}\|S(x^{s-d_s})\|_j^2.\numberthis\label{eq:ASYNCSAGAFPRBOUND}
\end{align*}
\end{lemma}

\begin{lemma}[Recursive $\kappa_k$ bound]\label{lem:kappa}
For all $s \in \NN$ and $\alpha \in [0, 1]$, we have
\begin{align*}
&\EE\left[\kappa_{s+1} \mid \cF_s\right] \\
&\leq \kappa_s  + \frac{2\lambda_s}{m} \dotp{S(x^{s-d_s}), x^s - x^\ast} - \frac{2\alpha\lambda_s}{m}\dotp{S(x^{s-d_s}), x^{s-d_s} - x^\ast} + \sum_{j=1}^m \sum_{i=1}^n\frac{2\alpha\lambda_s n^2p_{ij}^2\beta_{ij}}{m}r_{ij}^p(x^{s-d_s}) \\
&+\lambda_s^2 \sum_{j=1}^m\frac{\tau_p\umu_j}{mC}\EE\left[\|\mkQ^s_j\|_j^2 \mid\cF_s\right]  - \sum_{j=1}^m\sum_{i=1}^n p_{ij}\left( \frac{2(1-\alpha)\lambda_s\beta_{ij}n^2p_{ij}}{m} - \frac{\tau_p \umu_jC_j \lambda_s^2}{m}- \gamma_{ij}\right)r_{ij}^p(x^{s-d_s})  \\
& - \sum_{j=1}^m\sum_{i=1}^n  \frac{p_{ij} \gamma_{ij}}{\tau_d + 1}r_{ij}^d(y_i^{s - e^{i}_{s}}) -  \sum_{j=1}^m \sum_{i=1}^n   \sum_{\substack{h = 0\\ h\neq \tau_d - e^{i}_{s, j} }}^{\tau_d}\frac{p_{ij}\gamma_{ij}}{\tau_d + 1}r_{ij}^d(y_i^{s - \tau_d +h}).
\end{align*}
\end{lemma}

Lemmas in hand, we can now prove~\eqref{eq:asyncfejerinequality}. Compared to the synchronous case, we prove a refined bound, depending on the parameter $\alpha$, which will figure into our proof of linear convergence.
\begin{lemma}[Proof of~\eqref{eq:asyncfejerinequality}]\label{lem:ASYNCSAGADESCEND}
Equation~\eqref{eq:asyncfejerinequality} holds for all $s \in \NN$. Moreover, for every $s \in \NN$, $\alpha \in [0, 1]$ and $x^\ast \in \cS$, we have
\begin{align*}
&\EE\left[\|x^{s+1} - x^\ast\|^2 + \kappa_{s+1}\mid \cF_s\right] + Y_s \\
&\hspace{10pt}\leq \|x^s - x^\ast\|^2 + \kappa_s - \frac{2\alpha\lambda_s}{m}\dotp{S(x^{s-d_s}), x^{s-d_s} - x^\ast} + \sum_{j=1}^m \sum_{i=1}^n\frac{2\alpha\lambda_s n^2p_{ij}^2\beta_{ij}}{m}r_{ij}^p(x^{s-d_s}). \numberthis\label{eq:asyncfejerinequalityalpha}
\end{align*}
\end{lemma}
\begin{proof}
Fix $s \in \NN$. Then
\begin{align*}
&\EE\left[\|x^{s+1} - x^\ast\|^2 + \kappa_{s+1} \mid \cF_s\right] \\
&=\|x^s - x^\ast\|^2 - 2\lambda_s\EE\left[\dotp{\mkQ^s, x^s - x^\ast}\mid\cF_s\right]+ \lambda_s^2 \EE\left[\|\mkQ^s\|^2 \mid \cF_s\right] + \EE\left[ \kappa_{s+1}   \mid \cF_s\right]  \\
&=\|x^s - x^\ast\|^2 - \frac{2\lambda_s}{m}\dotp{S(x^{s-d_s}), x^s - x^\ast}+ \lambda_s^2 \EE\left[\|\mkQ^s\|^2 \mid \cF_s\right] + \EE\left[ \kappa_{s+1}   \mid \cF_s\right],  
\end{align*}
where the second equality follows from the linearity of expectation. In the above inequality, apply Lemma~\ref{lem:kappa} to $\EE\left[ \kappa_{s+1}   \mid \cF_s\right]$ and bound $\|\mkQ^s\|^2$ with the equivalence of $\|\cdot\|$ and $\|\cdot\|_\pr$: 
\begin{align*}
&\leq \|x^s - x^\ast\|^2 + \kappa_s - \frac{2\alpha\lambda_s}{m}\dotp{S(x^{s-d_s}), x^{s-d_s} - x^\ast}+ \lambda_s^2 \EE\left[\sum_{j=1}^m \umu_j\left(1 + \frac{\tau_p}{mC}\right) \|\mkQ_j^s\|_j^2 \mid \cF_s\right] \\
&\hspace{0pt}  + \sum_{j=1}^m \sum_{i=1}^n\frac{2\alpha\lambda_s n^2p_{ij}^2\beta_{ij}}{m}r_{ij}^p(x^{s-d_s}) -\sum_{j=1}^m\sum_{i=1}^n  \frac{p_{ij}\gamma_{ij}}{\tau_d +1}r_{ij}^d(y_i^{s-e^{i}_{s}}) \\
&\hspace{0pt} -  \sum_{j=1}^m\sum_{i=1}^n p_{ij}\left( \frac{2\lambda_s\beta_{ij}n^2p_{ij}}{m} - \frac{\tau_p\umu_j C\lambda_s^2}{m} - \gamma_{ij}\right)r_{ij}^p(x^{s-d_s})  \\
&\hspace{0pt}-  \sum_{j=1}^m \sum_{i=1}^n   \sum_{\substack{h = 0\\ h\neq \tau_d - e^{i}_{s, j} }}^{\tau_d}\frac{p_{ij}\gamma_{ij}}{\tau_d + 1}r_{ij}^d(y_i^{s - \tau_d +h}).
\end{align*}
Finally, apply Lemma~\ref{lem:AsyncFPRbound} to get
\begin{align*}
&\leq \|x^s - x^\ast\|^2 + \kappa_s - \frac{2\alpha\lambda_s}{m}\dotp{S(x^{s-d_s}), x^{s-d_s} - x^\ast} +  \sum_{j=1}^m \sum_{i=1}^n\frac{2\alpha\lambda_s n^2p_{ij}^2\beta_{ij}}{m}r_{ij}^p(x^{s-d_s}) \\
&\hspace{0pt} - \lambda_s^2\sum_{j=1}^m \frac{\umu_j}{m^2}\left(1 + \frac{\tau_p}{mC}\right)\|(S(x^{s-d_s}))_j\|_j^2  - 2\lambda_s^2\sum_{j=1}^m \frac{\umu_j}{q_jm^2}\left(1 + \frac{\tau_p}{mC}\right)\left\|\frac{1}{n} \sum_{i=1}^n y_{i,j}^{s-e_{s}^{i}}  \right\|^2_j\\
&\hspace{20pt} - \sum_{j=1}^m\sum_{i=1}^n  p_{ij}\left(\frac{\gamma_{ij}}{\tau_d+1} - \frac{2\umu_j}{q_jm^2}\left(1 + \frac{\tau_p}{mC}\right)\lambda_s^2 \right)r_{ij}^d(y_i^{s-e_{s}^{i}}) \\
&\hspace{0pt}  -\sum_{j=1}^m \sum_{i=1}^n p_{ij}\left(\frac{2\lambda_s\beta_{ij}n^2p_{ij}}{m} - \frac{\tau_p \umu_jC\lambda_s^2}{m} - \frac{2\umu_j}{q_jm^2}\left(1 + \frac{\tau_p}{mC}\right)\lambda_s^2 - \gamma_{ij}\right)r_{ij}^p(x^{s-d_s}) \\
&\hspace{0pt}-  \sum_{j=1}^m \sum_{i=1}^n   \sum_{\substack{h = 0\\ h\neq \tau_d - e_{s,j}^{i} }}^{\tau_d}\frac{p_{ij}\gamma_{ij}}{\tau_d + 1}r_{ij}^d(y_i^{s - \tau_d +h}),
\end{align*}
and combine terms
\begin{align*}
&\leq \|x^s - x^\ast\|^2 + \kappa_s - \frac{2\alpha\lambda_s}{m}\dotp{S(x^{s-d_s}), x^{s-d_s} - x^\ast}  + \sum_{j=1}^m \sum_{i=1}^n\frac{2\alpha\lambda_s n^2p_{ij}^2\beta_{ij}}{m}r_{ij}^p(x^{s-d_s})\\
&\hspace{0pt} - \lambda_s^2\sum_{j=1}^m \frac{\umu_j}{m^2}\left(1 + \frac{\tau_p}{mC}\right)\|(S(x^{s-d_s}))_j\|_j^2  - 2\lambda_s^2\sum_{j=1}^m \frac{\umu_j}{q_jm^2}\left(1 + \frac{\tau_p}{C}\right)\left\|\frac{1}{n} \sum_{i=1}^n y_{i,j}^{s-e_{s}^{i}}  \right\|^2_j\\
&\hspace{0pt} - \sum_{j=1}^m\sum_{i=1}^n  p_{ij}\left(\frac{\gamma_{ij}}{\tau_d+1} - \frac{2\umu_j}{q_jm^2}\left(1 + \frac{\tau_p}{mC}\right)\lambda_s^2 \right)r_{ij}^d(y_i^{s-e_{s}^{i}}) \\
&\hspace{0pt} -  \sum_{j=1}^m \sum_{i=1}^n p_{ij}\left( \frac{2\lambda_s n^2p_{ij}\beta_{ij}}{m} - \umu_j\left(\frac{2}{q_jm^2} + \tau_p\left(\frac{2}{q_j m^3 C} + \frac{C}{m}\right)\right)\lambda_s^2 - \gamma_{ij} \right)r_{ij}^p(x^{s-d_s})  \\
&\hspace{0pt} -  \sum_{j=1}^m \sum_{i=1}^n   \sum_{\substack{h=0\\ h\neq \tau_d - e_{s,j}^{i} }}^{\tau_d}\frac{p_{ij}\gamma_{ij}}{\tau_d + 1}r_{ij}^d(y_i^{s - \tau_d +h}) \\
&=\|x^s - x^\ast\|^2 + \kappa_s - Y_s - \frac{2\alpha\lambda_s}{m}\dotp{S(x^{s-d_s}), x^{s-d_s} - x^\ast} + \sum_{j=1}^m \sum_{i=1}^n\frac{2\alpha\lambda_s n^2p_{ij}^2\beta_{ij}}{m}r_{ij}^p(x^{s-d_s}).
\end{align*}
Inequality~\eqref{eq:asyncfejerinequality} is simply the above inequality with $\alpha = 0$.
\qed\end{proof}

We finish the proof of Parts~\ref{thm:newalgoconvergence:part:operatorvalues} and~\ref{thm:newalgoconvergence:part:weak} by proving Lemmas~\ref{lem:AsyncFPRbound} and~\ref{lem:kappa}.

\begin{proof}[of Lemma~\ref{lem:AsyncFPRbound} (variance bound)]
Up to a change of indices, the proof of Lemma~\ref{lem:AsyncFPRbound} follows the proof Lemma~\ref{eq:syncvariance} exactly; we omit the proof and avoid repeating ourselves.
\qed\end{proof}

\begin{proof}[Lemma~\ref{lem:kappa} (recursive $\kappa_k$ bound)]

The bound is the addition of two further bounds:
for all $s \in \NN$ and $\alpha \in [0,1]$, we have
\begin{enumerate}
\item \textbf{Recursive $\kappa_{1, k}$ bound.}
\begin{align*}
\EE\left[ \kappa_{1, s+1} \mid \cF_s\right] &\leq \kappa_{1,s} -\sum_{j=1}^m\sum_{i=1}^n  \frac{p_{ij} \gamma_{ij}}{\tau_d + 1}r_{ij}^d(y_i^{s - e^{i}_{s}})  + \sum_{j=1}^m \sum_{i=1}^n p_{ij} \gamma_{ij}r_{ij}^p(x^{s-d_s}) \\
&-  \sum_{j=1}^m \sum_{i=1}^n  \sum_{\substack{h = 0\\ h\neq \tau_d - e^{i}_{s, j} }}^{\tau_d}\frac{p_{ij}\gamma_{ij}}{\tau_d + 1}r_{ij}^d(y_i^{s - \tau_d +h}); \text{ and}
\end{align*} 
\item \textbf{Recursive $\kappa_{2, k}$ bound.}
\begin{align*}
\kappa_{2,s+1} &\leq \kappa_{2, s} + \frac{2\lambda_s}{m} \dotp{S(x^{s-d_s}), x^s - x^\ast} - \frac{2\alpha\lambda_s}{m}\dotp{S(x^{s-d_s}), x^{s-d_s} - x^\ast} \\
&+ \sum_{j=1}^m \sum_{i=1}^n\frac{2\alpha\lambda_s n^2p_{ij}^2\beta_{ij}}{m}r_{ij}^p(x^{s-d_s})\\
&-\sum_{j=1}^m\sum_{i=1}^n p_{ij}\left( \frac{2\beta_{ij}\lambda_sn^2p_{ij}}{m} - \frac{\tau_p\umu_jC}{m} \lambda_s^2 \right)r_{ij}^p(x^{s-d_s}) +\sum_{j=1}^m\frac{\tau_p\umu_j}{mC}\|\mkQ^s_j\|_j^2.
\end{align*}
\end{enumerate}

The proofs of the above bounds are orthogonal; the $\kappa_{1, k}$ bound is a consequence of how we sample variables in SMART, while the $\kappa_{2, k}$ bound is a consequence of~\eqref{eq:coherence}. Thus, we separately prove each bound.

\textbf{Proof of recursive $\kappa_{1, k}$ bound.} For any $i$ and $j$ with $\mathbf{S}_{ij}^\ast \neq 0$, we have
$\EE\left[ r_{ij}^d(y_i^{s+1}) \mid \cF_s\right] = \left(1-\rho p_{ij}^T\right) r_{ij}^d(y_i^{s}) + \rho p_{ij}^T r_{ij}^p(x^{s-d_s})$
because $r_{ij}^d(y_i^{s+1})$ depends only on $y_{i,j}^{s+1}$, not on its other components, and the probability of update for $(y_i^k)_j$ is $P((i_k, i) \in E, t_{k,j} \neq 0, \epsilon_k = 1) = \rho p_{ij}^T$. We can also break up the sum:
\begin{align*}
&\sum_{h = 0}^{\tau_d-1} \sum_{j=1}^m \sum_{i=1}^n \frac{(h+1)p_{ij}\gamma_{ij}}{\tau_d + 1} r_{ij}^d(y_i^{s + 1 - \tau_d + h}) \\
&= \sum_{h = 0}^{\tau_d - 1} \sum_{j=1}^m \sum_{i=1}^n \frac{(h+1)p_{ij}\gamma_{ij}}{\tau_d + 1}r_{ij}^d(y_i^{s - \tau_d +h}) + \sum_{j=1}^m \sum_{i=1}^n \frac{\tau_dp_{ij}\gamma_{ij}}{\tau_d + 1}r_{ij}^d(y_i^{s}) \\
&\hspace{20pt} - \sum_{j=1}^m \sum_{i=1}^n\frac{p_{ij} \gamma_{ij}}{\tau_d + 1}r_{ij}^d(y_i^{s -e^{i}_{s}}) -  \sum_{j=1}^m \sum_{i=1}^n  \sum_{\substack{h = 0\\ h\neq \tau_d - e^{i}_{s, j} }}^{\tau_d-1}\frac{p_{ij}\gamma_{ij}}{\tau_d + 1}r_{ij}^d(y_i^{s - \tau_d +h}) \\
&= \sum_{h = 0}^{\tau_d - 1} \sum_{j=1}^m \sum_{i=1}^n \frac{(h+1)p_{ij}\gamma_{ij}}{\tau_d + 1}r_{ij}^d(y_i^{s - \tau_d +h}) + \sum_{j=1}^m \sum_{i=1}^n p_{ij}\gamma_{ij}r_{ij}^d(y_i^{s}) \\
&\hspace{20pt} - \sum_{j=1}^m \sum_{i=1}^n\frac{p_{ij} \gamma_{ij}}{\tau_d + 1}r_{ij}^d(y_i^{s -e^{i}_{s}}) -  \sum_{j=1}^m \sum_{i=1}^n  \sum_{\substack{h = 0\\ h\neq \tau_d - e^{i}_{s, j} }}^{\tau_d}\frac{p_{ij}\gamma_{ij}}{\tau_d + 1}r_{ij}^d(y_i^{s - \tau_d +h})
\end{align*}
Therefore, 
\begin{align*}
\EE\left[ \kappa_{1,s+1} \mid \cF_s\right] &=\sum_{j=1}^m \sum_{i=1}^n\frac{\gamma_{ij}p_{ij}}{\rho p_{ij}^T}\EE\left[ r_{ij}^d(y_i^{s+1}) \mid \cF_s\right] + \sum_{h=0}^{\tau_d - 1} \sum_{j=1}^m \sum_{i=1}^n \frac{(h+1)p_{ij}\gamma_{ij}}{\tau_d + 1} r_{ij}^d(y_i^{s + 1 - \tau_d + h})  \\
&\leq \sum_{j=1}^m \sum_{i=1}^n\left(\frac{\gamma_{ij}p_{ij}}{\rho p_{ij}^T} - p_{ij}\gamma_{ij}\right) r_{ij}^d(y_i^{s}) +  \sum_{j=1}^m \sum_{i=1}^np_{ij}\gamma_{ij} r_{ij}^p(x^{s-d_s})\\
&\hspace{20pt} +\sum_{h = 0}^{\tau_d - 1} \sum_{j=1}^m \sum_{i=1}^n \frac{(h+1)p_{ij}\gamma_{ij}}{\tau_d + 1}r_{ij}^d(y_i^{s - \tau_d +h}) + \sum_{j=1}^m \sum_{i=1}^n p_{ij}\gamma_{ij}r_{ij}^d(y_i^{s}) \\
&\hspace{20pt} - \sum_{j=1}^m \sum_{i=1}^n\frac{p_{ij} \gamma_{ij}}{\tau_d + 1}r_{ij}^d(y_i^{s -e^{i}_{s}}) -  \sum_{j=1}^m \sum_{i=1}^n  \sum_{\substack{h = 0\\ h\neq \tau_d - e^{i}_{s, j} }}^{\tau_d}\frac{p_{ij}\gamma_{ij}}{\tau_d + 1}r_{ij}^d(y_i^{s - \tau_d +h}),
\end{align*}
where the inequality uses the trivial bound: $ r_{ij}^p(x^{s-d_s})  \geq 0$ for all $i$ and $j$ with $\mathbf{S}_{ij}^\ast = 0$. The claimed inequality now follows by combinning terms.~\\

\textbf{Proof of the recursive $\kappa_{2, k}$ bound.} By Jensen's inequality, ($\sum_{i=1}^n p_{ij} = 1$)
\begin{align*}
&\sum_{j=1}^m \tau_p C\|(S(x^{s-d_s}))_j\|_j^2 = 
\sum_{j=1}^m \tau_p C\left\|\frac{1}{n}\sum_{i=1}^n\left[(S_i(x^{s-d_s}))_j - (S_i(x^\ast))_j\right]\right\|_j^2 \\
&=\sum_{j=1}^m \tau_p C\left\|\sum_{i=1}^np_{ij}\left[(Q_i^p(x^{s-d_s}))_j - (Q_{i}^\ast)_j\right]\right\|_j^2 \leq \sum_{j=1}^m \sum_{i=1}^np_{ij} \tau_p Cr_{ij}^p(x^{s-d_s}). 
\end{align*}
From~\eqref{eq:coherence}, 
 $\dotp{S(x^{s-d_s}), x^{s-d_s} - x^\ast} \geq  \sum_{j=1}^m\sum_{i=1}^n\beta_{ij}\|(S_{i}(x^{s-d_s}))_j - (S_{i}(x^\ast))_j\|_j^2$ \\$\geq  \sum_{j=1}^m\sum_{i=1}^n \beta_{ij}n^2p_{ij}^2r_{ij}^p(x^{s-d_s}).$ 
Because $x_j^h - x_j^{h+1} = \lambda_h\mkQ^h_j$ for all $h \in \NN$, 
\begin{align*}
&\sum_{j=1}^m\sum_{h = s-\tau_p+1}^{s}\frac{\umu_j}{C} \|x_j^{h} - x_j^{h-1}\|_j^2 = \sum_{j=1}^m \sum_{h=s-\tau_p+1}^s\frac{\umu_j(h - s + \tau_p)}{C} \|x_j^{h} - x_j^{h-1}\|_j^2 \\
&\hspace{0pt}- \sum_{j=1}^m \sum_{h=s-\tau_p+2}^{s+1}\frac{\umu_j(h - (s+1) + \tau_p)}{C} \|x_j^{h} - x_j^{h-1}\|_j^2 +  \lambda_s^2\sum_{j=1}^m\frac{\umu_j\tau_p}{C}\|\mkQ^s_j\|_j^2. \\
&= m\kappa_{2,s} - m\kappa_{2, s+1} + \lambda_s^2 \sum_{j=1}^m\frac{\umu_j\tau_p}{C}\|\mkQ^s_j\|_j^2. 
\end{align*}
To get the next bound, we apply the Cauchy-Schwarz inequality, Young's inequality, the equivalence of $\|\cdot\|$ and $\|\cdot\|_\pr$, and the convexity of $\|\cdot\|^2$---in that order: 
\begin{align*}
2\lambda_s &\dotp{S(x^{s-d_s}),\underbrace{\left(\sum_{h=s-d_{s,j}+1}^{s} (x_j^{h} - x_j^{h-1})\right)_{j=1}^m}_{:=T_1}} \geq -\|S(x^{s-d_s})\|\left\|T_1\right\| \\
&\geq-C\tau_p\lambda_s^2\|S(x^{s-d_s})\|^2 - \frac{1}{C\tau_p}\left\|T_1\right\|^2 \\
&\geq -C\tau_p\lambda_s^2\sum_{j=1}^n\umu_j\|(S(x^{s-d_s}))_j\|^2_{j}  - \frac{1}{C\tau_p}\sum_{j=1}^m\umu_j\left\|\sum_{h=s-d_{s,j}+1}^{s} (x_j^{h} - x_j^{h-1})\right\|_j^2\\
&\geq - \lambda_s^2\sum_{j=1}^m \tau_p C\umu_j\|(S(x^{s-d_s}))_j\|_j^2 -  \sum_{j=1}^m\sum_{h=s-\tau_p+1}^{s}\frac{\umu_j}{C} \|x_j^{h} - x_j^{h-1}\|_j^2.
\end{align*}

Therefore,
\begin{align*}
&2\lambda_s \dotp{S(x^{s-d_s}), x^s - x^\ast}  = 2\lambda_s \dotp{S(x^{s-d_s}), x^{s-d_s} - x^\ast} + 2\lambda_s \dotp{S(x^{s-d_s}),T_1}\\
&\geq 2\alpha\lambda_s \dotp{S(x^{s-d_s}), x^{s-d_s} - x^\ast} + 2(1-\alpha)\lambda_s\dotp{S(x^{s-d_s}), x^{s-d_s} - x^\ast}  \\
&\hspace{20pt} - \lambda_s^2\sum_{j=1}^m \tau_p C\umu_j\|(S(x^{s-d_s}))_j\|_j^2 -  \sum_{j=1}^m\sum_{h=s-\tau_p+1}^{s}\frac{\umu_j}{ C} \|x_j^{h} - x_j^{h-1}\|_j^2 \\
&\geq 2\alpha\lambda_s\dotp{S(x^{s-d_s}), x^{s-d_s} - x^\ast} - \sum_{j=1}^m \sum_{i=1}^n2\alpha\lambda_s n^2p_{ij}^2\beta_{ij}r_{ij}^p(x^{s-d_s}) \\
&\hspace{20pt} + \sum_{j=1}^m\sum_{i=1}^n p_{ij}\left( 2\lambda_s\beta_{ij}n^2p_{ij} - \tau_p C\lambda_s^2\right)r_{ij}^p(x^{s-d_s})  - \left(m\kappa_{2,s} - m\kappa_{2,s+1} + \lambda_s^2 \sum_{j=1}^m\frac{\tau_p}{C}\|\mkQ^s_j\|_j^2 \right).
\end{align*}
After a simple rearrangement of terms, this last bound completes the proof of the $\kappa_{2, k}$ bound, and, consequently, the proof of the lemma is complete. 
\qed\end{proof}

~\\

\textbf{Part~\ref{thm:newalgoconvergence:part:linear}}: 
In this part, we introduce two constants, called $T_{1,ij}$ and $T_{2, ij}$, which play the same role as $R_{1, ij}^k$ and $R_{2, ij}^k$ play in Parts~\ref{thm:newalgoconvergence:part:operatorvalues} and~\ref{thm:newalgoconvergence:part:weak} (in fact, $T_{1,ij} = R_{1, ij}^k$):
 \begin{align*}
 T_{1, ij} &:=  \frac{\gamma_{ij}}{\tau_d+1} - \frac{2\umu_j}{q_jm^2}\left(1 + \frac{\tau_p}{mC}\right)\lambda^2; \text{ and} \\
 T_{2, ij} &:=  \frac{2(1-\alpha)\lambda n^2p_{ij}\beta_{ij}}{m} - \umu_j\left(\frac{2}{q_jm^2} + \tau_p\left(\frac{2}{q_j m^3 C} + \frac{C}{m}\right)\right)\lambda^2 - \gamma_{ij} .
 \end{align*}
 Not only do we require that $\gamma_{ij}$ be chosen so that $T_{1, ij}$ and $T_{2, ij}$ are positive, but $\gamma_{ij}$ must be further constrained so that several other inequalities hold.  (We defer the proof of the lemma for a moment.)
\begin{lemma}[Choosing $\gamma_{ij}$]\label{eq:linearconstant}
Let $\lambda$ satisfy~\eqref{eq:asynclambdalinear}. Set 
\begin{align*}
\xi := \frac{2\alpha\mu\lambda}{m} && \text{ and } && \zeta :=  \frac{\mu\umu}{2.5\alpha}\left(1 + \frac{\tau_p\sqrt{\underline{q}}}{\sqrt{2(\tau_d +2)}}\right).
\end{align*} 
Then for all $i$ and $j$, there exists $\gamma_{ij} > 0$ such that 
 \begin{align*}
T_{1,ij} - \frac{\xi \gamma_{ij}}{\rho p_{ij}^T}  -  \frac{4\alpha\mu\delta \umu_j}{q_jm^3}\left(\frac{1}{\zeta } + \lambda\right) \lambda^2\geq 0; && && T_{2,ij} -  \frac{4\alpha\mu\delta \umu_j}{q_jm^3}\left(\frac{1}{\zeta } + \lambda\right) \lambda^2 \geq 0; \\
T_{1, ij}- \frac{2\umu_j\tau_p\lambda^2\xi}{Cq_jm^3(1-\xi)}  \geq 0;  && &&  T_{2, ij}- \frac{2\umu_j\tau_p\lambda^2\xi}{Cq_jm^3(1-\xi)}\geq 0. \numberthis \label{eq:nonnegativecoefficients}
 \end{align*}
 Moreover,  $1- \xi(\tau_d+1) \geq 0$. 
\end{lemma}

The expectation $\EE\left[\|x^k - x^\ast\| + \kappa_k\right]$ contracts, not every iteration, but over each series of $\tau_p$ iterations, and for this reason we will apply~\eqref{eq:asyncfejerinequalityalpha} multiple times, saving nonnegative summands from later iterations to absorb positive summands from earlier iterations which would otherwise, if not absorbed, prevent the expectation from contracting at all. 

We split our descent into three stages: \emph{initial descent}, in which we descend one iteration, from iteration $k+1$ to iteration $k$; \emph{intermediate descent}, in which we descend at most $\tau_p$ iterations, from iteration $k$ to iteration $k - h$ for some $h\in \{0, \ldots, \tau_p\}$; and \emph{final descent}, in which we descend one final iteration, from iteration $k- h$ to iteration $k- h - 1$. In each stage of descent, we apply~\eqref{eq:asyncfejerinequalityalpha} between $1$ and $\tau_p$ times---with the law of iterated expectations being never mentioned, but always applied to turn all conditional expectation inequalities into unconditional expectation inequalities.

We no longer work with arbitrary zeros $x^\ast \in \cS$. Instead we work with the sequence of zeros $P_{\cS}(x^k)$, which, by definition, satisfy \begin{align*}
d_{\cS}^2(x^{k+1}) &= \|x^{k+1} - P_{\cS}(x^{k+1})\|^2 \leq \|x^{k+1} - P_{\cS}(x^{k-h})\|^2; \text{ and}\\
d_{\cS}^2(x^{k-h})  &= \|x^{k-h} - P_{\cS}(x^{k-h})\|^2 \leq \|x^{k-h} - P_{\cS}(x^{k-h-1})\|^2.
\end{align*}
for any $h \in \NN$. Below, our choice of $h$ is precisely the maximal delay of any coordinate of $x^{k-d_k}$: let the integers $j_1, \ldots, j_m \in \{1, \ldots, m\}$ order the delays from smallest to largest: $d_{k,j_1} \leq \ldots \leq d_{k, j_m}$. With this notation, $\delta \geq d_{k, j_l} -  d_{k,j_m}$ for $l = 1, \ldots, m$, and we set $h = d_{k,j_m}$.

\textbf{Initial descent (from iteration $k+1$ to iteration $k$).} Apply~\eqref{eq:asyncfejerinequalityalpha} to descend to the $k$th iteration:
\begin{align*}
&\EE\left[d_\cS^2(x^{k+1}) + \kappa_{k+1} \right] \\
&\leq \EE\left[\|x^{k+1} -P_{\cS}(x^{k-d_{k,j_m}} )\|^2 + \kappa_{k+1} \right] \qquad \text{(by definition of $P_{\cS}$)} \\
&\leq \EE\left[\|x^k - P_{\cS}(x^{k-d_{k,j_m}} )\|^2 + \kappa_s - \frac{2\alpha\lambda}{m}\dotp{S(x^{k-d_k}), x^{k-d_k} - P_{\cS}(x^{k-d_{k,j_m}} )}\right] \\
&\hspace{20pt} -\lambda^2\sum_{j=1}^m \frac{\umu_j}{m^2}\left(1 + \frac{\tau_p}{mC}\right)\EE\left[\|(S(x^{k-d_k}))_j\|_j^2\right] -\sum_{j=1}^m\sum_{i=1}^n  p_{ij}T_{1, ij}\EE\left[r_{ij}^d(y_i^{k-e^{i}_{k}})\right]   \\
&\hspace{20pt} -  \sum_{j=1}^m \sum_{i=1}^n p_{ij}T_{2, ij}\EE\left[r_{ij}^p(x^{k-d_k})\right] -  \sum_{j=1}^m \sum_{i=1}^n   \sum_{\substack{h=0\\ h\neq \tau_d - e^{i}_{k,j} }}^{\tau_d}\frac{p_{ij}\gamma_{ij}}{\tau_d + 1}\EE\left[r_{ij}^d(y_i^{k - \tau_d +h})\right]. \numberthis\label{eq:initialdescent}
\end{align*}

\textbf{Intermediate descent (from iteration $k$ to iteration $k-d_{k,j_m}$).} To bound our initial descent~\eqref{eq:initialdescent}, apply~\eqref{eq:asyncfejerinequalityalpha} a total of $d_{k,j_m}$ times:
\begin{align*}
&\eqref{eq:initialdescent} \leq \EE\left[\|x^{k-d_{k,j_m}} - P_{\cS}(x^{k-d_{k,j_m}})\|^2 + \kappa_{k-d_{k, j_m}} - \frac{2\alpha\lambda}{m}\dotp{S(x^{k-d_k}), x^{k-d_k} - P_{\cS}(x^{k-d_{k,j_m}})} \right] \\
&\hspace{0pt}-\lambda^2\sum_{j=1}^m \frac{\umu_j}{m^2}\left(1 + \frac{\tau_p}{mC}\right)\EE\left[\|(S(x^{k-d_k}))_j\|_j^2\right]  -\sum_{h=k-d_{k,j_m}}^k\sum_{j=1}^m\sum_{i=1}^n  p_{ij}T_{1, ij}\EE\left[r_{ij}^d(y_i^{h-e^{i}_{h}})\right]   \\
& \hspace{0pt}- \sum_{h=k-d_{k,j_m}}^k  \sum_{j=1}^m \sum_{i=1}^n p_{ij}T_{2, ij}\EE\left[r_{ij}^p(x^{h- d_{h}})\right] -  \sum_{j=1}^m \sum_{i=1}^n   \sum_{\substack{h=0\\ h\neq \tau_d - e^{j}_{k-d_{k,j_m}, j}}}^{\tau_d}\frac{p_{ij}\gamma_{ij}}{\tau_d + 1}\EE\left[r_{ij}^d(y_i^{k -d_{k,j_m}- \tau_d +h})\right].\numberthis\label{eq:intermediate_descent_1}
\end{align*}
As we descended to the $(k-d_{k,j_m})$th iteration, we dropped a lot of negative terms including $-2\alpha\lambda\dotp{S(x^{s-d_s}), x^{s-d_s} - P_{\cS}(x^{k-d_{k,j_m}} )}$ for all $s \in \{k-1, \ldots, k-d_{k,j_m}\}$. To show that the expectation contracts, we do not need these terms, so we omit them. 

We need to rearrange the terms in~\eqref{eq:intermediate_descent_1} in order to derive a contraction. First we squeeze a norm term out of the inner product term---this is one of only two times we use the essential strong quasi-monotonicity of $S$:
\begin{align*}
&\dotp{S(x^{k-d_k}), x^{k-d_k} - P_{\cS}(x^{k-d_{k,j_m}})} \\
&= \dotp{S(x^{k-d_k}), x^{k-d_k} -P_{\cS}(x^{k-d_k})} + \dotp{S(x^{k-d_k}), P_{\cS}(x^{k-d_k}) -P_{\cS}(x^{k-d_{k,j_m}} )} \\
&\geq \mu\|x^{k-d_k} - P_{\cS}(x^{k-d_k})\|^2 +  \dotp{S(x^{k-d_k}), P_{\cS}(x^{k-d_k}) -P_{\cS}(x^{k-d_{k,j_m}} )}.
\end{align*}
Then we simply split up the term
$\|x^{k-d_{k,j_m}} - P_{\cS}(x^{k-d_{k,j_m}} )\|^2 = \left(1-2\alpha\mu\lambda m^{-1} \right)\|x^{k-d_{k,j_m}} - P_{\cS}(x^{k-d_{k,j_m}} )\|^2 + 2\alpha\mu\lambda m^{-1}\|x^{k-d_{k,j_m}} - P_{\cS}(x^{k-d_{k,j_m}} )\|^2$
and couple it with $\mu\|x^{k-d_k} - P_\cS(x^{k-d_k})\|^2$:
\begin{align*}
&\eqref{eq:intermediate_descent_1} \leq \EE\left[\left(1- \frac{2\mu\alpha\lambda}{m}\right)\|x^{k-d_{k,j_m}} - P_{\cS}(x^{k-d_{k,j_m}} )\|^2 + \kappa_{k-d_{k, j_m}} \right] \\
&\hspace{0pt}-\lambda^2\sum_{j=1}^m \frac{\umu_j}{m^2}\left(1 + \frac{\tau_p}{mC}\right)\EE\left[\|(S(x^{k-d_k}))_j\|_j^2\right]  -\sum_{h=k-d_{k,j_m}}^k\sum_{j=1}^m\sum_{i=1}^n  p_{ij}T_{1, ij}\EE\left[r_{ij}^d(y_i^{h-e^{i}_{h}})\right]  \\
& \hspace{0pt}- \sum_{h=k-d_{k,j_m}}^k  \sum_{j=1}^m \sum_{i=1}^n p_{ij}T_{2, ij}\EE\left[r_{ij}^p(x^{h - d_{h}})\right] \\
&-  \sum_{j=1}^m \sum_{i=1}^n   \sum_{\substack{h=0\\ h\neq \tau_d - e^{i}_{k-d_{k,j_m}, j}}}^{\tau_d}\frac{p_{ij}\gamma_{ij}}{\tau_d + 1}\EE\left[r_{ij}^d(y_i^{k -d_{k,j_m}- \tau_d +h})\right] \\
&\hspace{0pt}+\frac{2\alpha\mu\lambda}{m}\EE\left[\|x^{k-d_{k,j_m}} - P_{\cS}(x^{k-d_{k,j_m}} )\|^2 - \|x^{k-d_{k}} - P_\cS(x^{k-d_k})\|^2\right] \\
&\hspace{0pt}- \frac{2\alpha\lambda}{m}\EE\left[\dotp{S(x^{k-d_k}), P_\cS(x^{k-d_k}) - P_{\cS}(x^{k-d_{k,j_m}} )}\right]. \numberthis\label{eq:intermediate_descent_2}
\end{align*}
The final two lines of the above equations are preventing the expectation from contracting, so we absorb them into the other terms; if $\delta$ is $0$, these terms are $0$, as $x^{k- d_{k,j_m}} = x^{k-d_k}$, and there is no need for the next lemma.  (We defer the proof of the lemma for a moment.)
\begin{lemma}[Swapping zeros in the inconsistent case]\label{lem:altbound}
The following bound holds:
\begin{align*}
&\frac{2\alpha\mu\lambda}{m}\EE\left[\|x^{k-d_{k,j_m}} - P_{\cS}(x^{k-d_{k,j_m}} )\|^2 - \|x^{k-d_{k}} - P_\cS(x^{k-d_k})\|^2\right] \\
&\hspace{20pt}- \frac{2\alpha\lambda}{m}\EE\left[\dotp{S(x^{k-d_k}), P_\cS(x^{k-d_k}) - P_{\cS}(x^{k-d_{k,j_m}} )}\right] \\
&\leq 4\alpha\mu\lambda^3\sum_{h=k-d_{k, j_m}}^k\sum_{j = 1}^m \sum_{i=1}^n \frac{\delta p_{ij}\umu_j}{q_jm^3}\left(\frac{1}{\zeta\lambda} + 1\right)\EE\left[r_{ij}^p(x^{h-d_h})\right] \\
&\hspace{20pt} +4\alpha\mu\lambda^3\sum_{h=k-d_{k, j_m}}^k\sum_{j = 1}^m \sum_{i=1}^n \frac{\delta p_{ij}\umu_j}{q_jm^3}\left(\frac{1}{\zeta\lambda} + 1\right)\EE\left[r_{ij}^d(y_i^{h-e^{i}_{h}})\right]\\
&\hspace{20pt} +\lambda^2\sum_{j=1}^m\frac{\umu_j}{m^2}\left(1 + \frac{\tau_p}{mC}\right)\EE\left[\|(S(x^{k-d_k}))_j\|_j^2\right].
\end{align*}
\end{lemma}
Lemma~\ref{lem:altbound} takes care of the offending terms in~\eqref{eq:intermediate_descent_2}:
\begin{align*}
\eqref{eq:intermediate_descent_2} &\leq \EE\left[\left(1- \frac{2\mu\alpha\lambda}{m} \right)\|x^{k-d_{k,j_m}} - P_{\cS}(x^{k-d_{k,j_m}})\|^2 + \kappa_{k-d_{k, j_m}} \right] \\
&\hspace{20pt} -\sum_{h=k-d_{k,j_m}}^k\sum_{j=1}^m\sum_{i=1}^n  p_{ij}\left(T_{1,ij} -  \frac{4\alpha\mu\delta \umu_j}{q_jm^3}\left(\frac{1}{\zeta } + \lambda\right) \lambda^2 \right)\EE\left[r_{ij}^d(y_i^{h-e^{i}_{h}})\right]\\
&\hspace{20pt}  -\sum_{h=k-d_{k,j_m}}^k \sum_{j=1}^m \sum_{i=1}^n p_{ij}\left(T_{2,ij} -  \frac{4\alpha\mu\delta \umu_j}{q_jm^3}\left(\frac{1}{\zeta } + \lambda\right) \lambda^2  \right)\EE\left[r_{ij}^p(x^{h - d_{h}})\right] \\
& \hspace{20pt}-  \sum_{j=1}^m \sum_{i=1}^n   \sum_{\substack{h = 0\\ h\neq \tau_d - e^{i}_{k-d_{k,j_m},j} }}^{\tau_d}\frac{p_{ij}\gamma_{ij}}{\tau_d + 1}\EE\left[r_{ij}^d(y_i^{k - d_{k,j_m}- \tau_d +h})\right] .\numberthis\label{eq:intermediate_descent_3}
\end{align*}
Our intermediate descent is completed simply by writing $\kappa_{k-d_{k,j_m}} = (1-\xi) \kappa_{k-d_{k,j_m}}  + \xi\kappa_{k-d_{k,j_m}}$ and absorbing the $\xi\kappa_{1, k-d_{k, j_m}}$ terms in the second and fourth lines of the above inequality (we add $\xi\gamma_{ij}(\rho p_{ij}^T)^{-1}$ to all terms on the second line even though it only appears in the $h = k - d_{k,j_m}$ term in $\xi\kappa_{1,k-d_{k, j_m}}$):
\begin{align*}
\eqref{eq:intermediate_descent_3}& \leq\EE\left[\left(1-\xi\right)\left(\|x^{k-d_{k,j_m}} - P_{\cS}(x^{k-d_{k,j_m}})\|^2 + \kappa_{k-d_{k,j_m}}\right) + \xi\kappa_{2, k-d_{k,j_m}} \right] \\
&\hspace{20pt} -\sum_{h=k-d_{k,j_m}}^k\sum_{j=1}^m\sum_{i=1}^n  p_{ij}\left(T_{1,ij} - \frac{\xi \gamma_{ij}}{\rho p_{ij}^T}  -  \frac{4\alpha\mu\delta \umu_j}{q_jm^3}\left(\frac{1}{\zeta } + \lambda\right) \lambda^2\right)\EE\left[r_{ij}^d(y_i^{h-e^{i}_{h}})\right] \\
&\hspace{20pt}  -\sum_{t=k-d_{k,j_m}}^k \sum_{j=1}^m \sum_{i=1}^n p_{ij}\left( T_{2, ij} -  \frac{4\alpha\mu\delta \umu_j}{q_jm^3}\left(\frac{1}{\zeta } + \lambda\right) \lambda^2 \right)\EE\left[r_{ij}^p(x^{h - d_{h}})\right] \\
&\hspace{20pt} -  \sum_{j=1}^m \sum_{i=1}^n   \sum_{\substack{h=0\\ h\neq \tau_d - e^{i}_{k-d_{k,j_m}, j} }}^{\tau_d}\frac{p_{ij}\gamma_{ij}}{\tau_d + 1}\left( 1- \xi(h+1)\right)\EE\left[r_{ij}^d(y_i^{k - d_{k,j_m}- \tau_d +h})\right] .\numberthis\label{eq:intermediate_descent_4}
\end{align*}

\textbf{Final descent (from iteration $k-d_{k,j_m}$ to iteration $k-d_{k,j_m}-1$).}  The expectation is nearly contracting by a factor of $(1-\xi)$. The term $\xi\kappa_{2, k-d_{k,j_m}}$ is all that stands in our way, but it can, by descending one more iteration, be absorbed.  (We defer the proof of the lemma for a moment.)
\begin{lemma}[Recursive $\kappa_{2, k}$ bound]\label{eq:kappa2bound}
For all $s\in \NN$, 
\begin{align*}
\EE\left[\kappa_{2,s}\right] &\leq \lambda^2 \sum_{j=1}^m \sum_{i=1}^n \frac{2\umu_j \tau_p p_{ij}}{Cq_jm^3}\EE\left[r_{ij}^p(x^{s-1 - d_{s-1}})\right]  \\
&\hspace{20pt}+ \lambda^2 \sum_{j=1}^m \sum_{i=1}^n \frac{2\umu_j \tau_p p_{ij}}{Cq_jm^3}\EE\left[r_{ij}^d(y_i^{s-1 - e^{i}_{s-1}})\right]+  \frac{\tau_p}{\tau_p+1} \EE\left[ \kappa_{2,s-1}\right].
\end{align*}
\end{lemma}
By Equation~\eqref{eq:nonnegativecoefficients}, the last three lines in~\eqref{eq:intermediate_descent_4} are nonnegative; drop these nonnegative terms, use the bound $\|x^{k-d_{k,j_m}} - P_{\cS}(x^{k-d_{k,j_m}})\|^2 \leq \|x^{k-d_{k,j_m}} - P_{\cS}(x^{k-d_{k,j_m}-1})\|^2$, descend one more step with the aid of~\eqref{eq:asyncfejerinequalityalpha}, and use Lemma~\ref{eq:kappa2bound} to bound $\EE\left[\xi\kappa_{2, k-d_{k,j_m}}\right]$---in that order: if $h = k-d_{k,j_m}-1$, then
\begin{align*}
\eqref{eq:intermediate_descent_4} &\leq \EE\left[(1-\xi)\left(\|x^{h} - P_{\cS}(x^{h})\|^2 + \kappa_{h}\right) + \frac{\tau_p}{\tau_p+1}\xi\kappa_{2,h} \right] \\
&\hspace{20pt} - \left(1-\xi\right) \sum_{j=1}^m\sum_{i=1}^n  p_{ij}T_{1, ij}\EE\left[r_{ij}^d(y_i^{h-e^{i}_{h}})\right] +\xi\sum_{j=1}^m\sum_{i=1}^n \frac{2\umu_j \tau_p p_{ij}\lambda^2}{Cq_jm^3}\EE\left[r_{ij}^d(y_i^{h - e^{i}_{h}})\right] \\
&\hspace{20pt} - \left(1-\xi\right)  \sum_{j=1}^m \sum_{i=1}^n p_{ij}T_{2, ij}\EE\left[r_{ij}^p(x^{h - d_{h}})\right]  +\xi \sum_{j=1}^m \sum_{i=1}^n \frac{2\umu_j \tau_p p_{ij}\lambda^2}{Cq_jm^3}\EE\left[r_{ij}^p(x^{h -d_{h}})\right] \\
&\leq\EE\left[\left(1-\frac{\xi}{\tau_p+1}\right)\left(\|x^{h} - P_{\cS}(x^{h})\|^2 + \kappa_{h}\right) \right] \\
&\hspace{20pt} - \left(1-\xi\right) \sum_{j=1}^m\sum_{i=1}^n  p_{ij}\left(T_{1, ij} -  \frac{2\umu_j\tau_p\lambda^2\xi}{Cq_jm^3(1-\xi)} \right)\EE\left[r_{ij}^d(y_i^{h-e^{i}_{h}})\right]\\
&\hspace{20pt} - \left(1-\xi\right)  \sum_{j=1}^m \sum_{i=1}^n p_{ij}\left( T_{2, ij}- \frac{2\umu_j\tau_p\lambda^2\xi}{Cq_jm^3(1-\xi)} \right)\EE\left[r_{ij}^p(x^{h - d_{h}})\right].
\end{align*}
Equation~\eqref{eq:nonnegativecoefficients} implies that the last two lines of the equation are negative, so 
\begin{align*}
\EE\left[d_\cS(x^{k+1})^2 + \kappa_{k+1} \right] &\leq  \left(1-\frac{\xi}{\tau_p+1}\right)\EE\left[d_\cS^2(x^{k-d_{k,j_m} - 1})  + \kappa_{k+1} \right] \\
&\leq \left(1-\frac{\xi}{\tau_p+1}\right)\max_{h \in  \{k-\tau_p - 1, \ldots, k-1\}}\left\{\EE\left[d_\cS^2(x^{h})+ \kappa_{h}\right]\right\}. 
\end{align*}
holds for all $k$. Set $C(z^0, \phi^0) := \kappa_0$ and unfold these recursive bounds with the next lemma to get the rate. 
\begin{lemma}[Linear convergence rate of sequences that contract within a fixed number of steps]{\cite[Lemma 6]{peng2015arock}}]
Let $\{a_k\}_{k \in \NN}$ be a nonnegative sequence of real numbers. Suppose that there exists $\rho \in [0, 1)$ and $\kappa  \in \NN$ such that for all $k \in \NN$, we have
$a_{k+1} \leq \rho\max_{k - \kappa \leq h \leq k} a_k.$
Then for all $s \in \NN$, we have
$a_{k} \leq \rho^{k/(\kappa+1)}a_0.$
\end{lemma}

We finish the proof of Part~\ref{thm:newalgoconvergence:part:linear} by proving Lemmas~\ref{eq:linearconstant}, \ref{lem:altbound}, and \ref{eq:kappa2bound}.

\begin{proof}[of Lemma~\ref{eq:linearconstant} (choosing $\gamma_{ij}$)]
We assume that $\lambda$ satisfies~\eqref{eq:asynclambdalinear}
\begin{align*}
\lambda &\leq \min_{i,j}\left\{\frac{2\eta(1-\alpha)n^2p_{ij}\beta_{ij}}{\frac{2\umu_j\eta(\tau_d + 2)}{q_jm}\left(1 + \frac{\delta\eta}{\tau_d + 1} + \frac{5\sqrt{2(\tau_d + 2)}\alpha^2\delta}{m\umu \left( \sqrt{2(\tau_d + 2)} + \tau_p \sqrt{\underline{q}} \right) }\right) + \frac{\umu_j\eta\tau_p\sqrt{2(\tau_d + 2)}}{m\sqrt{\underline{q}}} \left(2+\frac{\eta}{1-\eta}\right) + 4\mu(\tau_d+1)\alpha(1-\alpha)n^2p_{ij}\beta_{ij}} \right\},
\end{align*}
and consequently,\footnote{Use
$2\mu\alpha \delta(\zeta m)^{-1} = 5\alpha^2\delta\left(m\umu\left(1 + \frac{\tau_p \sqrt{\underline{q}}}{\sqrt{2(\tau_d +2)}}\right)\right)^{-1} = \frac{5\sqrt{2(\tau_d + 2)}\alpha^2\delta}{m\umu \left( \sqrt{2(\tau_d + 2)} + \tau_p \sqrt{\underline{q}} \right) }$.}
\begin{align*}
&\text{if} \quad w_{ij} := \frac{\frac{2(1-\alpha)n^2p_{ij}\beta_{ij}}{\frac{2\umu_j(\tau_d + 2)}{q_jm}\left(1 + \frac{\delta\eta}{\tau_d + 1} +  \frac{2\mu\alpha \delta}{m\zeta }\right) + \umu_j\tau_p C\left(2+\frac{\eta}{1-\eta}\right)}}{\frac{m\eta}{2\alpha\mu(\tau_d+1)}+\frac{2(1-\alpha)n^2p_{ij}\beta_{ij}}{\frac{2\umu_j(\tau_d + 2)}{q_jm}\left(1 + \frac{\delta\eta}{\tau_d + 1} +  \frac{2\mu\alpha \delta}{m\zeta }\right) + \umu_j\tau_p C\left(2+\frac{\eta}{1-\eta}\right)}},\\
&\text{then} \quad 
\lambda \leq \frac{w_{ij}m\eta}{2\alpha\mu (\tau_d+1)} = \frac{2(1-\alpha)n^2p_{ij}\beta_{ij}(1-w_{ij})}{\frac{2\umu_j(\tau_d + 2)}{q_jm}\left(1 + \frac{\delta\eta}{\tau_d + 1} +  \frac{2\mu\alpha \delta}{m\zeta }\right) + \umu_j\tau_p C\left(2+\frac{\eta}{1-\eta}\right)},\numberthis\label{eq:asynclambdaebound}
\end{align*}
which implies the bound: $\xi(\tau_d+1) = 2m^{-1}\alpha\mu\lambda(\tau_d + 1) \leq w_{ij}\eta \leq 1$. 

The four inequalities that remain hold when there exists $\gamma_{ij} > 0$ such that:
\begin{align*}
\frac{(\tau_d + 1)}{\left(1-\frac{(\tau_d+1)\xi }{\rho p_{ij}^T}\right)}\underbrace{\left(-T_{1, ij} + \frac{\gamma_{ij}}{\tau_d+1}+\frac{2\umu_j\tau_p\xi\lambda^2}{Cq_jm^3(1-\xi)} +\frac{4\alpha\mu\delta \umu_j\lambda^2}{q_jm^3}\left(\frac{1}{\zeta} + \lambda  \right)\right)}_{:=R_1} &\leq \gamma_{ij}; \text{ and}\\
\underbrace{T_{2, ij}  - \frac{2\umu_j\tau_p\lambda^2\xi}{Cq_jm^3(1-\xi)} -\frac{4\alpha\mu\delta \umu_j}{q_jm^3}\left(\frac{1}{\zeta} + \lambda\right)  + \gamma_{ij}}_{:=R_2}&\geq \gamma_{ij}. \numberthis\label{eq:themastequation}
\end{align*}
Existence of such $\gamma_{ij}$ is implied by the following inequality:
$(\tau_d + 1)\left(1-\frac{(\tau_d+1)\xi }{\rho p_{ij}^T}\right)^{-1}R_1\leq R_2$
Furthermore, the bound $\left(1-\frac{(\tau_d+1)\xi }{\rho p_{ij}^T}\right) \geq \left(1 - \frac{w_{ij}\eta }{\rho p_{ij}^T}\right)  \geq \left(1-w_{ij}\right)$ shows that~\eqref{eq:themastequation} is implied by the following inequality: $(\tau_d + 1)R_1 \leq (1-w_{ij})R_2.$

Thus, we finish the proof by proving the above bound: Because $w_{ij} \leq 1$, $\xi \leq w_{ij}\eta < 1$, and $\xi/(1-\xi) \leq w_{ij}\eta/(1-w_{ij}\eta) \leq \eta/(1-\eta)$, we have
{\footnotesize
\begin{align*}
&\lambda^2\biggl[ (\tau_d + 1)\left(\frac{2\umu_j}{q_jm^2}\left(1 + \frac{\tau_p}{mC}\right) + \frac{2\umu_j\tau_p\xi}{Cq_jm^3(1-\xi)} +\frac{4\alpha\mu\delta \umu_j}{q_jm^3}\left(\frac{1}{\zeta} + \lambda\right)\right) \\
&\hspace{20pt}+ (1-w_{ij})\left(\umu_j\left(\frac{2}{q_jm^2} + \tau_p\left(\frac{2}{q_jm^3C} + \frac{C}{m}\right)\right) + \frac{2\umu_j\tau_p\lambda^2\xi}{Cq_jm^3(1-\xi)} +\frac{4\alpha\mu\delta \umu_j}{q_jm^3}\left(\frac{1}{\zeta} + \lambda\right)\right)\biggr] \\ 
&\leq\lambda^2\biggl[ (\tau_d + 1)\left(\frac{2\umu_j}{\underline{q}m^2}\left(1 + \frac{\tau_p}{mC}\right) + \frac{2\umu_j\tau_p\xi}{C\underline{q}m^3(1-\xi)} +\frac{4\alpha\mu\delta \umu_j}{\underline{q}m^3}\left(\frac{1}{\zeta} + \lambda\right)\right) \\
&\hspace{20pt}+ (1-w_{ij})\left(\umu_j\left(\frac{2}{\underline{q}m^2} + \tau_p\left(\frac{2}{\underline{q}m^3C} + \frac{C}{m}\right)\right) + \frac{2\umu_j\tau_p\lambda^2\xi}{C\underline{q}m^3(1-\xi)} +\frac{4\alpha\mu\delta \umu_j}{\underline{q}m^3}\left(\frac{1}{\zeta} + \lambda\right)\right)\biggr] \\ 
&\leq\umu_j\lambda^2\biggl[ (\tau_d + 2)\left(\frac{2}{\underline{q}m^2}\left(1 + \frac{\tau_p}{mC}\right) + \frac{2\tau_p\xi}{C\underline{q}m^3(1-\xi)} +\frac{4\alpha\mu\delta }{\underline{q}m^3}\left(\frac{1}{\zeta} + \lambda\right)\right) + \frac{\tau_pC}{m} \biggr] \\
&\leq \umu_j\lambda^2\left[(\tau_d + 2)\left(\frac{2}{\underline{q}m^2} +\frac{4\alpha\mu\delta }{\underline{q}m^3}\left(\frac{1}{\zeta} + \lambda\right)\right) + \frac{\tau_pC}{m} \left(2+\frac{\eta}{1-\eta}\right)\right] \quad \left(\text{b/c } \frac{2(\tau_d + 2)}{m^2\underline{q}C} = C\right) \\
&\leq \umu_j\lambda^2\left[\frac{2(\tau_d + 2)}{\underline{q}m^2}\left(1 + \frac{\delta\eta}{\tau_d + 1} + \frac{2\mu\alpha \delta }{m\zeta }\right) + \frac{\tau_pC}{m} \left(2+\frac{\eta}{1-\eta}\right)\right]  \quad\left( \text{b/c } \frac{2\alpha\mu\lambda}{m} \leq \frac{\eta}{\tau_d + 1}\right)\\
&\stackrel{\eqref{eq:asynclambdaebound}}{\leq} \frac{2(1-\alpha) n^2 \lambda p_{ij} \beta_{ij}(1-w_{ij})}{m}. \qquad \qed
\end{align*}
}
%
 \end{proof}

\begin{proof}[of Lemma~\ref{lem:altbound} (swapping zeros in the inconsistent case)]
We split the proof according to the bounds we apply:
\begin{enumerate}
\item \textbf{Firm nonexpansiveness.} Because $\cS$ is convex, the map $P_\cS$ is firmly nonexpansive in $\|\cdot\|$. So for all $x, y \in \cH$, $\|P_\cS(x) - P_\cS(y)\|^2  + \|(x- P_\cS(x)) - (y- P_\cS(y))\|^2 \leq \|x- y\|^2$. Therefore,   
\begin{align*}
&\underbrace{\|\left(x^{k-d_{k,j_m}} - P_{\cS}(x^{k-d_{k,j_m}} )\right) - \left(x^{k-d_{k}} - P_\cS(x^{k-d_k})\right)\|^2}_{:= R_1}  \\
&\leq \underbrace{\|x^{k-d_{k,j_m}} - x^{k-d_k}\|^2}_{:= R_2} - \underbrace{\| P_{\cS}(x^{k-d_{k,j_m}}) - P_\cS(x^{k-d_k})\|^2}_{:= R_3}.\numberthis\label{eq:firm-nonexpansive-bound}
\end{align*}
\item \textbf{Young's inequality.} Apply the bound $\|a + b\|^2 \leq (1+\varepsilon)\|a\|^2 + (1+\varepsilon^{-1})\|b\|^2$ to get
\begin{align*}
\underbrace{\|x^{k-d_{k,j_m}} - P_{\cS}(x^{k-d_{k,j_m}} )\|^2}_{:=R_4} - \underbrace{\|x^{k-d_{k}} - P_\cS(x^{k-d_k})\|^2}_{:=R_5} &\leq\left(\frac{1}{\zeta \lambda} + 1\right)R_1 + \zeta\lambda R_5\\
&\hspace{-10pt}\stackrel{\eqref{eq:firm-nonexpansive-bound}}{\leq} \left(\frac{1}{\zeta \lambda} + 1\right)\left(R_2 - R_3\right) + \zeta\lambda R_5.
\end{align*}
\item \textbf{Equivalence of norms, convexity of $\|\cdot\|^2$, and the variance bound.}
\begin{align*}
&\hspace{-5pt}\EE\left[R_2\right] \leq \sum_{j=1}^m\EE\left[\umu_j\|x^{k-d_{k,j_m}}_j - x^{k-d_{k,j}}_j\|_j^2\right] \leq \sum_{j=1}^m\EE\left[\umu_j\delta\sum_{h=k-d_{k, j_m}+1}^{k-d_{k,j}}\| x_{j}^{h} - x_{j}^{h-1}\|_{j}^2 \right] \\
&\hspace{-10pt}\leq \sum_{j=1}^m\EE\left[\umu_j\delta\sum_{h=k-d_{k, j_m}+1}^{k+1}\| x_{j}^{h} - x_{j}^{h-1}\|_{j}^2 \right] \stackrel{\eqref{eq:ASYNCSAGAFPRBOUND}}{\leq} 2\lambda^2\sum_{h=k-d_{k, j_m}}^k\sum_{j = 1}^m \sum_{i=1}^n \frac{\delta p_{ij}\umu_j}{q_jm^2}\EE\left[r_{ij}^p(x^{h-d_h})\right] \\
&\hspace{170pt}+2\lambda^2\sum_{h=k-d_{k, j_m}}^k\sum_{j = 1}^m \sum_{i=1}^n \frac{\delta p_{ij}\umu_j}{q_jm^2} \left[r_{ij}^d(y_i^{h-e^{i}_{h}})\right]. 
\end{align*}
\item \textbf{Essential strong quasi-monotonicity.} The next two bounds hold
\begin{align*}
\dotp{S(x^{k-d_k}), x^{k-d_k} -P_{\cS}(x^{k-d_k})} \geq \mu R_5 && \text{and} && \frac{1}{\mu^2}\|S(x^{k-d_k})\|^2\geq R_5.
\end{align*} 
(After an application of the Cauchy-Schwarz inequality, the second bound follows from the first.) Thus,
$2\alpha\mu \zeta \lambda^2m^{-1}R_5 \leq 2\alpha\zeta\lambda^2(m\mu)^{-1}\|S(x^{k-d_k})\|^2.$
\item \textbf{Cauchy-Schwarz and Young's Inequality.}
\begin{align*}
&\frac{2\alpha\lambda}{m}|\dotp{S(x^{k-d_k}), P_\cS(x^{k-d_k}) - P_{\cS}(x^{k-d_{k,j_m}})}| \leq \frac{\alpha \zeta\lambda^2}{2m\mu}\|S(x^{k-d_k})\|^2 + \frac{2\alpha\mu}{m\zeta}R_3. 
\end{align*}
\item \textbf{The $\zeta$ bound.} 
\begin{align*}
&\frac{2\alpha\mu \zeta \lambda^2}{m}\EE\left[R_5 \right] - \frac{2\alpha\mu\lambda}{m}\left(\frac{1}{\zeta \lambda} + 1\right)\EE\left[R_3\right] 
-\frac{2\alpha\lambda}{m}\EE\left[\dotp{S(x^{k-d_k}),  P_\cS(x^{k-d_k}) - P_{\cS}(x^{k-d_{k,j_m}})}\right] \\
&\leq \frac{2.5\alpha\zeta\lambda^2}{m\mu}\EE\left[\|S(x^{k-d_k})\|^2\right] -  \left(\frac{2\alpha\mu}{m\zeta } + \frac{2\alpha\mu\lambda}{m} - \frac{2\alpha\mu}{m\zeta}\right) \EE\left[R_3\right] \\
&\leq \lambda^2\sum_{j=1}^m\frac{\umu_j}{m^2}\left(1 + \frac{\tau_p}{mC}\right)\EE\left[\|(S(x^{k-d_k}))_j\|_j^2\right]. 
\end{align*}
(The third line follows from the second because (i) $2.5\alpha\zeta(m\mu)^{-1} \leq \umu_j\left(1+\tau_p (mC)^{-1}\right)$ and (ii) the second term on line three is negative.) 
\end{enumerate}
The following bounds complete the proof.
\begin{align*}
&\frac{2\alpha\mu\lambda}{m}\EE\left[R_4 - R_5\right] - \frac{2\alpha\lambda}{m}\EE\left[\dotp{S(x^{k-d_k}), P_\cS(x^{k-d_k}) - P_{\cS}(x^{k-d_{k,j_m}} )}\right] \\
&\leq \frac{2\alpha\mu\lambda}{m}\left(\frac{1}{\zeta \lambda} + 1\right)\EE\left[R_2\right]  - \frac{2\alpha\lambda}{m}\EE\left[\dotp{S(x^{k-d_k}),  P_\cS(x^{k-d_k}) - P_{\cS}(x^{k-d_{k,j_m}})}\right] \\
&+\frac{2\alpha\mu \zeta \lambda^2}{m}\EE\left[R_5 \right] - \frac{2\alpha\mu\lambda}{m}\left(\frac{1}{\zeta \lambda} + 1\right)\EE\left[R_3\right] \\
&\leq 4\alpha\mu\lambda^3\sum_{h=k-d_{k, j_m}}^k\sum_{j = 1}^m \sum_{i=1}^n \frac{\delta p_{ij}\umu_j}{q_jm^3}\left(\frac{1}{\zeta\lambda} + 1\right)\EE\left[r_{ij}^p(x^{h-d_h})\right] \\
&\hspace{20pt}+4\alpha\mu\lambda^3\sum_{h=k-d_{k, j_m}}^k\sum_{j = 1}^m \sum_{i=1}^n \frac{\delta p_{ij}\umu_j}{q_jm^3}\left(\frac{1}{\zeta\lambda} + 1\right)\EE\left[r_{ij}^d(y_i^{h-e^{i}_{h}})\right] \\
&\hspace{20pt} +\lambda^2\sum_{j=1}^m\frac{\umu_j}{m^2}\left(1 + \frac{\tau_p}{mC}\right)\EE\left[\|(S(x^{k-d_k}))_j\|_j^2\right]. \qquad  \qed
\end{align*}
\end{proof}

\begin{proof}[of Lemma~\ref{eq:kappa2bound} (recursive $\kappa_{2, k}$ bound)]
For any $s \in \NN$, we have
\begin{align*}
\EE\left[\kappa_{2,s}\right] &= \sum_{h=s-\tau_p+1}^{s}\sum_{j=1}^m \frac{\umu_j(h - s + \tau_p)}{mC} \EE\left[\|x_j^{h} - x_j^{h-1}\|_j^2\right] \\
&= \sum_{j=1}^m \frac{\umu_j\tau_p}{mC}\EE\left[\|x_j^s - x_j^{s-1}\|^2_j\right] +  \sum_{h=s-\tau_p+1}^{s-1}\sum_{j=1}^m \frac{\umu_j(h - s + \tau_p)}{mC} \EE\left[\|x_j^{h} - x_j^{h-1}\|_j^2\right] \\
&\leq \lambda^2\sum_{j=1}^m \frac{\umu_j\tau_p}{mC}\EE\left[\|\mkQ^{s-1}_j\|^2_j\right] +  \frac{\tau_p}{\tau_p+1} \EE\left[ \kappa_{2, s-1}\right] \\
&\stackrel{\eqref{eq:ASYNCSAGAFPRBOUND}}{\leq}\lambda^2 \sum_{j=1}^m \sum_{i=1}^n \frac{2\umu_j \tau_p p_{ij}}{Cq_jm^3}\EE\left[r_{ij}^p(x^{s-1-d_{s-1}})\right] \\
&+ \lambda^2 \sum_{j=1}^m \sum_{i=1}^n \frac{2\umu_j \tau_p p_{ij}}{Cq_jm^3}\EE\left[r_{ij}^d(y_i^{s-1 - e^{i}_{s-1}})\right] +  \frac{\tau_p}{\tau_p+1} \EE\left[ \kappa_{2,s-1}\right]. \qquad \qed
\end{align*}
\end{proof}
\end{proof}
\newpage
\section{Auxiliary Results}

\subsection{Properties of Operators}

\begin{definition}[Cocoercive Operators]
An operator $S : \cH \rightarrow \cH$ is called \emph{$\beta$-cocoercive} if
$$
\left(\forall x, y \in \cH\right) \qquad \dotp{Sx - Sy,x-y} \geq \beta \|Sx - Sy\|^2.
$$
\end{definition}
\begin{definition}[Averaged Operators]
Let $\alpha \in [0, 1]$. An operator $T : \cH \rightarrow \cH$ is called \emph{$\alpha$-averaged} if there is a nonexpansive map $N : \cH \rightarrow \cH$ such that 
$$
 \qquad T = (1-\alpha)I_{\cH} + \alpha N.
$$
The map $T$ is is called \emph{firmly nonexpansive} if it is $(1/2)$-averaged.
\end{definition}

\begin{proposition}[Cocoercivness from Averagedness]\label{eq:identminusaveraged}
Let $\alpha \in [0, 1]$, and let $T : \cH \rightarrow \cH$ be an $\alpha$-averaged operator. Then for all $\beta > 0$, $\beta(I-T)$ is $(1/(2\alpha\beta))$-cocoercive.
\end{proposition}
\begin{proof}
There is a nonexpansive map $N : \cH \rightarrow \cH$ such that $T = (1-\alpha)I + \alpha N$, and by definition, $T' := (1/2) I + (1/2) N$ is firmly nonexpansive. Thus, $I-T' = (1/2)(I-N)$ is firmly nonexpansive, and hence, $1$-cocoercive~\cite[Remark 4.24(iii)]{bauschke2011convex}. Therefore, for all $x, y \in \cH$, the next bound holds and proves that $\beta(I-T)$ is $(1/\alpha\beta)$-cocoercive:
\begin{align*}
\dotp{\alpha\beta(I-N)x - \alpha\beta(I-N)y, x-y} = 2\alpha\beta\dotp{\frac{1}{2}(I-N)x - \frac{1}{2}(I-N)y, x-y} & \geq 2\alpha\beta\left\|\frac{1}{2}(I-N)x - \frac{1}{2}(I-N)y\right\|^2 \\
&= \frac{1}{2\alpha\beta}\left\|\alpha\beta(I-N)x - \alpha\beta(I-N)y\right\|^2.\qquad \qed
\end{align*}
\end{proof}

\begin{proposition}[The Composition of Averaged Operators is Averaged \cite{combettes2014compositions}]\label{prop:averagedconstants}
Let $\alpha_1, \alpha_2 \in [0, 1]$, and let $T_1, T_2 : \cH \rightarrow \cH$ be $\alpha_1$- and $\alpha_2$-averaged maps, respectively. Then $T_1 \circ T_2$ is $\alpha_{1,2}$ averaged where
\begin{align*}
\alpha_{1,2} := \frac{\alpha_1 + \alpha_2 - 2\alpha_1\alpha_2}{1-\alpha_1\alpha_2}.
\end{align*}
\end{proposition}

\begin{proposition}[Strongly monotone operators from Lipschitz operators]\label{prop:strongmonofromLipschitz}
Let $L \in (0, 1)$ and suppose that $T : \cH \rightarrow \cH$ be $L$-Lipschitz continuous. Then $I - T$ is $(1-L)$-strongly monotone.
\end{proposition}

\begin{proposition}[Coordinate Lipschitz Constants and Coordinate Cocoercivity.]\label{prop:coordinate-BH}
Let $f : \cH \rightarrow (-\infty, \infty]$ be a Fr{\'e}chet differentiable convex function. If $j \in \{1, \ldots, m\}$ and 
\begin{align}\label{eq:coordinatewise}
\left(\forall x \in \cH\right), \, \left(\forall y_j \in \cH_j\right) \qquad f(x + \hat{y_j}) \leq f(x) + \dotp{ \nabla f(x), \hat{y_j}} + \frac{L_{j}}{2} \|\hat{y_j}\|^2; \qquad \hat{y_j} = (0, \ldots, 0, y_j, 0, \ldots, 0),
\end{align}
then 
\begin{align*}
\left(\forall x \in \cH\right), \, \left(\forall y \in \cH\right) \qquad \frac{1}{L_j}\|\nabla f(x) - \nabla  f(y) \|_j^2 \leq \dotp{\nabla f(x) - \nabla f(y), x - y}.
\end{align*}
\end{proposition}
\begin{proof}
Fix $y \in \cH$, and let define a function $g$, which continues to satisfy~\eqref{eq:coordinatewise}, 
$$
g(x) = f(x) - f(y) - \dotp{\nabla f(y), x - y}.
$$
Then $g \geq 0$, and $g(y) = 0$, so $g$ is minimized at $y$. In addition, 
\begin{align*}
\left(\forall x \in \cH\right), \, \left(\forall y_j \in \cH_j\right) \qquad g(y) \leq g(x + \hat{y_j}) \leq g(x) + \dotp{ \nabla f(x), \hat{y_j}} + \frac{L_{j}}{2} \|\hat{y_j}\|^2; \qquad \hat{y_j} = (0, \ldots, 0, y_j, 0, \ldots, 0).
\end{align*}
Thus, we minimize the right hand side of this inequality, over all $\hat{y_j} \in \cH$, and get the next bound:
\begin{align*}
f(x) - f(y) - \dotp{\nabla f(y), x-y} = g(x) - g(y) \geq \frac{1}{2L_j}\|\nabla_j f(x) - \nabla_j f(y) \|^2_j.
\end{align*}
We likewise get the opposite inequality in which the points $x$ and $y$ are exchanged; add both inequalities, and get the result.
\end{proof}

\begin{proposition}[Contractive Forward-Gradient Operator]\label{prop:stronglymonodiffable}
Suppose that the function $f : \cH \rightarrow (-\infty, \infty)$ is $\mu_f$-strongly convex, differentiable, and that the gradient $\nabla f$ is $L$-Lipschitz continuous. Then $I - \gamma \nabla f$ is 
$$
\sqrt{1- 2\gamma\mu_f +  \gamma^2 L\mu_f}
$$ 
Lipschitz continuous whenever $\gamma \leq 2L^{-1}.$
\end{proposition}
\begin{proof}
In the next sequence of inequalities, we use the bound $\dotp{\nabla f(x) - \nabla f(y), x-y} \geq L^{-1}\|\nabla f(x) - \nabla f(y)\|^2$ once:
\begin{align*}
\left( \forall x, y \in \cH\right) \qquad \|(I - \gamma \nabla f)x - (I - \gamma \nabla f)y\|^2 &\leq \|x - y\|^2 - 2\gamma \dotp{ \nabla f(x) - \nabla f(y), x - y} + \gamma^2 \|\nabla f(x) - \nabla f(y)\|^2\\
&\leq  \|x - y\|^2 + \gamma^2 \|\nabla f(x) - \nabla f(y)\|^2\\
&\hspace{20pt} - 2\gamma\left( 1- \frac{\gamma L}{2}\right) \dotp{ \nabla f(x) - \nabla f(y), x - y} - \gamma^2 L \dotp{ \nabla f(x) - \nabla f(y), x - y} \\
&\leq  \left(1- 2\gamma\mu_f\left( 1- \frac{\gamma L}{2}\right)\right) \|x - y\|^2. \qquad \qed
\end{align*}
\end{proof}
\newpage
\section{Examples of $S$}

\subsection{Old Operators}

\begin{proposition}[SAGA/SVRG/S2GD Operator Properties]\label{eq:SAGAOP}
Assume the setting of Section~\ref{sec:SAGA}, and in particular, that 
$$
S_i = \nabla f_i 
$$ 
Then
\begin{enumerate}
\item \textbf{Coherence:} \label{eq:SAGAOP:part:coherence} $S$ satisfies the coherence condition
\begin{align*}
\left( \forall x \in \cH\right), \left( \forall x^\ast \in \zer(S)\right)  \qquad \dotp{S(x) , x- x^\ast} \geq  \sum_{i=1}^N \beta_{i1}\|S_i(x) - S_i(x^\ast)\|^2.
\end{align*}
with $\beta_{i1} \equiv nL^{-1}$. 
\item \textbf{Essential strong quasi-monotonicity:}\label{eq:SAGAOP:part:ess_strong_mono} $S$ is $\mu$-essentially strongly monotone whenever $N^{-1} \sum_{i=1}^N f_i$ is $\mu$-strongly convex.
\item \textbf{Roots:} \label{eq:SAGAOP:part:roots} $\zer(S)$ is precisely the set of minimizers of the dual problem to~\eqref{eq:SDCA_prob}. 
\item \textbf{Demiclosedness:}\label{eq:SAGAOP:part:demiclosed} $S$ is demiclosed at $0$.
\end{enumerate}
\end{proposition}
\begin{proof}
Part~\ref{eq:SAGAOP:part:coherence} (coherence):  The Baillon-Haddad Theorem~\cite{baillon1977quelques} guarantees that for all $i \in \{1, \ldots, N\}$ and $x \in \cH$, 
\begin{align*}
\frac{1}{L}\|\nabla f_i(x) - \nabla f_i(x^\ast)\|^2 \leq \dotp{ \nabla f_i(x) - \nabla f_i(x^\ast), x - x^\ast},
\end{align*}
which implies that
\begin{align*}
\frac{1}{N}\sum_{i=1}^N \frac{1}{L}\|S_i(x) - S_i(x^\ast)\|^2= \frac{1}{N}\sum_{i=1}^N \frac{1}{L}\|\nabla f_i(x) - \nabla f_i(x^\ast)\|^2 \leq \frac{1}{N}\sum_{i=1}^N\dotp{ \nabla f_i(x) - \nabla f_i(x^\ast), x_i - x^\ast} = \dotp{S(x) , x - x^\ast}. 
\end{align*}

Parts~\ref{eq:SAGAOP:part:ess_strong_mono} (essential strong quasi-monotonicity) and~\ref{eq:SAGAOP:part:roots} (roots) are simple, so we omit the proofs.

Part~\ref{eq:SAGAOP:part:demiclosed} (demiclosedness): The operator $I - \gamma S$ is nonexpansive by~\cite[Proposition 4.33]{bauschke2011convex} (for some $\gamma > 0$); thus $S$ is demiclosed at $0$.  \qed
\end{proof}

~\\~\\
\begin{proposition}[Finito Operator Properties]\label{prop:finito_op}
Assume the setting of Section~\ref{sec:finito}, and in particular, that 
\begin{align}
\left(\forall x \in \cH_0^N\right) \qquad S(x)  = x - P_D  \left(x_1 - \gamma \nabla f_1(x_1), \ldots,x_N -  \gamma \nabla f_N(x_N)\right).
\end{align}
In addition, let $\gamma\leq 2L^{-1}$. Then
\begin{enumerate}
\item \textbf{Coherence:}\label{prop:finito_op:part:coherence} $S$ satisfies the coherence condition
\begin{align*}
\left( \forall x \in \cH\right), \left( \forall x^\ast \in \zer(S)\right) \qquad \dotp{S(x) , x- x^\ast} \geq  \sum_{j=1}^N \beta_{1j}\|(S(x))_j\|_j^2.
\end{align*}
with $\beta_{1j} \equiv 4^{-1}\gamma L $.
\item \textbf{Essential strong quasi-monotonicity:}\label{prop:finito_op:part:ess_strong_mono} $S$ is 
$$
\mu := 1 - \sqrt{1-2\gamma\hat{\mu} + \gamma^2\hat{\mu} L }
$$
essentially strong quasi-monotone whenever each $f_i$ is $\hat{\mu}$-strongly convex.  
\item \textbf{Roots:}\label{prop:finito_op:part:roots} $\zer(S) = \{(x_0^\ast, \ldots, x_0^\ast) \in \cH_0^N\mid \text{ minimizes }~\eqref{eq:simplesmooth}\}$
\item \textbf{Demiclosedness:}\label{prop:finito_op:part:demiclosed} $S$ is demiclosed at $0$.
\end{enumerate}
\end{proposition}
\begin{proof}
Part~\ref{prop:finito_op:part:coherence} (coherence):  Let $\mathbf{f} : \cH \rightarrow \cH$ be the function $\mathbf{f}(x_1, \ldots, x_N) = \sum_{j=1}^N f_j(x_j)$. The operator $\nabla\mathbf{f} : x\mapsto (\nabla f_1(x_1), \ldots, \nabla f_N(x_N))$ is $L^{-1}$-cocoercive~\cite{baillon1977quelques} and $L$-Lipschitz, and thus, $I_{\cH} - \gamma \nabla \mathbf{f}$ is $(2^{-1}\gamma L)$-averaged. In addition, $P_D$ is $2^{-1}$-averaged~\cite[Proposition 4.8]{bauschke2011convex}, so by Proposition~\ref{prop:averagedconstants}, the composition $P_D \circ (I_{\cH} - \gamma \nabla \mathbf{f})$ is $2/ (4 - \gamma L)$-averaged. Therefore, $S = \left(I_{\cH} - P_D \circ (I_{\cH} - \gamma \nabla \mathbf{f})\right)$ is $(4^{-1}\gamma L)$-cocoercive, and so $\beta_{1j} =4^{-1} \gamma L$.

Part~\ref{prop:finito_op:part:ess_strong_mono} (essential strong quasi-monotonicity): by~\ref{prop:stronglymonodiffable}, if $\gamma \leq 2L^{-1}$, then the map $I_{\cH} - \gamma \nabla \mathbf{f}$ is $(1- \mu)$-Lipschitz, so the composition $P_D(I_{\cH} - \gamma\nabla \mathbf{f})$ is $(1- \mu)$-Lipschitz. Therefore, by Proposition~\ref{prop:strongmonofromLipschitz}, the operator $S$ is $\mu$-strongly quasi-monotone.

Part~\ref{prop:finito_op:part:roots} (roots) is simple so we omit the proof.

Part~\ref{prop:finito_op:part:demiclosed} (demiclosedness): By~\cite[Proposition 4.33 and Proposition 4.8]{bauschke2011convex}, the operator $I-S$ is nonexpansive because it is the composition of two nonexpansive maps. Thus, $S$ is demiclosed at $0$. 
\qed\end{proof}

~\\~\\

\begin{proposition}[SDCA Operator Properties]\label{prop:SDCA_op}
Assume the setting of section~\ref{sec:SDCA}, and in particular, that 
\begin{align*}
S = I - \prox_{ \mu_0N f^\ast(-\cdot)}\circ (I - \mu_0N\nabla g).
\end{align*}
\begin{enumerate}
\item \textbf{Coherence:}\label{prop:SDCA_op:part:coherence} $S$ satisfies the coherence condition
\begin{align*}
\left( \forall x \in \cH\right), \left( \forall x^\ast \in \zer(S)\right) \qquad \dotp{S(x) , x- x^\ast} \geq  \sum_{j=1}^N \beta_{1j}\|(S(x))_j\|_j^2.
\end{align*}
with $\beta_{1j} \equiv 3/4$
\item \textbf{Essential strong quasi-monotonicity:} \label{prop:SDCA_op:part:ess_strong_mono}$S$ is 
$$
\mu = \frac{\mu_0N}{(\mu_0N + L)}
$$
essentially strongly quasi-monotone---whether or not any $f_j$ is strongly convex.
\item \textbf{Roots:} \label{prop:SDCA_op:part:roots}$\zer(S)$ is precisely the set of dual solutions to~\eqref{eq:SDCA}.
\item \textbf{Demiclosedness:} \label{prop:SDCA_op:part:demiclosed} $S$ is demiclosed at $0$.
\end{enumerate}
\end{proposition}
\begin{proof}
Part~\ref{prop:SDCA_op:part:coherence} (coherence): The gradient $\nabla g$ is $(\mu_0N)^{-1}$-Lipschitz, so $I_{\cH}- \mu_0N \nabla g$ is $(1/2)$-averaged~\cite[Proposition 4.33]{bauschke2011convex}. Thus, by Proposition~\ref{prop:averagedconstants}, the composition of the operators $\prox_{\mu_0 N f^\ast(-\cdot)} \circ(I_{\cH}- \mu_0N \nabla g)$, both of which are $2^{-1}$-averaged, is $(2/3)$-averaged. Thus, by Lemma~\ref{eq:identminusaveraged}, $S = I_{\cH} - \prox_{\mu_0 N f^\ast(-\cdot)} \circ(I_{\cH}- \mu_0N \nabla g)$ is $(3/4)$-cocoercive.

Part~\ref{prop:SDCA_op:part:ess_strong_mono} (essential strong quasi-monotonicity): All the strong monotonicity comes from the $\prox_{\mu_0 Nf^\ast(-\cdot)}$ operator: The gradient of $f(x_1, \ldots, x_N) = \sum_{i=1}^N f_i(x_i)$ is $L$-Lipschitz continuous on $\cH$, so the conjugate function $f^\ast$ is $L^{-1}$-strongly convex~\cite[Theorem 18.15]{bauschke2011convex}, and this, in turn, implies that $\prox_{\mu_0N f^\ast(-\cdot)}$ is $L(\mu_0N + L)^{-1}$-Lipschitz continuous. Thus, $\prox_{\mu_0Nf^\ast(-\cdot)}\circ(I_{\cH}- \mu_0N \nabla g)$ is $L(\mu_0N + L)^{-1}$-Lipschitz continuous~\cite[Proposition 23.11]{bauschke2011convex}, so by Proposition~\ref{prop:strongmonofromLipschitz}, $S$ is $\mu_0N(\mu_0N + L)^{-1}$-essentially strongly quasi-monotone (indeed, strongly monotone). 

Part~\ref{prop:SDCA_op:part:roots} (roots) follows from~\cite[Proposition 25.1(iv)]{bauschke2011convex}.

Part~\ref{prop:SDCA_op:part:demiclosed} (demiclosedness): In part~\ref{prop:SDCA_op:part:coherence}, we showed that $ I - S$ is averaged, and hence nonexpansive. Thus, $S$ is demiclosed at $0$. 
\qed
\end{proof}

~\\~\\
\begin{proposition}[Randomized Projection Operator Properties]\label{prop:RPA}
Assume the setting of Section~\ref{sec:RPA}, and in particular, that
\begin{align*}
 S_i = \begin{cases}
I - P_{C_i} & \text{if } i = 1, \ldots, s_1; \\
I - G_{f_{i-s_1}} & \text{if } i = s_1+1, \ldots, s_1 + s_2. 
\end{cases}
\end{align*}
Then
\begin{enumerate}
\item \textbf{Coherence:}\label{prop:RPA:part:coherence} $S$ satisfies the coherence condition
\begin{align*}
\left( \forall x \in \cH\right), \left( \forall x^\ast \in \zer(S)\right)  \qquad\dotp{S(x) , x- x^\ast} \geq  \sum_{i=1}^N \beta_{i1}\|S_i(x) - S_i(x^\ast)\|^2.
\end{align*}
with $\beta_{i1} \equiv 1$.
\item \textbf{Essential strong quasi-monotonicity:} \label{prop:RPA:part:ess_strong_mono} $S$ is 
$$
\mu = \frac{\max\{1, \varepsilon^2L^{-2}\}}{N\hat{\mu}^2}
$$
essentially strongly quasi-monotone whenever
\begin{enumerate}
\item $\{C_i \mid i = 1, \ldots, s_1\} \cup \{\{x \mid f_i(x) \leq 0\} \mid i = 1, \ldots, s_2\}$ are $\hat{\mu}$-linearly regular,\footnote{A set family $\{D_1, \ldots, D_N\}$ is $\hat{\mu}$-linearly regular if $\forall x\in \cH$, $d_{D_1\cap\cdots \cap D_N}(x) \leq \hat{\mu}\max\{d_{D_1}(x), \ldots, d_{D_N}(x)\}$.}
\item there is an $\varepsilon > 0$ such that $f_i(x) \geq \varepsilon d_{\{ f_i(x) \leq 0\}}(x)$ for all $x \in \cH$, 
\item and there is an $L > 0$ such that $\|g_i(x)\| \leq L$ for all $x \in \cH$.
\end{enumerate}
\item \textbf{Roots:} \label{prop:RPA:part:roots} $\zer(S)$ is precisely the set of solutions to the feasibility problem~\eqref{eq:feasibilityproblem}. 
\item \textbf{Demiclosedness:} \label{prop:RPA:part:demiclosed} $S$ is demiclosed at $0$. 
\end{enumerate}
\end{proposition}
\begin{proof}
Parts~\ref{prop:RPA:part:coherence} (coherence) and~\ref{prop:RPA:part:roots} (roots): Each projection operator $P_{C_i}$ is $2^{-1}$-averaged by~\cite[Proposition~4.8]{bauschke2011convex}, and so $I - P_{C_i}$ is $1$-cocoercive by Proposition~\ref{eq:identminusaveraged}. Similarly, by~\cite[Fact 2.1(v)]{subgradient_projectors}, each subgradient projector $2^{-1}$ quasi-averaged, i.e., 
\begin{align*}
\left(\forall x \in \cH\right), \left(\forall x^\ast \in \{z \mid f(z) \leq 0\}\right) \qquad \|G_{f_i}(x) - x^\ast\|^2 \leq \|x - x^\ast\|^2 - \|x - G_{f_i}(x)\|^2.
\end{align*}
Thus, by rearranging this inequality, we have
\begin{align*}
\left(\forall x \in \cH\right), \left(\forall x^\ast \in \{z \mid f(z) \leq 0\}\right) \qquad \dotp{ x - G_{f_i}(x), x - x^\ast} \geq \|x - G_{f_i}(x)\|^2.
\end{align*}
Altogether, with $C$ equal to the set of points in the intersection~\eqref{eq:feasibilityproblem}, we have
\begin{align*}
\left(\forall x \in \cH\right), \left(\forall x^\ast \in C\right) \qquad\dotp{S(x), x - x^\ast} \geq \frac{1}{N}\sum_{i=1}^N \|S_i(x)\|^2.
\end{align*}

To complete Parts~\ref{prop:RPA:part:coherence} and~\ref{prop:RPA:part:roots}, we need only show that $C = \zer(S)$. But the operator $(I - S) = \frac{1}{N}\sum_{i=1}^N(I - S_{i})$ is the average of quasi-nonexpansive operators, and all of the fixed-point sets of these operators overlap (because $\Fix(G_{f_i}) = \{x \mid f_i(x) \leq 0\}$~\cite[Fact 2.1(ii)]{subgradient_projectors} and there is a point in the intersection~\eqref{eq:feasibilityproblem}). Thus, by~\cite[Proposition 4.34]{bauschke2011convex}, the set of fixed points of $I-S$ is equal to the set of all points in the intersection; in other words, $C = \zer(S)$.

Part~\ref{prop:RPA:part:ess_strong_mono} (essential strong quasi-monotonicity): Let $x^\ast \in \zer(S)$. Then for all $x \in \cH$, with $f(x) > 0$, we have
\begin{align*}
\dotp{x - G_{f_i}(x), x - x^\ast} = \frac{f_i(x)}{\|g_i(x)\|^2}\dotp{g_i(x), x - x^\ast} \geq \frac{f_i(x)}{\|g_i(x)\|^2}(f_i(x) - f_i(x^\ast)) \geq \frac{f_i(x)^2}{\|g_i(x)\|^2} \geq \frac{\epsilon^2}{L^2}d_{\{f_i(z) \leq 0\}}^2(x).
\end{align*}
Similarly, because $I - P_{C_i}$ is $1$-cocoercive, for all $x \in \cH$, we have
\begin{align*}
\dotp{x - P_{C_i}(x), x - x^\ast} \geq \|x - P_{C_i}(x)\|^2 = d_{C_i}^2(x).
\end{align*}
By adding these bounds all together, we have
\begin{align*}
\left(\forall x \in \cH\right) \qquad \dotp{S(x), x- x^\ast} &\geq \frac{1}{N}\min\left\{1, \frac{\epsilon^2}{L^2}\right\}\left( \sum_{i=1}^{s_1} d_{C_i}^2(x) + \sum_{i=s_1 + 1}^{s_1+s_2} d_{\{f_{i-s_1}(z) < 0\}}^2(x)\right) \\
&\geq \frac{1}{N}\min\left\{1, \frac{\epsilon^2}{L^2}\right\}\max\left( \left\{d_{C_i}^2(x) \mid i = 1, \ldots, s_1\right\} \cup \left\{d^2_{\{ f_i(z) \leq 0\}}(x)\mid i = 1, \ldots, s_2\right\} \right) \\
&\geq \frac{d^2_{\zer(S)}(x)}{N\hat{\mu}^2}\min\left\{1, \frac{\epsilon^2}{L^2}\right\},
\end{align*}
which proves that $S$ is $\mu$ essentially strongly quasi-monotone.

Part~\ref{prop:RPA:part:demiclosed} (demiclosedness): Let $\{x^k \}_{k \in \NN} \subseteq \cH$ be a sequence of points, let $x \in \cH$, and suppose that two limits hold: $x^k \rightharpoonup x$ and $S(x^k) \rightarrow 0$. To show that $S$ is demiclosed at $0$, we need to show that $S(x) = 0$. We prove this in two parts: first we show that for all $i$, $S_i(x^k) \rightarrow 0$; and second we show that these limits imply that $x \in \zer(S)$.

``$S_i(x^k) \rightarrow 0$:" For all $i$, let $T_i = I - S_i$. These operators $T_i$ are quasinonexpansive (as noted in Part~\ref{prop:RPA:part:coherence}). Thus, for any $y \in \zer(S)$, we have
\begin{align*}
2\dotp{ T_ix^k - x^k, x^k - y}  &= \|T_ix^k  - y \|^2 - \|T_i x^k - x^k\|^2 - \|x^k - y\|^2 \leq  - \|T_ix^k - x^k\|^2.
\end{align*}
Thus, 
\begin{align*}
\frac{1}{N}\sum_{i=1}^N - \|T_ix^k - x^k\|^2 \geq \frac{1}{N} \sum_{i=1}^N 2\dotp{ T_ix^k - x^k, x^k - y} = \dotp{ S(x^k), x^k - y} \rightarrow 0,
\end{align*}
and so $S_i(x^k) = T_ix^k - x^k \rightarrow 0$,

``$x \in \zer(S)$:" All the projection operators $P_{C_i}$ are nonexpansive, and moreover, $P_{C_i}x^k - x^k \rightarrow 0$. Therefore, because all nonexpansive operators are demiclosed at $0$, we have $P_{C_i} x = x$ and $x \in C_i$. 

The subgradient projectors are not nonexpansive, but we can still show that $G_{f_i}x = x$, and hence, $f_i(x)\leq 0$. Indeed, $\|G_{f_i}(x^k) - x^k\| = f_i(x^k)\|g_i(x^k)\|^{-1} \rightarrow 0$. And because subdifferential operators of continuous convex functions are locally bounded, because the sequence  $\{x^k\}_{k \in \NN}$ is bounded, and because the subgradients $\{g_i(x^k) \}_{k \in \NN}$ are bounded, it follows that $f_i(x^k) \rightarrow 0$. Then because convex functions are weakly lower semicontinuous, we have $f_i(x) \leq \liminf_k f_i(x^k) \leq 0$. Altogether, $x \in \zer(S)$, and consequently, $S$ is demiclosed at $0$.\qed
\end{proof}

\begin{corollary}[Kaczmarz Operator Properties]\label{cor:Kaczmarz}
Assume the setting of Section~\ref{sec:RPA}, and in particular, that
\begin{align*}
\left(\forall i \right) \qquad S_i = (I - P_{C_i}) = (\dotp{a_i, \cdot} -b_i)a_i.
\end{align*} 
Then
\begin{enumerate}
\item \textbf{Coherence:}\label{cor:Kaczmarz:part:coherence} $S$ satisfies the coherence condition
\begin{align*}
\left( \forall x \in \cH\right), \left( \forall x^\ast \in \zer(S)\right)  \qquad \dotp{S(x) , x- x^\ast} \geq  \sum_{i=1}^N \beta_{i1}\|S_i(x) - S_i(x^\ast)\|^2.
\end{align*}
with $\beta_{i1} \equiv 1$.
\item \textbf{Essential strong quasi-monotonicity:} \label{cor:Kaczmarz:part:ess_strong_mono} $S$  is 
$$
\mu = \frac{1}{N\|A^{-1}\|_2^{2}}.
$$
essentially strong quasi-monotone, where
$$
\|A^{-1}\|_2 := \inf \{M \mid \left(\forall x \in \cH\right) \; M\|Ax \|_2 \geq \|x\|_2 \}.
$$
\item \textbf{Roots:} \label{cor:Kaczmarz:part:roots} $\zer(S)$ is precisely the set of solutions to the linear equation $Ax = b$.
\item \textbf{Demiclosedness:} \label{cor:Kaczmarz:part:demiclosed} $S$ is demiclosed at $0$. 
\end{enumerate}
\end{corollary}
\begin{proof}
Parts~\ref{cor:Kaczmarz:part:coherence} (coherence),~\ref{cor:Kaczmarz:part:roots} (roots), and~\ref{cor:Kaczmarz:part:demiclosed} (demiclosedness) follow from Proposition~\ref{prop:RPA}. 

Part~\ref{cor:Kaczmarz:part:ess_strong_mono}: Observe that
\begin{align*}
\dotp{S(x), x - x^\ast} = \frac{1}{N} \sum_{i=1}^N \dotp{(\dotp{a_i, x} - b_i)a_i, x - x^\ast} =   \frac{1}{N} \sum_{i=1}^N \dotp{(\dotp{a_i, x} - \dotp{a_i, x^\ast})a_i, x - x^\ast}  &=  \frac{1}{N} \sum_{i=1}^N \dotp{a_i, x-x^\ast}^2 \\
&= \frac{1}{N} \|A(x - x^\ast)\|_2^2 \\
&\geq \frac{1}{N\|A^{-1}\|_2^2} \|x - x^\ast\|^2. \qquad \qed
\end{align*}
\end{proof}

\subsection{New Operators}

\begin{proposition}[Proximal SAGA/SVRG Operator Properties]\label{prop:SAGAopprop}
Assume the setting of Section~\ref{sec:proximable_SAGA}, and in particular, that
\begin{align*}
\left(\forall i < N +1\right) \qquad& S_{i} = \frac{\gamma}{N}\nabla f_i\circ \prox_{\gamma g}; \\
& S_{N+1} = (I-\prox_{\gamma g}),
\end{align*}
Then 
\begin{enumerate}
\item \textbf{Coherence:}\label{prop:SAGAopprop:part:coherence} $S$ satisfies the coherence condition
\begin{align*}
\left( \forall x \in \cH\right), \left( \forall x^\ast \in \zer(S)\right) \qquad \dotp{S(x) , x- x^\ast} \geq  \sum_{i=1}^N \beta_{i1}\|S_i(x) - S_i(x^\ast)\|^2.
\end{align*}
with
\begin{align*}
\beta_{i1} :=  \frac{N}{2\gamma L (N+1)} \qquad i = 1, \ldots, N && \text{and} && \beta_{(N+1)1} := \frac{1}{(N+1)}\left(1 - \frac{\gamma L}{2}\right)
\end{align*}
\item \textbf{Essential strong quasi-monotonicity:}\label{prop:SAGAopprop:part:ess_strong_mono} $S$ is 
$$
\mu =\frac{1 + \gamma \mu_g - \sqrt{1- 2\gamma\mu_f +  \gamma^2 L\mu_f}}{(N+1)(1+\gamma \mu_g)}
$$
essentially strongly quasi-monotone, where $N^{-1}\sum_{i=1}^N f_i$ is $\mu_f$-strongly convex and $g$ is $\mu_g$-strongly convex.
\item \textbf{Roots:}\label{prop:SAGAopprop:part:roots} with any choice of $\gamma$, we have
\begin{align*}
x\in \zer(S) \implies \prox_{\gamma g} (x) \text{ solves \eqref{eq:simplesmooth_one_nonsmooth}}, 
\end{align*}
and $\zer(S) \neq \emptyset$ if, and only if, $g + N^{-1}\sum_{j=1}^N f_j$ has a minimizer.
\item \textbf{Demiclosedness:} \label{prop:SAGAopprop:part:demiclosed} $S$ is demiclosed at $0$.
\end{enumerate}
\end{proposition}
\begin{proof}

Part~\ref{prop:SAGAopprop:part:coherence} (coherence): The operator $S_{N+1} = (I- \prox_{\gamma g})$ is $1$-cocoercive 
 and $\nabla f_i$ is $N(\gamma L)^{-1}$-cocoercive. Thus, for all $x \in \cS$, we have
\begin{align*}
\dotp{S(x), x- x^\ast} &= \frac{1}{N+1}\sum_{i=1}^{N+1}\dotp{S_{i}(x) - S_{i}(x^\ast), x - x^\ast} \\
&= \frac{1}{N+1}\sum_{i=1}^{N}\dotp{S_{i}(x) - S_{i}(x^\ast), \prox_{\gamma g}(x) - \prox_{\gamma g}(x^\ast)}  + \frac{1}{N+1}\dotp{S_{N+1}(x) - S_{N+1}(x^\ast), x - x^\ast}\\
&\hspace{20pt} + \frac{1}{(N+1)}\sum_{i=1}^{N}\dotp{S_{i}(x) - S_{i}(x^\ast), S_{N+1}(x) - S_{N+1}(x^\ast)}   \\
&\geq \frac{N}{\gamma L(N+1)} \sum_{i=1}^N\|S_i(x) - S_i(x^\ast)\|^2 + \frac{1}{(N+1)}\|S_{N+1}(x) - S_{N+1}(x^\ast)\|^2 \\
&\hspace{20pt} - \frac{N}{2\gamma L (N+1)} \sum_{i=1}^N\|S_i(x) - S_i(x^\ast)\|^2 -  \frac{\gamma L }{2(N+1)}\|S_{N+1}(x) - S_{N+1}(x^\ast)\|^2 \\
&\geq \frac{N}{2\gamma L (N+1)}\sum_{i=1}^N\|S_i(x) - S_i(x^\ast)\|^2 + \frac{1}{(N+1)}\left(1 - \frac{\gamma L}{2}\right)\|S_{N+1}(x) - S_{N+1}(x^\ast)\|^2.
\end{align*}

Part~\ref{prop:SAGAopprop:part:ess_strong_mono} (essential strong quasi-monotonicity): Let 
$$
T = \left(I - \frac{\gamma}{N} \sum_{i=1}^N \nabla f_i\right) \circ \prox_{\gamma g}.
$$
The operator $\prox_{\gamma g}$ is $(1+\gamma \mu_g)^{-1}$-Lipschitz continuous~\cite[Proposition 23.11]{bauschke2011convex}. In addition, by Proposition~\ref{prop:stronglymonodiffable}, $\left(I - \gamma N^{-1} \sum_{i=1}^N \nabla f_i\right)$ is $\sqrt{1- 2\gamma\mu_f +  \gamma^2 L\mu_f}$-Lipschitz continuous, whenever $\gamma \leq 2L^{-1}$. Thus, $T$, the composition of the two Lipschitz operators, is $ (1+\gamma \mu_g)^{-1}\sqrt{1- 2\gamma\mu_f +  \gamma^2 L\mu_f}$ Lipschitz continuous. Therefore, $S = (N+1)^{-1} (I - T) $ is
$$
\mu =  \frac{1}{N+1} \left( 1- \frac{\sqrt{1- 2\gamma\mu_f +  \gamma^2 L\mu_f}}{1+\gamma \mu_g}\right) =  \frac{1 + \gamma \mu_g - \sqrt{1- 2\gamma\mu_f +  \gamma^2 L\mu}}{(N+1)(1+\gamma \mu_g)}
$$
essentially strongly quasi-monotone.

Part~\ref{prop:SAGAopprop:part:roots} (roots): 
\begin{align*}
x \in \zer(S) \Leftrightarrow 0 = (I- \prox_{\gamma g}) + \frac{\gamma}{N} \sum_{i=1}^N \nabla f_i(\prox_{\gamma g}(x)) &\in \gamma \partial g(\prox_{\gamma g}(x)) + \frac{\gamma }{N} \sum_{i=1}^N \nabla f_i(\prox_{\gamma g}(x)) \\
&\Leftrightarrow \prox_{\gamma g}(x) \text{ minimizes }  g + \frac{1}{N}\sum_{j=1}^N f_j.\qquad \qed
\end{align*}

Part~\ref{prop:SAGAopprop:part:demiclosed} (demiclosedness): The operator $T = (N+1)(I- S)$ in (displayed in Part~\ref{prop:SAGAopprop:part:ess_strong_mono}) is the composition of two nonexpansive maps, and thus, it is nonexpansive. Therefore, $S$ is demiclosed at $0$.
\end{proof}

~\\~\\
\begin{proposition}[LinSAGA/LinSVRG/SuperSAGA/SuperSVRG Operator Properties]\label{prop:LinSAGA}
Assume the setting of Section~\ref{sec:LinSAGA}, and in particular, that
\begin{align*}
\left(\forall i < N +1\right) \qquad& S_i = \frac{\gamma}{N} P_V \circ \nabla f_i  \circ P_V \circ \prox_{\gamma g};\\
&S_{N+1} = (I- 2P_V)\circ \prox_{\gamma g} + P_V. 
\end{align*} 
 Then 
 \begin{enumerate}
\item \textbf{Coherence:} \label{prop:LinSAGA:part:coherence} $S$ satisfies the coherence condition
\begin{align*}
\left( \forall x \in \cH\right), \left( \forall x^\ast \in \zer(S)\right)  \qquad\dotp{S(x) , x- x^\ast} \geq  \sum_{i=1}^N \beta_{i1}\|S_i(x) - S_i(x^\ast)\|^2.
\end{align*}
with 
\begin{align*}
\beta_{i1} :=  \frac{N}{2\gamma \hat{L}_i (N+1)} \qquad i = 1, \ldots, N && \text{and} && \beta_{(N+1)1} :=  \frac{1}{N+1}\left(1 - \frac{1}{2N} \sum_{i=1}^N\gamma \hat{L}_i\right).
\end{align*}
\item \textbf{Essential strong quasi-monotonicity:} \label{prop:LinSAGA:ess_strong_mono} $S$ is 
$$
\mu = \frac{1}{N+1}\left(1 - \left(\frac{1}{(1+(\gamma L_g)^{-1})} + \frac{\sqrt{1- 2\gamma\mu_f +  \gamma^2 L\mu_f}}{(1+\gamma \mu_g)}\right)\right)
$$
essentially strongly quasi-monotone (when $\gamma \leq 2L^{-1}$ and $\mu > 0$), where $L = N^{-1} \sum_{i=1}^N \hat{L}_i$, the function $N^{-1}\sum_{i=1}^N f_i$ is $\mu_f$-strongly convex, the function $g$ is differentiable and $\mu_g$-strongly convex, and the gradient $\nabla g$ is $L_g$-Lipschitz continuous.\label{prop:LinSAGA:part:ess_strong_mono} 
\item \textbf{Roots:} \label{prop:LinSAGA:part:roots} with any choice of $\gamma$, we have
\begin{align*}
x^\ast\in \zer(S) \implies  \prox_{\gamma g}(x^\ast) \text{ solves \eqref{eq:super_SAGA_1}}, 
\end{align*}
and $\zer(S) \neq \emptyset$ if, and only if, $ \zer\left(\partial g + N^{-1}\sum_{i=1}^N \nabla f_i  + N_{V}\right) \neq \emptyset$.
\item \textbf{Demiclosedness:} \label{prop:LinSAGA:part:demiclosed} $S$ is demiclosed at $0$. \end{enumerate}
\end{proposition}
\begin{proof}
Part~\ref{prop:LinSAGA:part:coherence} (coherence): The operator $S_{N+1} = (I- 2P_V)\circ \prox_{\gamma g} + P_V$ is $1$-cocoercive because 
$$
S_{N+1} = I - \left[P_V (2\prox_{\gamma g} - I) + I - \prox_{\gamma g}\right]
$$
and by~\cite[Proposition 4.21]{bauschke2011convex}, $P_V (2\prox_{\gamma g} - I) + I - \prox_{\gamma g}$ is $2^{-1}$ averaged. Not only is the operator $1$-cocoercive, it also nicely partitions into $V$ and $V^\perp$ components: 
$$
S_{N+1} = P_{V^\perp} \prox_{\gamma g} + P_V (I - \prox_{\gamma g})
$$
Therefore, $\dotp{S_i(x), S_{N+1}(y)} = \dotp{ S_i(x), P_V (y - \prox_{\gamma g}(y))}$ for all $i < N+1$ and any $x, y \in \cH$. In addition, $ \frac{\gamma}{N} P_V \circ \nabla f_i  \circ P_V$ is $N(\gamma L_i)^{-1}$-cocoercive.  Thus, for all $x \in \cS$, we have
\begin{align*}
\dotp{S(x), x- x^\ast} &= \frac{1}{N+1}\sum_{i=1}^{N+1}\dotp{S_{i}(x) - S_{i}(x^\ast), x - x^\ast} \\
&= \frac{1}{N+1}\sum_{i=1}^{N}\dotp{S_{i}(x) - S_{i}(x^\ast), P_V\prox_{\gamma g}(x) - P_V\prox_{\gamma g}(x^\ast)}  + \frac{1}{N+1}\dotp{S_{N+1}(x) - S_{N+1}(x^\ast), x - x^\ast}\\
&\hspace{20pt} + \frac{1}{N+1}\sum_{i=1}^{N}\dotp{S_{i}(x) - S_{i}(x^\ast), S_{N+1}(x) - S_{N+1}(x^\ast)}   \\
&\geq \frac{N}{N+1} \sum_{i=1}^N\frac{1}{\gamma \hat{L}_i}\|S_i(x) - S_i(x^\ast)\|^2 + \frac{1}{N+1}\|S_{N+1}(x) - S_{N+1}(x^\ast)\|^2 \\
&\hspace{20pt} - \frac{N}{N+1}\sum_{i=1}^N\frac{1}{2\gamma \hat{L}_i } \|S_i(x) - S_i(x^\ast)\|^2 -  \frac{\frac{1}{N}\sum_{i=1}^N\gamma \hat{L}_i}{2(N+1)}\|S_{N+1}(x) - S_{N+1}(x^\ast)\|^2 \\
&\geq \frac{N}{2 (N+1)}\sum_{i=1}^N\frac{1}{\gamma \hat{L}_i}\|S_i(x) - S_i(x^\ast)\|^2 + \frac{1}{N+1}\left(1 - \frac{1}{2N} \sum_{i=1}^N\gamma \hat{L}_i\right)\|S_{N+1}(x) - S_{N+1}(x^\ast)\|^2.
\end{align*} 

Part~\ref{prop:LinSAGA:ess_strong_mono} (essential strong quasi-monotonicity): Let 
$$
T: = P_{V^\perp}\circ (I - \prox_{\gamma g}) + P_V\circ\left(I - \frac{\gamma}{N} \sum_{i=1}^N \nabla f_i\right)\circ P_V\circ \prox_{\gamma g}
$$
The operator $T$ is Lipschitz continuous because it is built from Lipschitz continuous pieces. For example, by Proposition~\ref{prop:stronglymonodiffable}, the operator $P_V\circ \left(I - \gamma N^{-1} \sum_{i=1}^N \nabla f_i\right)\circ P_V$ is $\sqrt{1- 2\gamma\mu_f +  \gamma^2 L\mu_f}$-Lipschitz continuous; by~\ref{prop:stronglymonodiffable} the operator $\prox_{\gamma g}$ is $(1+ \mu_g \gamma)^{-1}$-Lipschitz continuous; by~\cite[Remark 23.19]{bauschke2011convex}, we have $I - \prox_{\gamma g} = \prox_{(\gamma g)^\ast}$, and by~\cite[Theorem 18.15]{bauschke2011convex}, the function $(\gamma g)^\ast$ is $(\gamma L_g)^{-1}$-strongly convex; thus, by~\cite[Proposition 23.11]{bauschke2011convex}, $I - \prox_{\gamma g}$ is $(1 + (\gamma L_g)^{-1})^{-1}$-Lipschitz continuous. When taken together, these properties yield
\begin{align*}
\left(\forall x, y \in \cH\right) \|T(x) - T(y)\| &\leq \|P_{V^\perp}(I - \prox_{\gamma g})(x) - P_{V^\perp}(I - \prox_{\gamma g})(y) \|\\
&\hspace{10pt} + \left\|P_V\circ\left(I - \frac{\gamma}{N} \sum_{i=1}^N \nabla f_i\right)\circ P_V\circ \prox_{\gamma g}(x) - P_V\circ\left(I - \frac{\gamma}{N} \sum_{i=1}^N \nabla f_i\right)\circ P_V\circ \prox_{\gamma g}(y)\right\| \\
&\leq \left(\frac{1}{(1+(\gamma L_g)^{-1})} + \frac{\sqrt{1- 2\gamma\mu_f +  \gamma^2 L\mu_f}}{(1+\gamma \mu_g)}\right) \|x - y \|.
\end{align*}
Thus, by Proposition~\ref{prop:strongmonofromLipschitz}, $S = (N+1)^{-1}(I-T)$ is 
$$
\frac{1}{N+1}\left(1 - \left(\frac{1}{(1+(\gamma L_g)^{-1})} + \frac{\sqrt{1- 2\gamma\mu_f +  \gamma^2 L\mu_f}}{(1+\gamma \mu_g)}\right)\right)
$$
strongly monotone.

Part~\ref{prop:LinSAGA:part:roots} (roots): 
\begin{align*}
x \in \zer(S) \Leftrightarrow 
0 &= P_{V^\perp}\prox_{\gamma g}(x); \text{ and}   \\
0 & = P_V(x - \prox_{\gamma g}(x)) +  \frac{\gamma}{N} \sum_{i=1}^N P_V\nabla f_i( P_V\prox_{\gamma g}(x)) \\
&\in P_V\partial g(\prox_{\gamma g}(x)) +  \frac{\gamma}{N} \sum_{i=1}^N P_V\nabla f_i( P_V\prox_{\gamma g}(x)) \\
\Leftrightarrow 0 &  \in \partial g(\prox_{\gamma}(x)) + \frac{\gamma}{N} \sum_{i=1}^N \nabla f_i( P_V\prox_{\gamma g}(x)) + N_V(\prox_{\gamma g}(x)), 
\end{align*}
where the last line follows because $\prox_{\gamma g}(x) \in V$, and so $N_V(\prox_{\gamma g}(x)) = V^\perp$ absorbs the nonzero $V^\perp$ component of $(x - \prox_{\gamma g}(x)) +  \gamma N^{-1} \sum_{i=1}^N \nabla f_i(\prox_{\gamma g}(x))$.

Part~\ref{prop:LinSAGA:part:demiclosed} (demiclosedness): The operator $T = (N+1)(I- S)$ in (displayed in Part~\ref{prop:LinSAGA:part:ess_strong_mono}) is nonexpansive:
\begin{align*}
\left(\forall x, y \in \cH\right) \|T(x) - T(y)\|^2 &= \|P_{V^\perp}(I - \prox_{\gamma g})(x) - P_{V^\perp}(I - \prox_{\gamma g})(y) \|^2\\
&\hspace{10pt} + \left\|P_V\circ\left(I - \frac{\gamma}{N} \sum_{i=1}^N \nabla f_i\right)\circ P_V\circ \prox_{\gamma g}(x) - P_V\circ\left(I - \frac{\gamma}{N}  \sum_{i=1}^N \nabla f_i\right)\circ P_V\circ \prox_{\gamma g}(y)\right\|^2 \\
&\leq \|(I - \prox_{\gamma g})(x) - (I - \prox_{\gamma g})(y)\|^2 + \| \prox_{\gamma g}(x) -  \prox_{\gamma g}(y)\|^2 \\
&\leq \|x - y\|^2,
\end{align*}
where the first ``$=$" follows because $V \perp V^\perp$ and the last line follows because $\prox_{\gamma g}$ is $2^{-1}$-averaged. Therefore $T$ is nonexpansive and $S$ is demiclosed at $0$. 
\qed
\end{proof}

~\\~\\
\begin{proposition}[TropicSMART Operator Properties]\label{prop:RSCMP}
Assume the setting of Section~\ref{sec:TropicSMART}, and in particular, that for all $x \in \cH$, we have
\begin{align*}
(S(x))_j = \begin{cases}
x_j - \prox_{\gamma_j g_j}\left(x_j - \gamma_j  A_j^\ast \left( x_{M+1} + 2\gamma_{M+1} \left(\sum_{l=1}^M A_l x_l - b\right)\right) - \gamma_j \nabla_j f(x)\right) & \text{if $j < M+1$};\\
 -\gamma_{M+1}\left( \sum_{l=1}^M A_l x_l - b\right) &\text{if $j = M+1$.}
\end{cases}
\end{align*}
Then
\begin{enumerate}
\item  \textbf{Coherence:} \label{prop:RSCMP:part:coherence}there is a strongly positive self-adjoint linear operator $P : \cH \rightarrow \cH$ such that if $\delta \in (0, 1)$;
\begin{align*}
\gamma_{M+1}\left( \sum_{j=1}^M \gamma_j\|A_j\|^2\right)\leq \delta;  &&\text{and} &&  \max_j\{\gamma_j\} \leq \frac{2(1-\sqrt{\delta})}{L},
\end{align*}
then the linear map $P$ satisfies $\sum_{j=1}^{M+1} \underline{M}_j\|x_j\|^2_j \leq \|x\|_P^2 \leq   \sum_{j=1}^{M+1} \overline{M}_j\|x_j\|^2_j$,
where for all $j$, we have 
\begin{align*}
\underline{M}_j := \frac{1-\sqrt{\delta}}{\gamma_j} && \text{and} && \overline{M}_j := \frac{1+\sqrt{\delta}}{\gamma_j}.
\end{align*}
With these parameters, $S$ also satisfies the coherence condition~\eqref{eq:coherence}
 \begin{align*}
\left( \forall x \in \cH\right), \left( \forall x^\ast \in \zer(S)\right)  \qquad \dotp{ S(x), x - x^\ast}_P \geq \sum_{j=1}^{M+1} \beta_{1j}\|(S(x))_j\|^2_j \qquad \text{with } \qquad \beta_{1j} :=  \frac{L\max_j\{\gamma_j\}}{4\gamma_j}.
\end{align*}
\item \textbf{Roots:}\label{prop:RSCMP:part:roots} with any choice of $\gamma_j$, we have
\begin{align*}
x^\ast\in \zer(S) \implies   (x_1^\ast, \ldots, x_M^\ast) \text{ solves \eqref{eq:super_SAGA_1}}, 
\end{align*}
and $\zer(S) \neq \emptyset$ if, and only if, $ \zer\left(\partial g + \nabla f  + N_{\{x \in \cH \mid \sum_{j=1}^M A_j x_j = b\}}\right) \neq \emptyset$.
\item \textbf{Demiclosedness:} \label{prop:RSCMP:part:demiclosed} $S$ is demiclosed at $0$. 
\end{enumerate}

\end{proposition}
\begin{proof}
Part~\ref{prop:RSCMP:part:coherence} (coherence): Here is the linear map $P$ in block matrix form
\begin{align*}
P &= \begin{bmatrix}
\frac{1}{\gamma_1} I_{\cH_1}  & 0 & \cdots & 0 &   A_{1}^\ast \\
0 & \frac{1}{\gamma_2} I_{\cH_2}  & \cdots & 0 &   A_2^\ast \\
\vdots& & \ddots & & \vdots  \\
0 & 0  & \cdots & \frac{1}{\gamma_M} I_{\cH_M} &   A_M^\ast \\
A_1 & A_2  & \cdots & A_M & \frac{1}{\gamma_{M+1}} I_{\cH_{M+1}} \\
\end{bmatrix}.
\end{align*}
The lower bound is below
\begin{align*}
\dotp{x, x}_P  &= \sum_{j=1}^{M+1} \frac{1}{\gamma_j}\|x_j\|^2_j + \sum_{j=1}^M 2\dotp{A_jx_j, x_{M+1}}_{M+1} \\
&\geq\sum_{j=1}^{M+1} \frac{1}{\gamma_j}\|x_j\|_j^2 - \sum_{j=1}^M 2\|A_j\|\|x_j\|_{j}\| x_{M+1}\|_{M+1} \\
&\geq \sum_{j=1}^{M+1} \frac{1}{\gamma_j}\|x_j\|_j^2 - \sum_{j=1}^M \left(\frac{\sqrt{\delta}\|x_j\|_j^2}{\gamma_j} + \frac{\|A_j\|^2\|x_{M+1}\|_{M+1}^2}{\sqrt{\delta}}\right) \\
&\geq \sum_{j=1}^{M} \frac{1-\sqrt{\delta}}{\gamma_j}\|x_j\|^2_j + \frac{1}{\gamma_{M+1}} \left(1- \gamma_{M+1}\sum_{j=1}^M \frac{\|A_j\|^2}{\sqrt{\delta}} \right)  \|x_{M+1}\|_{M+1}^2\\
&\geq \sum_{j=1}^{M+1} \frac{1-\sqrt{\delta}}{\gamma_j}\|x_j\|^2_j.
\end{align*}
The upper bound follows the exact same argument but all ``$-$" signs are changed to ``$+$" signs.

To get the $\beta_{ij}$, we find an $\alpha$-averaged operator $T$ in the norm $\|\cdot\|_P$ such that $S = (I - T)$; then we apply Proposition~\ref{eq:identminusaveraged} together with the lower on $\|\cdot\|_P^2$ that we just derived. 

Let $A: \cH \rightarrow 2^\cH$ be the monotone operator\footnote{This matrix notation is an intuitive visual form for the product of the subdifferentials.}
$$
\left(\forall x\in \cH\right) \qquad A(x_1, \ldots, x_{M+1}) := \begin{bmatrix}
\partial g_1(x_1) \\
\vdots \\
\partial g_M(x_M) \\
0
\end{bmatrix};
$$
let $B : \cH \rightarrow \cH$ be the skew symmetric linear map:
\begin{align*}
B := \begin{bmatrix}
0 &  \cdots & 0 & A_1 \\   
\vdots &  & \vdots & \vdots\\
0 & \cdots & 0 & A_{M+1} \\   
-A_1^\ast  & \cdots & -A_{M}^\ast & 0   
\end{bmatrix};
\end{align*}
and let $C : \cH \rightarrow \cH$ be the cocoercive operator
\begin{align*}
\left(\forall x\in \cH\right) \qquad  C(x_1, \ldots, x_{M+1}) : = \begin{bmatrix}
\nabla f(x_1, \ldots, x_M) \\
0
\end{bmatrix}.
\end{align*}
Then define $T$ to be the forward-backward operator 
$$
T:= J_{P^{-1}(A+B)} \circ ( I - P^{-1}C).
$$
Given $x$, we compute $Tx$:
\begin{align*}
x^+ = Tx & \Leftrightarrow P(x - x^+) \in (A + B)x^+ + Cx\\
&\Leftrightarrow 
\begin{bmatrix} 
\frac{1}{\gamma_1}(x_1 - x_1^+)  + A_1^\ast(x_{M+1} - x_{M+1}^+) \\
\vdots  \\
\frac{1}{\gamma_M}(x_M - x_M^+)  + A_M^\ast(x_{M+1} - x_{M+1}^+) \\
\frac{1}{\gamma_{M+1}}(x_{M+1} - x_{M+1}^+) + \sum_{j=1}^M A_j (x_j - x_j^+) 
\end{bmatrix}
\in 
\begin{bmatrix}
\partial g_1(x_1^+) + A_1^\ast x_{M+1}^+ + \nabla_1 f(x_1, \ldots, x_{M})\\
\vdots \\
\partial g_M(x_M^+) + A_M^\ast x_{M+1}^+ + \nabla_M f(x_1, \ldots, x_{M}) \\
-\sum_{j=1}^M A_jx_j^+
\end{bmatrix} \\
& \Leftrightarrow \left(\forall j < M+1\right) \qquad  x_j^{+} = \prox_{\gamma_j g_j}\left(x_j - \gamma_j  A_j^\ast \left( x_{M+1} + 2\gamma_{M+1} \left(\sum_{l=1}^M A_l x_l - b\right)\right) - \gamma_j \nabla_j f(x_1, \ldots, x_m)\right); \\
& \hphantom{\Leftrightarrow \left(\forall j < M+1\right)} \qquad  x_{M+1}^{+} = x_{M+1} + \gamma_{M+1}\left( \sum_{l=1}^M A_l x_l - b\right). \numberthis\label{eq:RSCMP_inclusion}
\end{align*} 
Thus, $S = I-T$. 

The operator $T$ is a forward-backward operator, and such operators are always $\alpha$-averaged. But computing $\alpha$ requires the averagedness coefficient of $I - P^{-1} C$; this constant can be computed from existing work:
\begin{lemma}[{\cite[Proposition 1.5]{myPDpaper}}]
Let $(\cG, \dotp{, })$ be a Hilbert space, and let $U : \cG \rightarrow \cG$ be a strongly positive linear map such that for all $x \in \cG$, we have $\|x\|_U^2 = \dotp{Ux , x} \geq \xi\|x\|^2$. Then if $C : \cH \rightarrow C$ is $\beta$-cocoercive in the norm $\|\cdot\|$, it is $\beta\xi$ cocoercive in the norm $\|\cdot\|_U$.
\end{lemma}

Thus, from the already computed lower bound $\|\cdot\|_P^2 \geq (1-\sqrt{\delta})(\max_j\{\gamma_j\})^{-1}\|\cdot \|_{\pr}^2$, and from the knowledge that $C$ is $L^{-1}$-cocoercive in $\|\cdot\|_\pr$, the composition $P^{-1}C$ is $(1-\sqrt{\delta})(L\max_j\{\gamma_j\})^{-1}$-cocoercive in $\|\cdot\|_P$. Then the standard result~\cite[Proposition 4.33]{bauschke2011convex} shows that $I - P^{-1}C$ is $\alpha_1 := 2^{-1}(1-\sqrt{\delta})^{-1}L\max_j\{\gamma_j\}$ averaged as long as 
$$
\max_j\{\gamma_j\} \leq \frac{2(1-\sqrt{\delta})}{L},
$$
which we assume to be true. Therefore, because $J_{ P^{-1}(A+B)}$ is $\alpha_2 := 2^{-1}$ averaged in $\|\cdot\|_P^2$, Proposition~\ref{prop:averagedconstants} shows that $T$ is 
\begin{align*}
\alpha := \frac{\alpha_1 + \alpha_2 - 2\alpha_1 \alpha_2}{ 1 - \alpha_1 \alpha_2} = \frac{\frac{1}{2}}{ 1 - \frac{L\max_j\{\gamma_j\}}{4(1-\sqrt{\delta})}} = \frac{2}{4 - \frac{L\max_j\{\gamma_j\}}{1-\sqrt{\delta}} }
\end{align*}
averaged in the norm $\|\cdot\|_P$. Thus, from Proposition~\ref{eq:identminusaveraged}, the operator $S = I - T$ is 
\begin{align*}
\beta := 1 - \frac{1}{2\alpha} = \frac{L\max_j\{\gamma_j\}}{4(1-\sqrt{\delta})}
\end{align*}
cocoercive in the norm $\|\cdot\|_P$; with this property, we obtain $\beta_{ij}$
\begin{align*}
\left(\forall x \in \cH\right) \qquad \dotp{S(x), x - x^\ast}_P &\geq \beta\|S(x)\|^2_P \geq  \sum_{j=1}^{M+1} \frac{\beta(1-\sqrt{\delta})}{\gamma_j}\|(S(x))_j\|_j^2 =  \sum_{j=1}^{M+1}  \frac{L\max_j\{\gamma_j\}}{4\gamma_j}\|(S(x))_j\|_j^2. 
\end{align*}

Part~\ref{prop:RSCMP:part:roots} (roots): From the first line of~\eqref{eq:RSCMP_inclusion}, the roots of $S$ are precisely the zeros of $A+B + C$. Then it is straightforward to check that $x^\ast \in \zer(A + B + C)$, implies that $(x_1^\ast, \ldots, x_m^\ast) \in \zer(\partial g + \nabla f  + N_{\{x \in \cH \mid \sum_{j=1}^M A_j x_j = b\}})$, and if $(x_1^\ast, \ldots, x_m^\ast) \in \zer(\partial g + \nabla f  + N_{\{x \in \cH \mid \sum_{j=1}^M A_j x_j = b\}})$, then there exists $x_{M+1}^\ast \in \cH_{M+1}$ such that $x^\ast = (x_1^\ast, \ldots, x_{M+1}^\ast) \in \zer(A + B+ C)$. Both inclusions imply that $(x_1^\ast, \ldots, x_M^\ast)$ minimizes~\eqref{eq:RSCMP}.

Part~\ref{prop:RSCMP:part:demiclosed} (demiclosedness): The operator $T = I- S$ in (displayed in Part~\ref{prop:RSCMP:part:coherence}) is the composition of two nonexpansive maps (in the norm $\|\cdot\|_P$), and thus, it is nonexpansive. Therefore, $S$ is demiclosed at $0$.
\qed\end{proof}

~\\~\\
\begin{proposition}[ProxSMART Operator Properties]\label{prop:ProxSMART}
Assume the setting of Section~\ref{sec:ProxSMART}, and in particular, that for all $x \in \cH$, we have
\begin{align*}
\left(S(x) \right)_j= 
\begin{cases}
x_1 - \prox_{\gamma_1 g_1}\left(x_1 - \gamma_1\sum_{j=2}^M A_j^\ast x_j \right)  &\text{if $j = 1$;} \\
x_j - \prox_{\gamma_j g_j^\ast}\left(x_j + \gamma_j A_j\left(2\overline{x}_1 - x_1\right)\right) &\text{otherwise;}
\end{cases}
\end{align*}
where $\overline{x}_1 = \prox_{\gamma_1 g_1}\left(x_1 - \gamma_1\sum_{j=2}^M A_j^\ast x_j \right)$.  Then 
\begin{enumerate}
\item \textbf{Coherence:} \label{prop:RNPD:part:coherence} there is a strongly positive self-adjoint linear operator $P : \cH \rightarrow \cH$ such that if $\delta \in (0, 1)$;
\begin{align*}
\gamma_{1}\left( \sum_{j=2}^M \gamma_j\|A_j\|^2\right)\leq \delta.
\end{align*}
Then the linear map $P$ satisfies $\sum_{j=1}^{M} \underline{M}_j\|x_j\|^2_j \leq \|x\|_P^2 \leq   \sum_{j=1}^{M} \overline{M}_j\|x_j\|^2_j$,
where for all $j $, we have 
\begin{align*}
\underline{M}_j := \frac{1-\sqrt{\delta}}{\gamma_j} && \text{and} && \overline{M}_j := \frac{1+\sqrt{\delta}}{\gamma_j}.
\end{align*}
With these parameters, $S$ also satisfies the coherence condition~\eqref{eq:coherence}:
 \begin{align*}
\left( \forall x \in \cH\right), \left( \forall x^\ast \in \zer(S)\right)  \qquad \dotp{ S(x), x - x^\ast}_P \geq \sum_{j=1}^{M} \beta_{1j}\|(S(x))_j\|^2_j \qquad \text{with } \qquad \beta_{1j} :=  \frac{1-\sqrt{\delta}}{\gamma_j}.
\end{align*}
\item \textbf{Roots:} \label{prop:RNPD:part:roots} with any choice of $\gamma_j$, we have
\begin{align*}
x^\ast\in \zer(S) \implies   x_1 \text{ solves \eqref{eq:RNPD}}, 
\end{align*}
and $\zer(S) \neq 0$ if, and only if, $\zer(\partial g_1(x) + \sum_{j=2}^M A_j^\ast\partial g_j\circ A_j) \neq \emptyset$.
\item \textbf{Demiclosedness:} \label{prop:RNPD:part:demiclosed} $S$ is demiclosed at $0$. 
\end{enumerate}
\end{proposition}
\begin{proof}
Part~\ref{prop:RNPD:part:coherence} (coherence): Here is the linear map $P$ in block matrix form
\begin{align}\label{eq:P_RCPD_RNPD}
P &= \begin{bmatrix}
\frac{1}{\gamma_1} I_{\cH_1}  & -A_2^\ast & \cdots & \cdots  &  -A_{M}^\ast \\
-A_2 & \frac{1}{\gamma_2} I_{\cH_2} &0 & \cdots & 0  \\
\vdots& & \ddots & & \vdots  \\
-A_M & 0  & \cdots & 0 & \frac{1}{\gamma_M} I_{\cH_M}\end{bmatrix}. 
\end{align}
The lower bound is below
\begin{align*}
\dotp{x, x}_P  &= \sum_{j=1}^{M} \frac{1}{\gamma_j}\|x_j\|^2_j - \sum_{j=1}^M 2\dotp{A_jx_j, x_{1}}_{1} \\
&\geq\sum_{j=1}^{M} \frac{1}{\gamma_j}\|x_j\|^2_j - \sum_{j=1}^M 2\|A_j\|\|x_j\|_{j}\| x_{1}\|_{1} \\
&\geq \sum_{j=1}^{M} \frac{1}{\gamma_j}\|x_j\|^2_j - \sum_{j=1}^M \left(\frac{\sqrt{\delta}\|x_j\|_j^2}{\gamma_j} + \frac{\|A_j\|^2\|x_{1}\|_1^2}{\sqrt{\delta}}\right) \\
&\geq \sum_{j=2}^{M} \frac{1-\sqrt{\delta}}{\gamma_j}\|x_j\|^2_j + \frac{1}{\gamma_{1}} \left(1- \gamma_{1}\sum_{j=1}^M \frac{\|A_j\|^2}{\sqrt{\delta}} \right)  \|x_{1}\|_{1}^2\\
&\geq \sum_{j=1}^{M} \frac{1-\sqrt{\delta}}{\gamma_j}\|x_j\|^2_j.
\end{align*}
The upper bound follows the exact same argument but all ``$-$" signs (except for the one on the first line) are changed to ``$+$" signs.

To get the $\beta_{ij}$, we find an $\alpha$-averaged operator $T$ in the norm $\|\cdot\|_P$ such that $S = (I - T)$; then we apply Proposition~\ref{eq:identminusaveraged} together with the lower on $\|\cdot\|_P^2$ that we just derived. 

Let $A: \cH \rightarrow 2^\cH$ be the monotone operator
$$
A(x_1, \ldots, x_{M}) := \begin{bmatrix}
\partial g_1(x_1) \\
\partial g_2^\ast(x_2) \\
\vdots \\
\partial g_M^\ast(x_M) 
\end{bmatrix};
$$
and let $B : \cH \rightarrow \cH$ be the skew symmetric linear map:
\begin{align*}
B := \begin{bmatrix}
0 &  \cdots & 0 & A_1 \\   
\\
0 & \cdots & 0 & A_{M} \\   
-A_1^\ast  & \cdots & -A_{M}^\ast & 0   
\end{bmatrix}.
\end{align*}
Then define $T$ to be the resolvent operator 
$$
T:= J_{P^{-1}(A+B)} 
$$
Given $x$, we compute $Tx$:
\begin{align*}
x^+ = Tx & \Leftrightarrow P(x - x^+) \in (A + B)x^+ \\
&\Leftrightarrow 
\begin{bmatrix} 
\frac{1}{\gamma_1}(x_1 - x_1^+)  - \sum_{j=2}^mA_j^\ast(x_{j} - x_{j}^+) \\
\frac{1}{\gamma_2}(x_2- x_2^+)  - A_2(x_{1} - x_{1}^+) \\
\vdots\\
\frac{1}{\gamma_M}(x_M- x_M^+)  - A_2(x_{1} - x_{1}^+) \\
\end{bmatrix}
\in 
\begin{bmatrix}
  \sum_{j=2}^mA_j^\ast x_{j}^+ \\
\partial g_2^\ast(x_2^+) - A_2^\ast x_{1}^+  \\
\vdots\\
\partial g_M^\ast(x_M^+) - A_2^\ast x_{1}^+  
\end{bmatrix} \\
&\Leftrightarrow \hphantom{\left(\forall j > 1\right)} \qquad  x_{1}^{+} = \prox_{\gamma_1 g_1}\left(x_1 - \gamma_1\sum_{j=2}^M A_j^\ast x_j \right); \\
& \hphantom{\Leftrightarrow} \left(\forall j > 1\right) \qquad  x_j^{+}  = \prox_{\gamma_j g_j^\ast} \left( x_j + \gamma_jA_j\left(2x_1^{+} - x_1\right)\right). \numberthis\label{eq:RNPD_inclusion}\\
\end{align*} 
Thus, $S = I-T$. 

The operator $T$ is a resolvent, and such operators are always $2^{-1}$-averaged.  Thus, from Proposition~\ref{eq:identminusaveraged}, the operator $S = I - T$ is $1$-cocoercive in the norm $\|\cdot\|_P$; with this property, we obtain $\beta_{ij}$
\begin{align*}
\left(\forall x \in \cH\right) \qquad \dotp{S(x), x - x^\ast}_P &\geq \|S(x)\|^2_P \geq  \sum_{j=1}^{M} \frac{1-\sqrt{\delta}}{\gamma_j}\|(S(x))_j\|_j^2. \qquad \qed
\end{align*}

Part~\ref{prop:RNPD:part:roots} (roots): From the first line of~\eqref{eq:RNPD_inclusion}, the roots of $S$ are precisely the zeros of $A+B$. Then it is straightforward to check that $x^\ast \in \zer(A + B)$, implies that $x_1^\ast \in \zer(\partial g_1(x) + \sum_{j=2}^M A_j^\ast\partial g_j\circ A_j)$, and if $x_1^\ast \in\zer(\partial g_1(x) + \sum_{j=2}^M A_j^\ast\partial g_j\circ A_j)$, then there exists $(x_2^\ast, \ldots, x_{M}^\ast) \in \cH_{2} \times \cdots \times \cH_M$ such that $x^\ast = (x_1^\ast, \ldots, x_{M}^\ast) \in \zer(A + B)$. Both inclusions imply that $x_1^\ast$ minimizes~\eqref{eq:RNPD}.

Part~\ref{prop:RSCMP:part:demiclosed} (demiclosedness): The operator $T = I- S$ in (displayed in Part~\ref{prop:RSCMP:part:coherence}) is a nonexpansive map (in the norm $\|\cdot\|_P$). Therefore, $S$ is demiclosed at $0$.
\qed\end{proof}

~\\~\\
\begin{proposition}[ProxSMART+ Operator Properties]\label{prop:ProxSMART+}
Assume the setting of Section~\ref{sec:ProxSMART+}, and in particular, that for all $x \in \cH$, we have
\begin{align*}
\left(\forall i < N+1\right)\qquad \left(S_i(x)\right)_j 
&= \begin{cases}
\frac{\gamma_1}{N}\nabla f_i\left(x_1 - 2\gamma_1\sum_{j=2}^M A_j^\ast x_j\right) &\text{if $j = 1$;}\\
0 &\text{otherwise.}
\end{cases}\\
\left(S_{N+1}(x) \right)_j
&= 
\begin{cases}
 \gamma_1\sum_{j=2}^M A_j^\ast x_j  &\text{if $j = 1$;} \\
x_j - \prox_{\gamma_j g_j^\ast}\left(x_j + \gamma_j A_j\left(x_1 -2\gamma_1\sum_{j=2}^M A_j^\ast x_j \right)\right) &\text{otherwise,}
\end{cases}
\end{align*}
Then 
\begin{enumerate}
\item \textbf{Coherence:} \label{prop:RCPD:part:coherence} there is a strongly positive self-adjoint linear operator $P : \cH \rightarrow \cH$ such that if $\delta \in (0, 1)$;
\begin{align*}
\gamma_{1}\left( \sum_{j=2}^M \gamma_j\|A_j\|^2 + \frac{1}{2N} \sum_{i=1}^N L_i\right)\leq \delta.
\end{align*}
Then the linear map $P$ satisfies $\sum_{j=1}^{M} \underline{M}_j\|x_j\|^2_j \leq \|x\|_P^2 \leq   \sum_{j=1}^{M} \overline{M}_j\|x_j\|^2_j$,
where for all $j > 0$, we have 
\begin{align*}
\underline{M}_j := \left(1- \sqrt{\delta} + \frac{1}{2\sqrt{\delta}N} \sum_{i=1}^N L_i\right) && \text{and} && \overline{M}_j := \left(1+ \sqrt{\delta} - \frac{1}{2\sqrt{\delta}N} \sum_{i=1}^N L_i\right) .
\end{align*}
With these parameters, $S$ also satisfies the coherence condition~\eqref{eq:coherence}:
 \begin{align*}
\left( \forall x \in \cH\right), \left( \forall x^\ast \in \zer(S)\right)  \qquad \dotp{ S(x), x - x^\ast}_P \geq \sum_{i=1}^{N+1}\sum_{j=1}^{M} \beta_{ij}\|(S_i(x))_j - (S_i(x^\ast))_j\|^2_j 
\end{align*}
with 
\begin{align*}
\left(\forall 1\leq  i < N+1\right) \qquad \beta_{i1}  := \frac{N(1-\sqrt{\delta})}{2(N+1)\gamma_1^2 L_i};  &&\left(\forall i < N+1\right),\left(\forall j > 1\right) \qquad \beta_{ij} \equiv 0;&& \left(\forall j\right) \qquad\beta_{(N+1)j} := \frac{1-\sqrt{\delta}}{\gamma_j}.
\end{align*}
\item \textbf{Roots:}\label{prop:RCPD:part:roots} with any choice of $\gamma_j$, we have
\begin{align*}
x^\ast\in \zer(S) \implies  x_1^\ast - 2\gamma_1\sum_{j=2}^M A_j^\ast x_j^\ast  \text{ solves \eqref{eq:RCPD}},
\end{align*}
and $\zer(S) \neq \emptyset$ if, and only if, $\zer(\sum_{j=2}^M A_j^\ast\partial g_j\circ A_j + N^{-1}\sum_{i=1}^N \nabla f_i) \neq \emptyset$.
\item \textbf{Demiclosedness:} \label{prop:RCPD:part:demiclosed} $S$ is demiclosed at $0$. 
\end{enumerate}
\end{proposition}
\begin{proof}
Part~\ref{prop:RCPD:part:coherence} (coherence): The linear map $P$ is given in~\eqref{eq:P_RCPD_RNPD}. The lower bound is below
\begin{align*}
\dotp{x, x}_P  &= \sum_{j=1}^{M} \frac{1}{\gamma_j}\|x_j\|^2_j - \sum_{j=1}^M 2\dotp{A_jx_j, x_{1}}_{1} \\
&\geq\sum_{j=1}^{M} \frac{1}{\gamma_j}\|x_j\|^2_j - \sum_{j=1}^M 2\|A_j\|\|x_j\|_{j}\| x_{1}\|_{1} \\
&\geq \sum_{j=1}^{M} \frac{1}{\gamma_j}\|x_j\|^2_j - \sum_{j=1}^M \left(\frac{\sqrt{\delta}\|x_j\|_j^2}{\gamma_j} + \frac{\|A_j\|^2\|x_{1}\|_1^2}{\sqrt{\delta}}\right) \\
&\geq \sum_{j=2}^{M} \frac{1-\sqrt{\delta}}{\gamma_j}\|x_j\|^2_j + \frac{1}{\gamma_{1}} \left(1- \gamma_{1}\sum_{j=1}^M \frac{\|A_j\|^2}{\sqrt{\delta}} \right)  \|x_{1}\|_{1}^2\\
&\geq \sum_{j=2}^{M} \frac{1-\sqrt{\delta}}{\gamma_j}\|x_j\|^2_j + \frac{1}{\gamma_{1}} \left(1- \sqrt{\delta} + \frac{1}{2\sqrt{\delta}N} \sum_{i=1}^N L_i\right)  \|x_{1}\|_{1}^2.
\end{align*}
The upper bound follows the exact same argument but all ``$-$" signs are changed to ``$+$" signs; the exception to this rule is the last line, in which the $-\frac{1}{2\sqrt{\delta}N} \sum_{i=1}^N L_i$ must be changed to $\frac{1}{2\sqrt{\delta}N} \sum_{i=1}^N L_i $.

Let $A: \cH \rightarrow 2^\cH$ be the monotone operator
$$
\left(\forall x \in \cH\right) \qquad A(x_1, \ldots, x_{M}) := \begin{bmatrix}
0 \\
\partial g_2^\ast(x_2) \\
\vdots \\
\partial g_M^\ast(x_M) 
\end{bmatrix};
$$
let $B : \cH \rightarrow \cH$ be the skew symmetric linear map:
\begin{align*}
B := \begin{bmatrix}
0 &  \cdots & 0 & A_1 \\   
\\
0 & \cdots & 0 & A_{M} \\   
-A_1^\ast  & \cdots & -A_{M}^\ast & 0   
\end{bmatrix}.
\end{align*}
Then from the proof of Theorem~\ref{prop:ProxSMART} (with $g_1 \equiv 0$), we find that $S_{N+1} = I - J_{P^{-1}(A+B)}$ is $1$-cocoercive in $\|\cdot\|_P$.  Thus, 
\begin{align*}
\dotp{S_{N+1}(x) - S_{N+1}(x^\ast), x - x^\ast}_P & \geq \|S_{N+1}(x) - S_{N+1}(x^\ast)\|^2_P \\
&\geq \sum_{j=2}^M \frac{1-\sqrt{\delta}}{\gamma_j} \|(S_{N+1}(x))_j -(S_{N+1}(x^\ast))_j\|^2_j \\
&\hspace{20pt}+ \left(1- \sqrt{\delta} + \frac{1}{2\sqrt{\delta}N} \sum_{i=1}^N L_i\right) \|(S_{N+1}(x))_1 -(S_{N+1}(x^\ast))_1\|^2_1\numberthis\label{eq:RCPD_coh_1}
\end{align*}

Now we take care of the $S_i$ operators for $i < N+1$: set $\hat{x}_1 = x_1 - 2\gamma_1\sum_{j=2}^M A_j^\ast x_j$ and $\hat{x}_1^\ast = x_1^\ast -2\gamma_1\sum_{j=2}^M A_j^\ast x_j^\ast$.
\begin{align*}
\dotp{ S_i(x) - S_i(x^\ast) , x - x^\ast}_P &= \dotp{\frac{1}{N} (\nabla f_i(\hat{x}_1) - \nabla f_i(\hat{x}_1^\ast )), x_1 - x_1^\ast}_1\\
&\hspace{20pt} - \gamma_1\dotp{\frac{1}{N}((\nabla f_i(\hat{x}_1) - \nabla f_i(\hat{x}_1^\ast )), \sum_{j=2}^M A_j^\ast(x_j - x_j^\ast)}_1 \\
&\geq \dotp{\frac{1}{N} (\nabla f_i(\hat{x}_1) - \nabla f_i(\hat{x}_1^\ast )), \hat{x}_1- \hat{x}_1^\ast }_1\\ 
&\hspace{20pt} + \gamma_1\dotp{\frac{1}{N}((\nabla f_i(\hat{x}_1) - \nabla f_i(\hat{x}_1^\ast )), \sum_{j=2}^M A_j^\ast(x_j - x_j^\ast)}_1 \\
&\geq \frac{N}{\gamma_1^2L_i}\left\|\frac{\gamma_1}{N} ((\nabla f_i(\hat{x}_1) - \nabla f_i(\hat{x}_1^\ast ))\right\|^2_1 \\
&\hspace{20pt} - \frac{N\sqrt{\delta}}{L_i\gamma_1^2}\left\|\frac{\gamma_1}{N} ((\nabla f_i(\hat{x}_1) - \nabla f_i(\hat{x}_1^\ast ))\right\|^2_1 - \frac{L_i}{4N\sqrt{\delta}}\left\|\gamma_1 \sum_{j=2}^M A_j^\ast(x_j - x_j^\ast)\right\|^2_1 \\
&= \frac{N(1-\sqrt{\delta})}{\gamma_1^2L_i}\left\|(S_i(x))_{1} - (S_i(x^\ast))_{1}\right\|^2_1 - \frac{L_i}{4N\sqrt{\delta}}\left\|(S_{N+1}(x))_1 -(S_{N+1}(x^\ast))_1\right\|^2_1\numberthis\label{eq:RCPD_coh_2} 
\end{align*}
Therefore, to show prove the coherence condition, we average~\eqref{eq:RCPD_coh_1} and \eqref{eq:RCPD_coh_2}:
\begin{align*}
\dotp{S(x), x - x^\ast}_P &= \frac{1}{N+1} \sum_{i=1}^{N+1} \dotp{S_i(x) - S_i(x^\ast), x - x^\ast}_P \\
&\geq \sum_{i=1}^N\frac{N(1-\sqrt{\delta})}{\gamma_1^2L_i}\left\|(S_i(x))_{1} - (S_i(x^\ast))_{1}\right\|^2_1
 + \sum_{j=1}^M\frac{1-\sqrt{\delta}}{\gamma_1}\left\|(S_{N+1}(x))_j -(S_{N+1}(x^\ast))_j\right\|^2_j. 
 \end{align*} 
 
Part~\ref{prop:RCPD:part:roots} (roots): First suppose that $x^\ast \in \zer(S)$.  Then $N^{-1}\sum_{i=1}^N\nabla f_i( x_1- 2\gamma_1\sum_{j=2}^M A_j^\ast x_j) + \sum_{j=2}^M A_j^\ast x_j^\ast = 0$, and moreover, the points $x_j^\ast$ for $j > 2$ satisfy:
\begin{align*}
x_j^\ast = \prox_{\gamma_j g_j^\ast}\left(x_j^\ast + \gamma_j A_j\left(x_1^\ast - 2\gamma_1\sum_{j=2}^MA_j^\ast x_j^\ast\right)\right) &\Leftrightarrow  A_j\left(x_1^\ast - 2\gamma_1\sum_{j=2}^MA_j^\ast x_j^\ast\right)\in \partial g_j^\ast(x_j^\ast) \\
&\Leftrightarrow x_j^\ast \in \partial g_j\left( A_j\left(x_1^\ast - 2\gamma_1\sum_{j=2}^MA_j^\ast x_j^\ast\right)\right)
\end{align*}
Therefore, 
$$
0 = \frac{1}{N}\sum_{i=1}^N\nabla f_i\left(x_1^\ast - 2\gamma_1\sum_{j=2}^MA_j^\ast x_j^\ast\right) + \sum_{j=2}^M A_j^\ast x_j^\ast \in   \frac{1}{N}\sum_{i=1}^N\nabla f_i\left(x_1^\ast - 2\gamma_1\sum_{j=2}^MA_j^\ast x_j^\ast\right) + \sum_{j=2}^M A_j^\ast \partial g_j\left( A_j\left(x_1^\ast - 2\gamma_1\sum_{j=2}^MA_j^\ast x_j^\ast\right)\right);
$$
in particular, $x_1^\ast - 2\gamma_1\sum_{j=2}^MA_j^\ast x_j^\ast$ minimizes~\eqref{eq:RCPD}. 

To show that $\zer(\sum_{j=1}^M A_j^\ast\partial g_j\circ A_j + N^{-1}\sum_{i=1}^N \nabla f_i) \neq \emptyset$ implies that $\zer(S) \neq \emptyset$, reverse the above argument: choose $\hat{x}_1^\ast \in \zer(\sum_{j=1}^M A_j^\ast\partial g_j\circ A_j + N^{-1}\sum_{i=1}^N \nabla f_i)$ and write $0 = N^{-1}\sum_{i=1}^N \nabla f_i + \sum_{j=1}^M A_j^\ast x_j^\ast$, where $x_j^\ast \in \partial g_j( A_j \hat{x}_1^\ast)$. Then $A_j x_1^\ast \in \partial g_1^\ast(x_j^\ast)$, and consequently, we have $x_j^\ast = \prox_{\gamma_j g_j^\ast}(x_j^\ast + A_j\hat{x}_1^\ast)$. Thus, if $ x_1^\ast = \hat{x}_1^\ast + 2\gamma_1 \sum_{j=2}^M A_j x_j^\ast$, then we have $x^\ast = (x_1^\ast, \ldots, x_M^\ast) \in \zer(S)$.

Part~\ref{prop:RNPD:part:demiclosed} (demiclosedness):  In Part~\ref{prop:RNPD:part:coherence}, our assumption that $x^\ast \in \zer(S)$ was not necessary, and if we instead assume that $x^\ast = y$, where $y$ is an arbitrary point in $\cH$, the exact same argument yields that
\begin{align*}
\dotp{S(x) - S(y), x - y}_P &\geq \sum_{i=1}^{N+1}\sum_{j=1}^{M} \beta_{ij}\|(S_i(x))_j - (S_i(y))_j\|^2_j \geq C \sum_{i=1}^{N+1}\|S_i(x) - S_i(y)\|^2_P \geq C(N+1) \|S(x) - S(y)\|^2_P,
\end{align*}
where $C > 0$ is a constant; in other words, the operator $S$ is $C(N+1)$-cocoercive. Thus, by~\cite[Proposition 4.33]{bauschke2011convex}, 
$$
T = I - \frac{1}{C(N+1)}S = \left(1 - \frac{1}{C(N+1)}\right)I + \frac{1}{C(N+1)}(I - S)
$$
is nonexpansive. Therefore, because $I-T$ is demiclosed at $0$, it is easy to see that $S$ is demiclosed at $0$.
\qed\end{proof}

\begin{proposition}[SMART for Monotone Inclusions]\label{prop:MONOopprop}
Assume the setting of Section~\ref{sec:general_monotone_inclusion}, and in particular, that $S$ is as in~\eqref{eq:SMART_mono_operators}
Suppose that the operator $B := N^{-1}\sum_{i=1}^N B_i$ is monotone and $A$ is $\mu_A$-strongly monotone. Then 
\begin{enumerate}
\item \textbf{Coherence:}\label{prop:MONOopprop:part:coherence} $S$ satisfies the coherence condition
with constants defined as in~\eqref{eq:mono_op_conditions};
\item \textbf{Essential strong quasi-monotonicity:}\label{prop:MONOopprop:part:ess_strong_mono} $S$ is $\mu$-strongly monotone with $\mu$ as in~\eqref{eq:mu_saddle};
\item \textbf{Roots:}\label{prop:MONOopprop:part:roots} with any choice of $\gamma$, we have
$x\in \zer(S) \iff J_{\gamma A} (x) \text{ solves \eqref{eq:general_monotone_inclusion}}$, 
\item \textbf{Demiclosedness:} \label{prop:MONOopprop:part:demiclosed} $S$ is demiclosed at $0$.
\end{enumerate}
\end{proposition}
\begin{proof}
Parts~\ref{prop:MONOopprop:part:coherence} (coherence) and~\ref{prop:MONOopprop:part:ess_strong_mono} (essential strong quasi-monotonicity): Let $ T := (I - \gamma B) \circ J_{\gamma A}$, which satisfies $S = (N+1)^{-1}(I-T)$. Noting that $J_{\gamma A}$ is $(1+\gamma\mu_A)^{-1}$ Lipschitz~\cite[Proposition 23.11]{bauschke2011convex} and that $B$ is monotone and $\overline{L}$-Lipschitz, we find that for all $x, y \in \cH$
\begin{align*}
\|Tx - Ty\|^2 &= \|J_{\gamma A}x - J_{\gamma A}y\|^2 - 2\dotp{J_{\gamma A}x - J_{\gamma A}y, \gamma BJ_{\gamma A}x - \gamma B J_{\gamma A}y} + \gamma^2 \|BJ_{\gamma A}x - B J_{\gamma A}y\|^2 \\
&\leq \frac{1 + \gamma^2 \overline{L}}{(1 + \gamma \mu_A)^2} \|x - y\|^2.
\end{align*}
Thus, $T$ is $\kappa$-Lipschitz for some $\kappa \in (0, 1)$. Thus, from the simple bound $(N+1)^{-1}\dotp{(I - T)x - (I-T)y, x - y} \geq (N+1)^{-1}(1-\kappa)\|x-y\|^2$ we conclude that $S$ is $\mu := (N+1)^{-1}(1-\kappa)$-strongly monotone.

To get $\beta_{i1}$, we use the strong monotonicity of $S$ combined with the Lipschitz continuity of $S_i$, which is $(L_i\gamma N^{-1})$-Lipschitz for $i <N+1$ and is $1$-Lipschitz for $i = N+1$:
\begin{align*}
\dotp{S(x) - S(y), x - y} &\geq \mu\|x- y\|^2\\
 &\geq \sum_{i = 1}^n \frac{\mu N^2}{L_i^2 \gamma^2(N+1)}\|S_i(x) - S_i(y)\|^2 + \frac{\mu}{N+1}\|S_{N+1}(x)  - S_{N+1}(y)\|^2.
\end{align*}

Part~\ref{prop:MONOopprop:part:roots} (roots): 
$x \in \zer(S) \Leftrightarrow 0 = (I- J_{\gamma A}) + \gamma N^{-1} \sum_{i=1}^N  B_i(J_{\gamma A}(x)) \in \gamma A J_{\gamma A}(x) + \gamma N^{-1} \sum_{i=1}^N  B_i(J_{\gamma A}(x)) \Leftrightarrow J_{\gamma A}(x) \in \zer \left(A + B\right).$ 

Part~\ref{prop:MONOopprop:part:demiclosed} (demiclosedness): The operator $T$ in (displayed in Part~\ref{prop:MONOopprop:part:ess_strong_mono}) is a contraction and $\Fix(T) = \zer(S)$. Therefore, by~\cite[Corollary 4.18]{bauschke2011convex}, $S$ is demiclosed at $0$.
\qed\end{proof}

\newpage
\begin{table}[htbp]\caption{Symbols}
\centering\vspace{-10pt} 
\begin{tabular}{r c p{10cm} }
\toprule
\multicolumn{3}{c}{\textbf{Indices}}\\
$n$ & $\in$& $\NN$;  number of functions/operators\\
$m$ & $\in$ & $\NN$; number of coordinates\\
$i$ & $\in$ & $\{1, \ldots, n\}$; index value for operators \\
$j$ & $\in$ & $\{1, \ldots, m \}$; index value for coordinates\\
$k$ & $\in$  & $\NN$; index value for  sequences\\
\multicolumn{3}{c}{}\\
\multicolumn{3}{c}{\textbf{Spaces}}\\
$\cH_j$ & & A separable Hilbert space \\
$\cH$ & := & $\cH_1 \times \cdots \times \cH_m$ \\
\multicolumn{3}{c}{}\\
\multicolumn{3}{c}{\textbf{Variables}}\\
$x, x_j$ & & Primal variables: $x \in \cH$ and $x_j$ is the $j$th coordinate of $x$  \\
$y_i, y_{i,j}$ & & Dual variables: $y_i \in \cH$ and $y_{i,j}$ is the $j$th coordinate of $y_i$ \\
\multicolumn{3}{c}{}\\
\multicolumn{3}{c}{\textbf{Norms and inner products}}\\
$\|\cdot\|_j, \; \dotp{\cdot, \cdot}_j$ & & A norm and an inner product on $\cH_j$ \\
$\|\cdot\|, \; \dotp{\cdot, \cdot}$ & & A norm and an inner product on $\cH$\\
$\|\cdot\|_\pr$ & & $(\forall x \in \cH)\;  \|x\|_\pr^2 := \sum_{j=1}^m \|x_j\|_j^2$.\\
$\underline{M}_j, \overline{M}_j$ & & $\left( \forall x\in \cH\right) \; \sum_{i=1}^m\lmu_j \|x_j\|_j^2 \leq \|x\|^2 \leq \sum_{i=1}^m \umu_j\|x_j\|_j^2  $\\
\multicolumn{3}{c}{}\\
\multicolumn{3}{c}{\textbf{Operators/Functions}}\\
$f_i$ &  &  A convex function with a Lipschitz continuous gradient\\
$g_{j}$ & & A nonsmooth/proximable function \\
$S_i$ & & A map from $\cH$ to $\cH$ \\
$S$ &$:=$ & $n^{-1} \sum_{i=1}^n S_i$ \\
$\cS$ & $:= $ & $\zer(S)$\\
$\mathbf{S}^\ast$ & $:=$ & $((S_i(x^\ast))_j)_{ij}, \qquad \left(x^\ast \in \cS\right)$ \\
\multicolumn{3}{c}{}\\
\multicolumn{3}{c}{\textbf{Operators properties}}\\
$\beta_{ij}$ & & Coherence constants: $$(\forall x\in \cH),(\forall x^\ast \in \cS) \; \sum_{j=1}^m \sum_{i=1}^n \beta_{ij} \|(S_{i}(x))_{j} - (S_{i}(x^\ast))_j\|_j^2  \leq n^{-1}\sum_{i=1}^n\dotp{S_i(x), x - x^\ast}$$\\
$\mu$ & & Essential strong quasi-monotonicity constant: $$(\forall x\in \cH) \; \mu\|x- P_\cS(x)\|^2 \leq  \dotp{S(x), x - P_\cS(x)}$$ \\
\multicolumn{3}{c}{\textbf{Random Variables}}\\
$i_k$ & $\in$ & $\{1, \ldots, n\}$; an IID sequence of random operator indices\\
$\sfS_k$ & $\subseteq$ & $\{1, \ldots, m\}$;  an IID sequence of random sets of coordinates \\ 
$\epsilon_k$ & $\in$ & $\{0, 1\}$; an IID sequence of  random binary decisions, which indicate whether to update the current dual variable ($\epsilon_k = 1$) or to leave it fixed ($\epsilon_k = 0$) \\
\multicolumn{3}{c}{}\\
\multicolumn{3}{c}{\textbf{Graphs}}\\
$G = (V, E)$ & & Trigger graph: vertices $V = \{1, \ldots, n\}$ and edges $E \subseteq \{1, \ldots, n\}^2$; \begin{quote}
\centering $(i,i') \in E$ if, and only if, $i_k = i$ triggers update of dual variable $y_{i'}^k$.
\end{quote}\\
\multicolumn{3}{c}{\textbf{Probabilities}}\\
$q_j$ & $:=$ & $P(j \in \sfS_k)$ \\
$p_{ij}$ & $:=$ & $P( i_k = i \mid j \in \sfS_k)$ \\
$\rho$ & $:= $ & $P(\epsilon_k = 1)$ \\
$p_{ij}^T$ & $:=$ & The probability that the $i$th dual variable is updated and $j \in \sfS^k$: 
\begin{quote}\centering
$P((i_k, i) \in E, j \in \sfS_k) = P((i_k, i) \in E \mid j \in \sfS_k)q_j = \sum_{(i', i) \in E} p_{i'j}q_j $\end{quote} \\
\multicolumn{3}{c}{\textbf{Delays}}\\
$\tau_p$ & $\in $& $\NN$; maximum primal variable delay \\
$\tau_d$ & $\in$ & $\NN$; maximum dual variable delay  \\
$d_k$ & $\in$ &$\{0, 1, \ldots, \tau_p\}^{m}$; the vector of primal variable delays at the $k$th iteration
\\
$\delta$&  $:=$ & $\sup \{ |d_{k,j} - d_{k,j'}|\mid k \in \NN \text{ and }  j, j' \in \{1, \ldots, m\} \}$\\
$e^{i}_{k}$ & $\in$ & $\{0, 1, \ldots, \tau_d\}^m$; the vector of dual variable delays for the $i$th dual variable at the $k$th iteration\\
\bottomrule
\end{tabular}
\label{tab:symbols}
\end{table}
\end{document}